\documentclass{amsart}
\usepackage{amsmath}
\usepackage{amsthm}
\usepackage{amsfonts}
\usepackage{amssymb}
\usepackage[all]{xy}
\usepackage{amscd}

\usepackage{enumerate}

\usepackage{hyperref}
\hypersetup{
    colorlinks=true,      
    linkcolor=blue,          
    citecolor=blue,      
    filecolor=blue,      
    urlcolor=blue           
}

    \newtheorem{Lem}{Lemma}[section]
    \newtheorem{Lem-Def}{Lemma-Definition}[section]
    \newtheorem{Prop}[Lem]{Proposition}

    \newtheorem{Thm}[Lem]{Theorem}

\theoremstyle{definition}

    \newtheorem{Rem}[Lem]{Remark}

\newcommand{\Spec}{\text{Spec}\,}

\newcommand{\W}{\mathcal W}
\newcommand{\A}{\mathcal A}
\newcommand{\B}{\mathcal B}
\newcommand{\E}{\mathcal E}

\newcommand{\I}{\mathcal I}
\newcommand{\M}{\mathcal M}
\renewcommand{\P}{\mathcal P}
\newcommand{\N}{\mathcal N}
\newcommand{\T}{\mathcal T}

\newcommand{\R}{\mathcal R}
\renewcommand{\S}{\mathcal S}
\renewcommand{\L}{\mathcal L}
\renewcommand{\O}{\mathcal O}
\newcommand{\C}{\mathcal C}
\newcommand{\X}{\mathcal X}
\newcommand{\Y}{\mathcal Y}
\newcommand{\Pa}{\mathcal P}

\newcommand{\col}{\colon}

\newcommand{\Ps}{\mathbb{P}}
\newcommand{\ra}{\rightarrow}
\newcommand{\ol}{\overline}

\newcommand{\ze}{\mathbb{Z}}

\newcommand{\wt}{\widetilde}
\newcommand{\ox}{\otimes}
\newcommand{\wed}{\wedge}
\newcommand{\wh}{\widehat}
\newcommand{\dra}{\dashrightarrow}

\newcommand{\J}{\ol{\mathcal J}}
\newcommand{\term}{\text{Term}}
\newcommand{\lra}{\longrightarrow}
\renewcommand{\:}{\colon}

\begin{document}

\title[The degree-2 Abel--Jacobi map for nodal curves -- II]{The degree-2 Abel--Jacobi map for  nodal curves -- II}

\author{Marco Pacini}

\begin{abstract}
\noindent
We construct a resolution of the degree-2 Abel--Jacobi map for a regular smoothing of a nodal curve. More precisely, let $\pi\col\C\ra B$ be a regular smoothing of a nodal curve $C$ with a section $\sigma$ through its smooth locus. Let $\J$ be the degree-0 Esteves' compactified Jacobian, parametrizing torsion-free rank-1 sheaves on $\C/B$ that are canonically $\sigma$-quasistable. Consider the \emph{degree-2 Abel--Jacobi map} $\alpha^2\col\C^2:=\C\times_B\C\dra\J$ of $\C/B$, sending a pair $(Q_1,Q_2)$ of points of the fiber $C_\eta$ of $\pi$ over the generic point $\eta$ of $B$ to the invertible sheaf $\O_{C_\eta}(2\sigma(\eta)-Q_1-Q_2)$. We show that if $\phi\col \wh\C^2\ra \C^2$ is the blowup of $\C^2$ first along its diagonal subscheme and then along products of 2-tails and 3-tails of $C$, then the map $\alpha^2\circ \phi\col \wh\C^2\dra \J$ is a morphism.
\end{abstract}
 
\maketitle

\section{Introduction}

\subsection{The problem}

Let $C$ be a smooth connected projective curve over an algebraically closed field $K$. The \emph{degree-$d$ Abel map of $C$} is a map $\alpha^d_\P\:C^d\to J_C^f$ from the product of $d$ copies of the curve to its degree-$f$ Jacobian $J_C^f$, sending a $d$-tuple $(Q_1,\dots,Q_d)$ of points of $C$ to $\P\ox\O_C(-Q_1-\dots-Q_d)$, where $\P$ is a line bundle of degree $d+f$ on $C$. If $\P=\O_C(dP)$, for a point $P$ of $C$, we call $\alpha^d_\P$ the \emph{degree-$d$ Abel--Jacobi map of $C$}. The Abel map encodes many geometric properties of the curve. For example, the Abel theorem states that the fibers of $\alpha^d_\P$ are projectivized complete linear series, up to the action of the symmetric group. Thus, all the possible embeddings of $C$ in projective spaces are known once we know its Abel maps.

What about Abel maps for singular curves that are limits of smooth curves? This question is natural and potentially very useful. Indeed, the Abel theorem suggests a possible interplay between Abel maps and limits of linear series on singular curve. This interplay has been recently explored in \cite{EO} for curves of compact type with two components. It was through the study of degenerations of linear series that the celebrated Brill--Noether and Gieseker--Petri theorems were proved in \cite{GH} and \cite{Gi}. 

The theory of limit linear system has been systemized for curves of compact type in the seminal work of Eisenbud and Harris \cite{EH}. Nevertheless, a satisfactory theory of limit linear series for nodal curves is not available, although there are works in this direction in \cite{CG} and \cite{EM}  for curves with two components and a recent new approach in \cite{O} for curves of compact type with two components. 

Not much is known for Abel maps for singular curves. In fact, Abel maps have been constructed only in few cases: For irreducible curves in \cite{AK}, in degree one in \cite{CE} and \cite{CCE}, 
 for curves of compact type and any degree in \cite{CP}, and recently for nodal curves with two components and any degree in \cite{ACP}.
 This paper is devoted to the construction of a degree-2 Abel--Jacobi map for any nodal curve. A part of our construction is strongly based of the previous work \cite{CEP}, where the existence of degree-2 Abel maps for nodal curves is reduced to a series of numerical conditions. More precisely, a crucial step in this paper is to check the validity of the \emph{admissibility conditions} given in \cite[Definition 6.4]{CEP}, over which our Theorem \ref{thmA} is based.  
 
 All the known Abel maps are natural, meaning that they are limits of Abel maps of smooth curves over one-dimensional bases. 
 The general problem of defining Abel maps for a smoothing $\pi\col\C\ra B$ of a nodal curve remains open. Here, a \emph{smoothing $\pi\col\C\ra B$ of a nodal curve $C$} is a family of curves over $B:=\Spec K[[t]]$ with special fiber isomorphic to $C$ and with $\C$ smooth. There are two main obstructions in the construction of a resolution of higher degree Abel maps. 
 
 First, since rational Abel maps for the smoothing $\pi$ are naturally defined over the product $\C^d:=\C\times_B\cdots\times_B\C$ of $d$ copies of $\C$ over $B$, one should be able to understand how the scheme $\C^d$ behaves under a sequence of blowups. This issue has been considered for $d=2$ in \cite{CEP} and for any $d$ in \cite{ACP}, where the relevant informations about the local and global geometry of the schemes $\C^d$ are unrevealed. 
 
 Second, since a resolution should be defined for a smoothing of any given nodal curve, complex combinatorial problems naturally arise. Since Abel maps take values in a compactified Jacobian defined by means of numerical conditions (see Section \ref{Jacobian}), a natural combinatorial issue is how to ``convert" an invertible sheaf in such a way that it satisfies these conditions. It is equivalent to give a modular description of the so-called \emph{Abel--N\'eron maps}. These maps are natural extensions of Abel maps to the $B$-smooth locus of $\C^d$ by means of N\'eron models and take values either in the Caporaso's or in Esteves' compactified Jacobian constructed in \cite{C} and in \cite{E01} (see \cite{CE} and Section \ref{Abel-maps} for more details). The modularity of the Abel--N\'eron maps is described for the degree-2 Abel--Jacobi map in \cite{P1} and it is strongly used in this paper. We believe that a generalization of this result is possible, at least for the degree-$d$ Abel--Jacobi map, although this should require significant new ideas. 
 
\subsection{The results} Let $C$ be a nodal curve defined over the algebraically closed field $K$. Let $\pi\col\C\ra B$ be a smoothing of $C$. Let $\sigma$ be a section of $\pi$ through its smooth locus. The degree-$2$ Abel--Jacobi map of $\C/B$ is defined as the rational map $\alpha^2\col \C^2=\C\times_B\C\dra \J$ sending a pair $(Q_1,Q_2)$ of points of the fiber $C_\eta$ over the generic point $\eta$ of $B$ of $\pi$ to the invertible sheaf $\O_{C_\eta}(2\sigma(\eta)-Q_1-Q_2)$. Here, $\J$ is the proper fine moduli scheme introduced by Esteves in \cite{E01}, parametrizing degree-$0$ torsion-free rank-1 sheaves on $\C/B$ that are $\sigma$-quasistable (with respect to a canonical polarization). We refer to Section \ref{Jacobian} for more details. 

In general, the map $\alpha^2$ is not defined everywhere on $\C^2$, as the case of a two-component two-node curve already shows (see \cite{Co}). Thus, to construct a geometrically meaningful resolution of $\alpha^2$, first we need to understand the geometry of the blowups of the scheme $\C^2$. We choose to blowup $\C^2$ along its diagonal subscheme and along (the strict transforms of) Weil divisors of type $Z_1\times Z_2$, where $Z_1$ and $Z_2$ are subcurves of $C$. In fact, $\C^2$ is singular, and a suitable chain of such blowups give rise to a \emph{good partial desingularization} $\phi\col \wt\C^2\ra \C^2$ of $\C^2$, meaning that all the strict transforms $\phi^{-1}(Z_1\times Z_2)$  become Cartier divisors (see Section \ref{Not2}). 

Fix a good partial desingularization $\phi\col \wt\C^2\ra \C^2$. 
We then consider the family of curves $p_1\col \wt\C^2\times_B\C\ra \wt\C^2$, where $p_1$ is the projection onto the first factor. Since $\J$ is a fine moduli scheme, to get a resolution of $\alpha^2$ we need a torsion-free rank-1 sheaf on $\wt\C^2\times_B\C/\wt\C^2$ that is $\sigma$-quasistable on the fibers of $p_1$ and whose induced map $\wt\C^2\ra\J$ generically agrees with $\alpha^2$.
Actually, we consider a blowup $\psi\col \wt\C^3\ra \wt\C^2\times_B\C$ and we  construct such a sheaf as $\psi_*\L_\psi$, where $\L_\psi$ is an invertible sheaf on $\wt\C^3$.
 We refer to Section \ref{Abel-maps} for the definition of $\L_\psi$. Here, $\psi$ is a \emph{good partial desingularization of $\wt\C^2\times_B\C$}, meaning that all the strict transforms of the divisors $\phi^{-1}(Z_1\times Z_2)\times Z_3$ of $\wt\C^2\times_B\C$ via $\psi$ become Cartier divisors, where  $Z_1,$ $Z_2$ and $Z_3$ are subcurves of $C$.

Finally, we consider \emph{distinguished points of $\wt\C^2$}, that is points $A$ that are contained in the `` most  degenerate locus" of $\wt\C^2$. This means that $A$ is contained in the intersection of three distinct divisors of $\wt\C^2$ 
\begin{equation}\label{intersection}
\phi^{-1}(C_{\gamma_1}\times C_{\gamma_2})\cap\phi^{-1}(C_{\gamma_1}\times C_{\gamma'_2})\cap  \phi^{-1}(C_{\gamma'_1}\times C_{\gamma_2}),
\end{equation}
where $C_{\gamma_1}$, $C_{\gamma'_1}$, $C_{\gamma_2}$ and $C_{\gamma'_2}$ are components of $C$; in particular, $\phi(A)=(R_1,R_2)$, where $R_1$ and $R_2$ are nodes of $C$ (see also Section \ref{Not2}). 

The first relevant information about $\psi_*\L_\psi$ is obtained in Theorems \ref{main2} and \ref{main3}; we can state them as follows.

  \begin{Thm}\label{thmA}
  Let $\pi\col \C\ra B$ be a smoothing of a nodal curve $C$. Let $\phi\col \wt\C^2\ra \C^2$ and $\psi\col\wt\C^3\ra \wt\C^2\times_B \C$ be good partial desingularizations. Then $\psi_*\L_\psi$ is a relatively torsion-free, rank-1 sheaf on $\wt\C^2\times_B\C/\wt\C^2$ of relative degree 0, whose formation commutes with base change. Moreover, the rational map $\alpha^2\circ\phi\col \wt\C^2\dra\J$ induced by $\psi_*\L_\psi$ is defined on an open subset of $\wt\C^2$ containing the distinguished points in $\phi^{-1}(R,R)$, for every reducible node $R$ of $C$.
  \end{Thm}

The question whether or not  $\psi_*\L_\psi$ is quasistable is then reduced to checking quasistabilty of 
the restriction of $\psi_*\L_\psi$ to the fibers $p_1^{-1}(A)$, where $A$ is a distinguished point of $\wt\C^2$ contained in $\phi^{-1}(R_1,R_2)$, for reducible nodes $R_1$ and $R_2$ of $C$ (see Lemma \ref{ext-red}). By Theorem \ref{thmA} we may also assume that $R_1$ and $R_2$ are distinct. 

The key step is to introduce a combinatorial condition on a distinguished point $A$ which turns out to be equivalent to the fact that $\alpha^2\circ\phi$ is defined on an open subset of $\wt\C^2$ containing $A$. If $A$ is contained in the intersection \eqref{intersection}, this condition only depends 
  on the \emph{nested sets of tails of $C$} associated to $(\gamma_1,\gamma_2)$, $(\gamma_1,\gamma'_2)$ and $(\gamma'_1,\gamma_2)$ (see Section \ref{Nestedsets} for the definition of these sets). 
  We say that $A$ is \emph{quasistable} if it satisfies such a combinatorial condition (see Section \ref{quas-sync-dist} for more details). The key result is contained in Theorem \ref{main4}; we can state it as follows.
 
  \begin{Thm}\label{thmB}
  Let $\pi\col \C\ra B$ be a smoothing of a nodal curve $C$ with a section $\sigma$ through its smooth locus. Let $\phi\col \wt\C^2\ra \C^2$ and $\psi\col\wt\C^3\ra \wt\C^2\times_B \C$ be good partial desingularizations. Consider a pair $(R_1,R_2)$ of reducible nodes of $C$, with $R_1\ne R_2$, and let $\A$ be the set of distinguished points of $\C^2$ contained in $\phi^{-1}(R_1,R_2)$. For every $A$ in $\A$, let $X_A=(p_1\circ\psi)^{-1}(A)$ and identify $p_1^{-1}(A)$ with $C$.
   The following properties are equivalent.
  \begin{itemize}
  \item
 $A$ is quasistable, for every $A$ in $\A$;
  \item
 $\psi|_{X_A*}(\L_\psi|_{X_A})$ is $\sigma(0)$-quasistable, for every $A$ in $\A$;
  \item
 $\psi|_{X_A*}(\L_\psi|_{X_A})$ is simple, for every $A$ in $\A$;     
   \item
      the rational map $\alpha^2\circ\phi\col\wt\C^2\dra\J$ induced by $\psi_*\L_\psi$ is defined on an open subset of $\wt\C^2$ containing $\A$.
  \end{itemize}
  \end{Thm}

Here, a torsion-free rank-1 sheaf $\I$ on a nodal curve $C$ is \emph{simple} if $\text{Hom}(\I,\I)=K$. In general, if $\I$ is quasistable, then it is simple, and it is easy to find examples showing that the two notions are not equivalent. It is worth to notice that Theorem \ref{thmB} implies the surprising result that $\psi_*\L_\psi$ is $\sigma$-quasistable if and only if it is simple.

Finally, we produce a global blowup of $\C^2$ fulfilling the local criterion of Theorem \ref{thmB}. This is done in Theorem \ref{main5} by taking blowups along (the strict transforms of) products 2-tails and 3-tails of $C$, that is connected subcurves $W$ of $C$ such that $\ol{C\setminus W}$ is connected and such that $W\cap \ol{C\setminus W}$ has cardinality 2 or 3. 
  
  \begin{Thm}
  Let $\pi\col \C\ra B$ be a smoothing of a nodal curve $C$. Let $W_1,\dots,W_N$ be the 2-tails and the 3-tails of $C$. Consider the sequence of blowups
  \[
\wh\C^2_N\stackrel{\phi_N}{\lra}\cdots \stackrel{\phi_2}{\lra}\wh\C^2_1\stackrel{\phi_1}{\lra}\wh\C^2_0\stackrel{\phi_0}{\lra}\C^2
  \]
  where $\phi_0$ is the blowup of $\C^2$ along its diagonal subscheme and $\phi_i$ is defined as the blowup of $\wh\C^2_{i-1}$ along the strict transform of the divisor $W_i\times W_i$ of $\C^2$ via $\phi_0\circ\cdots\circ \phi_{i-1}$. Then the rational map
  \[
\alpha^2\circ\phi_0\circ\cdots\circ\phi_N\col \wh\C^2_N\dra\J
  \]
  is a morphism.
  \end{Thm}

\subsection{Notation and terminology}\label{Not1}
Throughout the paper we will use the following notation and terminology.
We work over an algebraically closed field $K$. A \emph{curve} is a connected, projective and reduced scheme of dimension 1 over $K$. We will always consider curves with  nodal singularities.

Let $C$ be a curve. We denote the irreducible components of $C$ by $C_1,\ldots, C_p$. The genus of $C$ is $g:=1-\chi(\O_C)$. We denote by $\omega_C$ the dualizing sheaf of $C$.  A \emph{subcurve} of $C$ is a union of irreducible components of $C$. Let $Z$ be a proper subcurve of $C$. We let $Z^c:=\overline{C\setminus Z}$ and call it the \emph{complementary curve of $Z$}. We call a point in $Z\cap Z^c$ a \emph{terminal point of $Z$}, and we set 
\[
\term_Z:=Z\cap Z^c\; \text{ and } \; k_Z:=\#\term_Z.
\]
 Moreover, we set $\term_C=\term_\emptyset=\emptyset$. We say that $Z$ is 
  a \emph{tail} if $Z$ and $Z^c$ are connected; it is a \emph{$k$-tail} if $k_Z=k$.
  If $R$ is a node of $C$ such that $R\in C_\gamma\cap C_{\gamma'}$ for $(\gamma,\gamma')\in\{1,\dots,p\}^2$, we say that $Z$ \emph{crosses $R$} if $C_\gamma\cup C_{\gamma'}\subseteq Z$. 
  A node $R$ of $C$ is \emph{reducible} if $R\in Z\cap Z^c$ for some proper subcurve $Z$, otherwise it is \emph{irreducible}. We denote by $\N(C)$ the set of reducible nodes of $C$.  
 A subset $\Delta$ of the set of nodes of $C$ is \emph{disconnecting} if the normalization of $C$ at the points of $\Delta$ is not connected.

Let $Z$ and $Z'$ be subcurves of $C$.  
 We write $Z\prec Z'$ if $Z\subsetneq Z'$ and the intersection $\term_Z\cap\term_{Z'}$ is empty. Moreover, we write $Z\wed Z'$ to denote the union of the components of $C$ contained in $Z\cap Z'$. 
 If  $\term_Z\cap \term_{Z'}$ is nonempty, we say that the pair $(Z, Z')$ is \emph{terminal}, or that $Z$ is $Z'$-terminal, or that $Z'$ is $Z$-terminal; otherwise, we say that $(Z, Z')$ is \emph{free}.  
 We say that $(Z, Z')$ is \emph{perfect} if one of the following condition holds 
 $$Z\subseteq Z', \,\,\, Z'\subseteq Z, \,\,\, Z^c\subseteq Z', \,\,\,Z'\subseteq Z^c.$$
  If $\S$ is a set of subcurve of $C$, we say that $Z$ is $\S$-free if $(Z, W)$ is free, for every $W$ in $\S$. 
  If $\A$ and $\B$ are sets, we denote by $\A\sqcup \B$ the disjoint union of $\A$ and $\B$.

Given a map of curves $\mu\col C'\to C$ we say that an irreducible component of $C'$ is \emph{$\mu$-exceptional} if it is a smooth rational curve and is contracted by the map. A \emph{chain of rational curves (of length $d$)} is a curve which is the union of smooth rational curves $E_1,\ldots, E_d$ such that $E_i\cap E_j$ is empty if $|i-j|>1$ and $\#(E_i\cap E_{i+1})=1$. A \emph{chain of $\mu$-exceptional components} is a chain of $\mu$-exceptional curves.

 We define the curve $C(d)$ as a curve endowed with a map $\mu\col C(d)\to C$, called \emph{contraction map} such that $\mu$ is an isomorphism over the smooth locus of $C$, and the preimage of each node of $C$ consists of a chain of $\mu$-exceptional components of length $d$.
     Let $W$ be a subcurve  of $C$. A subcurve $Y$ of $C(d)$ is a 
 \emph{$W$-lifting} if $\mu(Y)=W$.  Since the maximal chains of $\mu$-exceptional components  have $d$ components, there are exactly $d$ $W$-liftings $L_0^W, L_1^W, \dots,L_{d-1}^W$ such that
 \[
 L_0^W\prec L_1^W\prec \cdots\prec L_{d-1}^W.
 \]
   We call the $W$-liftings $L_i^W$ \emph{the canonical $W$-liftings}.   
 We will use the previous setup just in the case $d=2$.

A \emph{family of curves} is a proper and flat morphism $\pi\col\mathcal C\ra B$ whose fibers are curves. The family $\pi\col\C\to B$ is called \emph{local} if $B=\Spec(K[[t]])$ and \emph{regular} if $\C$ is regular; in this case, we denote by $0$ and $\eta$ respectively the closed and the generic point of $B$, and we set $C_\eta:=\pi^{-1}(\eta)$. A \emph{smoothing} of a curve $C$ is a regular local family $\pi\col\C\to B$ whose fiber over $0$ is isomorphic to $C$ (and hence with $C_\eta$ smooth).

\section{Compactified Jacobians and Abel maps}

\subsection{Compactified Jacobians of nodal curves}\label{Jacobian}
\noindent
  Let $C$ be a nodal curve with irreducible components $C_1,\ldots,C_p$ and let $P$ be a smooth point of $C$. 
  The \emph{degree-$f$ Jacobian} of $C$ is the scheme parametrizing the equivalence classes of degree-$f$ invertible sheaves on $C$. In general, this scheme is neither proper nor of finite type. To solve these issues we resort to torsion-free rank-$1$ simple sheaves and to stability conditions.
  
A coherent sheaf $\I$ on $C$ is 
\emph{torsion-free} if it has no embedded components, \emph{rank-1} if it has 
generic rank 1 at each component of $C$, and \emph{simple} if 
$\text{Hom}(\I,\I)=K$. Equivalently, $\I$ is not simple if and only if the locus over which $\I$ is non-invertible consists of a set of disconnecting nodes of $C$. We call $\deg(\I) := \chi(\I )-\chi(\O_C)$ the \emph{degree}  of $\I$. 
  
  A \emph{polarization of degree $f$ on $C$} is a vector bundle of degree $f$ and rank $r>0$ on $C$. Given a polarization $\E$ of degree $f$ and rank $r$ on $C$, we consider its \emph{multi-slope}
  \[
  (e_{C_1},\dots,e_{C_p}):=\left(-\frac{\deg\E|_{C_1}}{r},\dots,-\frac{\deg\E|_{C_p}}{r}\right).
  \]
  Let $\I$ be a  torsion-free  rank-$1$ sheaf of degree $f$ on $C$. For every proper subcurve $Y$ of $C$, we define the sheaf $\I_Y$ as the sheaf $\I|_Y$ modulo torsion, and we set 
  \[
  \beta_{\I}(Y,\E):=\chi(\I_Y)-\sum_{C_i\subset Y} e_{C_i}.
  \]
  
   We say that $\I$ is \emph{$P$-quasistable over $Y$ with respect to $\E$} if the following conditions hold
\begin{eqnarray*}
0< \beta_\I(Y,\E) \leq k_Y, \text{ if } P\in Y & \text{ and } & 0\leq\beta_\I(Y,\E) < k_Y, \text{ if } P\not\in Y  
\end{eqnarray*}
Note that $\I$ is $P$-quasistable over $Y$ if and only if it is over $Y^c$. 
 We say that $\I$ is \emph{$P$-quasistable with respect to $\E$} if it is $P$-quasistable with respect to $\E$ over every proper subcurve of $C$. Since the conditions are additive on connected components it is enough to check them over connected subcurves. In fact, it is easy to see that it suffices to check on connected subcurves with connected complement.

Let $\pi\col \C\ra B$ be a family of nodal curves. Assume that there are sections $\sigma_1,\dots,\sigma_n\:  B\ra \C$ of $\pi$ through its smooth locus and such that, for every $b\in B$ and for every  irreducible component $Y$ of $\pi^{-1}(b)$, we have $\sigma_i(b)\in Y$, for some $i$ in $\{1,\dots, n\}$.  Notice that this condition is satisfied if $\pi$ is a smoothing of a nodal curve (see \cite[Proposition 5 of Section 2.3]{BLR}). 

Let $\sigma\:  B\ra \C$ be a section of $\pi$ through its smooth locus of and 
  $\E$ be a polarization of degree $f$ on $\C/B$, i.e. a vector bundle $\E$ on $\C$ of relative degree $f$ and rank $r>0$. We say that a sheaf $\I$ over $\C$ is \emph{$\sigma$-quasistable with respect to $\E$} if $\I|_{\pi^{-1}(b)}$ is a torsion-free rank-$1$ sheaf that is $\sigma(b)$-quasistable with respect to $\E|_{\pi^{-1}(b)}$, for every $b\in B$. The \emph{degree-$f$ compactified Jacobian of $\C/B$} is the scheme $\overline {\mathcal J^f}$ parametrizing degree-$f$ sheaves over $\C$ that are $\sigma$-quasistable with respect to $\E$. This scheme is proper and of finite type (see \cite[Theorems A and B]{E01}) and it represents the contravariant functor $\mathbf{J}$ from the category of locally Noetherian $B$-schemes to sets, defined on a $B$-scheme $S$ by
\[
\mathbf{J}(S):=\{\sigma_S\text{-quasistable sheaves of degree $f$ over } \C\times_B S\stackrel{\pi_S}\lra S\}/\sim
\]
where $\pi_S$ is the projection onto the second factor and $\sigma_S$ is the pullback of the section $\sigma$, and where $\sim$ is the equivalence relation given by $\I_1\sim \I_2$ if and only if there exists an invertible sheaf $M$ on $S$ such that $\I_1\cong \I_2\otimes \pi_S^*M$. 

If $g$ is the genus of the fibers of $\pi$, a \emph{canonical polarization on $\C/B$} is a polarization on $\C/B$ with the same multi-slope of the polarization on $\C/B$ given by
$$\mathcal E:=
\begin{cases}
\begin{array}{ll}
 \O_\C^{\oplus(2g-3)}\oplus \omega_\pi^{\otimes g-1} & \text{ if } g\ge 2, \\
\O_\C & \text{ if } g=1, \\
\O_\C\oplus \omega_\pi  & \text{ if } g=0.
\end{array}
\end{cases}$$
where $\omega_\pi$ is the relative dualizing sheaf of $\pi$.
If $\E$ is a canonical polarization on $\C/B$, for every $b\in B$ and for every proper subcurve $Y$ of $\pi^{-1}(b)$ we have
\[
\beta_{\I|_{\pi^{-1}(b)}}(Y):=\beta_{\I|_{\pi^{-1}(b)}}(Y,\E|_{\pi^{-1}(b)})=\deg(\I_Y)+\frac{k_Y}{2}.
\]
If a sheaf $\I$ over $\C$ is $\sigma$-quasistable with respect to a canonical polarization on $\C/B$, we simply say that \emph{$\I$ is $\sigma$-quasistable}.
 We will denote by $\J$ the degree-0 compactified Jacobian of $\C/B$ parametrizing sheaves on $\C/B$ that are $\sigma$-quasistable.

\subsection{The double and triple product}\label{Not2}

Let $C$ be a nodal curve with irreducible components $C_1,\dots,C_p$ and $\pi\:\C\ra B$ be a  smoothing of $C$.
We set
\[
\C^2:=\C\times_B \C \; \text{ and } \; \C^3:=\C^2\times_B\C. 
\]
Let us recall some important facts of the local geometry of $\C^2$ and $\C^3$; we refer the reader to \cite[Section 3 and 4]{CEP} for more details and the proofs of the statements. 

Recall that $B=\Spec K[[t]]$. The completion of the local ring of $\wt\C^2$ at $(R_1,R_2)$, 
with $\{R_1,R_2\}\subseteq \N(C)$, is given by
\[
\wh\O_{\C^2,(R_1,R_2)}\simeq \frac{K[[t,x_0,x_1,y_0,y_1]]}{(x_0x_1-t, y_0y_1-t)}\simeq \frac{K[[x_0,x_1,y_0,y_1]]}{(x_0x_1-y_0y_1)}
\]
where $x_0,x_1$ and $y_0,y_1$ are local coordinates of $\C$ at $R_1$ and $R_2$, and where the map $\pi\col\C\ra B$ is given locally at $R_1$ and $R_2$ by $t=x_0x_1$ and $t=y_0y_1$.

We see that, locally around $(R_1,R_2)$, the singular locus of $\C^2$ consists exactly of the point $(R_1,R_2)$. Assume that $R_1\in C_{\gamma_1}\cap C_{\gamma'_1}$ and $R_2\in C_{\gamma_2}\cap C_{\gamma'_2}$, for $(\gamma_1,\gamma'_1)$ and $(\gamma_2,\gamma'_2)$ in $\{1,\dots,p\}^2$. To get a desingularization of $\C^2$ locally around $(R_1,R_2)$, one can perform the blowup of the Weil divisor $C_{\gamma_1}\times C_{\gamma_2}$ of $\C^2$ (which is equivalent to the blowup of the Weil divisor $C_{\gamma'_1}\times C_{\gamma'_2}$ of $\C^2)$, or the blowup of the Weil divisor $C_{\gamma_1}\times C_{\gamma'_2}$ of $\C^2$ (which is equivalent to the blowup of the Weil divisor $C_{\gamma'_1}\times C_{\gamma_2}$ of $\C^2$). If $\eta\col \wh\C^2\ra \C$ is one of these two blowups, then there is an open subset $U$ of $\C^2$ containing $(R_1,R_2)$ such  that $\eta^{-1}(U)$ is smooth and such that $\eta$ is an isomorphism over $U\setminus\{(R_1,R_2)\}$ and $\eta^{-1}(R_1,R_2)$ is isomorphic to a smooth rational curve. This rational curve is contained in $C_{\gamma_1,\gamma_2,\phi}$ and $C_{\gamma'_1,\gamma'_2,\phi}$ (respectively in $C_{\gamma_1,\gamma'_2,\phi}$ and $C_{\gamma'_1,\gamma_2,\phi}$) if and only if  $\eta$ is the blowup along $C_{\gamma_1,\gamma_2}$ (respectively along $C_{\gamma_1,\gamma'_2}$).  If $R:=R_1=R_2$, with $R\in C_{\gamma}\cap C_{\gamma'}$, for $(\gamma,\gamma')$ in $\{1,\dots,p\}^2$, we can always desingularize $\C^2$ locally around $(R,R)$ by blowing up the diagonal subscheme of $\C^2$, which is equivalent to the blowup of the Weil divisor $C_{\gamma}\times C_{\gamma'}$ of $\C^2$.  The local picture of these blowups is illustrated in Figure 1.

Perform a chain of blowups
\begin{equation}\label{C2-good-def}
\phi\col\wt\C^2:=\wh\C^2_M\stackrel{\phi_M}{\lra}\cdots\stackrel{\phi_2}{\lra}\wh\C^2_1\stackrel{\phi_1}{\lra}\wh\C^2_0\stackrel{\phi_0}{\lra}\C^2
\end{equation}
where $\phi_0$ is the blowup along the diagonal divisor of $\C^2$ and where $\phi_i$ is the blowup of $\wh\C^2_{i-1}$ along the strict transform of some Weil divisor $Z_{i,1}\times Z_{i,2}$ of $\C^2$ via $\phi_0\circ\cdots\circ\phi_{i-1}$, for subcurves $Z_{i,1}$ and $Z_{i,2}$ of $C$. Since  $\phi_0$ gives rise to a desingularization of $\C^2$ locally around $(R,R)$, for $R$ reducible node  of $C$, and $\phi_i$ to a desingularization of $\C^2$ locally around the pairs $(R_1,R_2)$ of reducible nodes of $C$, for $R_1\in Z_{i,1}\cap Z_{i,1}^c$ and $R_2\in Z_{i,2}\cap Z_{i,2}^c$, we see that if the chain in \eqref{C2-good-def} is long and varied enough, the scheme $\wt\C^2$ is smooth away from the locus of points $(R_1,R_2)$, where one between $R_1$ and $R_2$ is an irreducible node of $C$. In particular, the strict transform via $\phi$ of any divisor of type $Z_1\times Z_2$, for subcurves $Z_1$ and $Z_1$ of $C$, is a Cartier divisor of $\wt\C^2$. In this case, we call a map $\phi\col \wt\C^2\ra \C^2$ as in \eqref{C2-good-def} a \emph{good partial desingularization of $\wt\C^2$}.

\[
\begin{xy} <20pt,0pt>:
(-3.5,-0.8)*{\bullet}="b";
"b"+(1.3,-1.3)*{\bullet}="c";
"b"+(1.3,-4)*{\bullet}="d";
(6,-0.8)*{\bullet}="f"; 
"f"+(3,0)*{\bullet}="e";
"f"+(-12.5,0)*{\bullet}="g";
"f"+(-1.3,-1.3)*{\bullet}; 
"f"+0;"f"+(2,0)**\dir{-};
"f"+0;"f"+(0,2)**\dir{-};
"f"+(-1.3,-1.3);"f"**\dir{-};
"f"+(1.3,1)*{C_{\gamma'_1,\gamma'_2,\phi}}; 
"f"+(-2.3,-2.3)*{C_{\gamma_1,\gamma_2,\phi}}; 
"f"+(0.7,-1.7)*{C_{\gamma'_1,\gamma_2,\phi}};
"f"+(-0.7,-0.4)*{\scriptstyle{\Ps^1}};
"f"+(-2.2,0.2)*{C_{\gamma_1,\gamma'_2,\phi}};
"f"+(-8.6,-0.4)*{\scriptstyle{\Ps^1}};
"f"+(-7.3,0.3)*{C_{\gamma'_1,\gamma'_2,\phi}};
"f"+(-10.2,-1.7)*{C_{\gamma_1,\gamma_2,\phi}};
"f"+(-10.5,1)*{C_{\gamma_1,\gamma'_2,\phi}}; 
"f"+(-7,-2.3)*{C_{\gamma'_1,\gamma_2,\phi}};
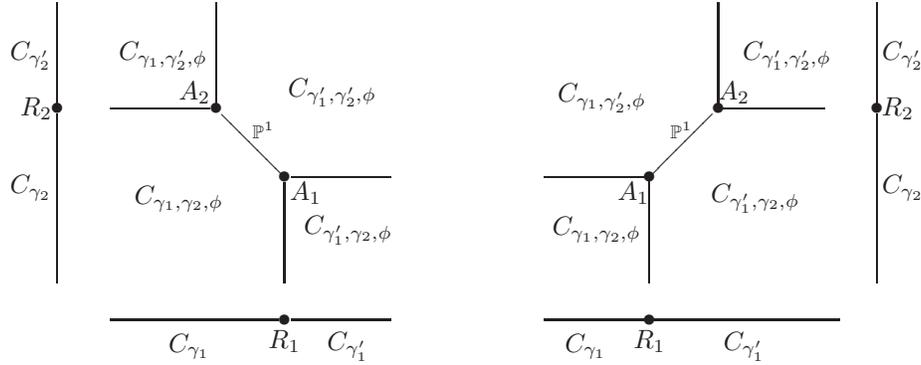
"f"+(-5,-5.5)*{\text{\bf Figure 1. } \text{The blowups $\wt\C^2\ra\C^2$ along $C_{\gamma_1}\times C_{\gamma_2}$ and $C_{\gamma_1}\times C_{\gamma'_2}$}.};
"f"+(0.3,0.3)*{A_2};
"f"+(-1.6,-1.6)*{A_1};
"f"+(-9.9,0.3)*{A_2};
"f"+(-7.8,-1.6)*{A_1};
"f"+(-1.3,-1.3);"f"+(-1.3,-3.3)**\dir{-};
"f"+(-1.3,-1.3);"f"+(-3.3,-1.3)**\dir{-};
"b";"b"+(1.3,-1.3)**\dir{-};
"b";"b"+(0,2)**\dir{-};
"b";"b"+(-2,0)**\dir{-};
"c";"c"+(0,-2)**\dir{-};
"c";"c"+(2,0)**\dir{-};
"d";"d"+(-3.3,0)**\dir{-};
"d";"d"+(2,0)**\dir{-};
"f"+(-1.3,-4);"f"+(-3.3,-4)**\dir{-};
"f"+(-1.3,-4);"f"+(2.3,-4)**\dir{-};
"f"+(-1.3,-4)*{\bullet};
"e";"e"+(0,2)**\dir{-};
"e";"e"+(0,-3.3)**\dir{-};
"g";"g"+(0,2)**\dir{-};
"g";"g"+(0,-3.3)**\dir{-};
"f"+(-10,-4.5)*{C_{\gamma_1}};
"f"+(-7,-4.5)*{C_{\gamma'_1}};
"f"+(-2.5,-4.5)*{C_{\gamma_1}};
"f"+(0.5,-4.5)*{C_{\gamma'_1}};
"f"+(-13,1)*{C_{\gamma'_2}};
"f"+(-13,-1.5)*{C_{\gamma_2}};
"f"+(3.5,1)*{C_{\gamma'_2}};
"f"+(3.5,-1.5)*{C_{\gamma_2}};
"f"+(-1.3,-4.4)*{R_1};
"f"+(-8.2,-4.4)*{R_1};
"f"+(-12.9,0)*{R_2};
"f"+(3.4,0)*{R_2};
\end{xy}
\]

Fix a good partial desingularizations $\phi\:\wt{\C}^2\ra \C^2$ of $\C^2$. 
For each $(\gamma,\gamma')$ in $\{1,\dots,p\}^2$, we let $C_{\gamma,\gamma',\phi}$ be the strict transform to $\wt\C^2$ of the divisor $C_\gamma\times C_{\gamma'}$ of $\C^2$ via the map $\phi$. 
We say that a point $A$ of $\wt\C^2$ is a \emph{distinguished point of $\wt\C^2$} if
\begin{equation}\label{distinguished-type}
A\in C_{\gamma_1,\gamma_2,\phi}\cap C_{\gamma_1,\gamma'_2,\phi}\cap C_{\gamma'_1,\gamma_2,\phi}
\end{equation}
 for distinct pairs $(\gamma_1,\gamma_2),(\gamma_1,\gamma'_2),(\gamma'_1,\gamma_2)$ in $\{1,\dots,p\}^2$. We have $\phi(A)=(R_1,R_2)$, for reducible nodes $R_1$ and $R_2$ such that $R_1\in C_{\gamma_1}\cap C_{\gamma'_1}$ and $R_2\in C_{\gamma_2}\cap C_{\gamma'_2}$. 
 There are exactly two distinguished points in $\phi^{-1}(R_1,R_2)$. Indeed, assume w.l.g. that $\phi$ is, locally around $(R_1,R_2)$, the blowup of $\C^2$ along $C_{\gamma_1}\times C_{\gamma'_2}$; these distinguished points are
\[
\begin{array}{lll}
A_1=C_{\gamma_1,\gamma_2,\phi}\cap C_{\gamma_1,\gamma'_2,\phi}\cap C_{\gamma'_1,\gamma_2,\phi}\cap \phi^{-1}(R_1,R_2), \\
A_2=C_{\gamma'_1,\gamma'_2,\phi}\cap C_{\gamma'_1,\gamma_2,\phi}\cap C_{\gamma_1,\gamma'_2,\phi}\cap \phi^{-1}(R_1,R_2).
\end{array}
\]

\smallskip
 
 Consider the fiber product $\wt\C^2\times_B\C$. For every distinguished point $A$ of $\C^2$ as in \eqref{distinguished-type} and for every node $S$ of $\C$, the local completion of $\wt\C^2\times_B\C$ at $(A,S)$ is given by
\[
\wh\O_{\wt\C^2\times_B\C,(A,S)}\simeq \frac{K[[u,v,w,z_0,z_1,t]]}{(uvw-t,z_0z_1-t)}\simeq \frac{K[[u,v,w,z_0,z_1]]}{(uvw-z_0z_1)},
\]
where $u,v,w$ are local coordinates of $\wt\C^2$ at $A$ and where $z_0,z_1$ are the local coordinates of $\C$ at $S$, and where the map $\pi\col\C\ra B$ is given locally at $S$ by $t=z_0z_1$.
 
 We see that $(A,S)$ is contained in the singular locus of $\wt\C^2\times_B\C$. Assume that $S\in C_{\gamma}\cap C_{\gamma'}$, for $(\gamma,\gamma')$ in $\{1,\dots,p\}^2$. To get a desingularization of $\wt\C^2\times_B \C$ locally around $(A,S)$, one can blowup first along the Weil divisor $C_{\gamma_1,\gamma_2,\phi}\times C_{\gamma}$ of $\C^3$ and then along the strict transform of the Weil divisor $C_{\gamma_1,\gamma'_2,\phi}\times C_{\gamma'}$. In fact, that there are six ways to produce a desingularization of $\C^3$ locally around $(A,S)$ by means of similar blowups. The local picture of these blowups is illustrated in Figure 2.

Perform a chain of blowups
\begin{equation}\label{C3-good-def}
\psi\col\wt\C^3:=\wh\C^3_N\stackrel{\psi_N}{\lra}\cdots\stackrel{\psi_3}{\lra}\wh\C^3_2\stackrel{\psi_2}{\lra}\wh\C^3_1\stackrel{\psi_1}{\lra}\wh\C^3_0:=\wt\C^2\times_B\C
\end{equation}
where $\psi_i$ is the blowup of $\wh\C^3_{i-1}$ along (the strict transform of) some Weil divisor 
$\phi^{-1}(Z_{i,1}\times Z_{i,2})\times Z_{i,3}$ of $\wt\C^2\times_B\C$ (via $\psi_1\circ\cdots\circ\psi_{i-1}$), 
for subcurves $Z_{i,1},$ $Z_{i,2},$ and $Z_{i,3}$ of $C$. If the chain in \eqref{C3-good-def} is long and varied enough, the strict transforms via $\psi$ of any divisor of type $Z_1\times Z_2\times Z_3$, for subcurves $Z_1,$ $Z_2,$ and $Z_3$ of $C$, is a Cartier divisor of $\wt\C^3$. In this case, we call a map $\psi\col\wt\C^3\ra \wt\C^2\times_B\C$ as in \eqref{C3-good-def} a \emph{good partial desingularization of $\wt\C^2\times_B\C$}. 
 
 \smallskip
 
Fix a good partial desingularization $\psi\:\wt{\C}^3\ra \wt{\C}^2\times_B\C$ of $\wt\C^2\times_B\C$. For each $(\gamma,\gamma',m)$ in $\{1,\dots,p\}^3$, we let $C_{\gamma,\gamma',m,\psi}$ be the strict transform to 
$\wt\C^3$ of the divisor $C_{\gamma,\gamma',\phi}\times C_m$ of $\wt\C^2\times_B \C$ via $\psi$. We will consider the family of curves $p_1\:\wt{\C}^2\times_B\C\ra \wt{\C}^2$, where $p_1$ is the projection onto the first factor. We let $\xi\:\wt\C^3\ra\wt\C^2\times_B \C\ra\C$ be the composition of $\psi$ with the projection onto the last factor. 
 For  a distinguished point $A$ of $\wt\C^2$, we set 
\[
X_A:=(p_1\circ\psi)^{-1}(A).
\]

 Recall that there is an identification of $X_A$ with $C(2)$ and of $\psi(X_A)$ with $C$, and that the contraction map $\mu\: C(2)\ra C$ identifies with $\psi|_{X_A}\: X_A\ra \psi(X_A)$ (see \cite[Statement 2 of Lemma 3.2]{CEP}). For each $\gamma$ in $\{1,\dots,p\}$, we will denote by $\wh C_\gamma$  the strict transform to $X_A$ of the component $C_\gamma$ of $\psi(X_A)$ via $\psi|_{X_A}$. 
\[
\begin{xy} <30pt,0pt>:
(0,0)*{}="a"; 
(-4.5,0)*{}="b";
"b"+(4,0)*{}="c";
(4.5,0)*{\bullet}="d";
(0,1)*{}="e";
(-0.2,0)*{\bullet};
"b"+(-0.2,-0.2);"b"+(0.7,0.7)**\dir{-};
"b"+(-0.2,0.2);"b"+(0.7,-0.7)**\dir{-};
"b"+(1.5,0.6);"b"+(0,0.6)**\crv{"b"+(0.7,0.3)};
"b"+(1.5,-0.6);"b"+(0,-0.6)**\crv{"b"+(0.7,-0.3)};
"b"+(0.63,0.2)*{\scriptstyle{E_{S,\gamma'}}};
"b"+(0.56,-0.2)*{\scriptstyle{E_{S,\gamma}}};
"b"+(1.8,0.6)*{\scriptstyle{\wh C_{\gamma'}}};
"b"+(1.8,-0.6)*{\scriptstyle{\wh C_{\gamma}}};
"c"+(1.5,0.3);"c"+(0,-0.2)**\crv{"c"+(0.7,0.3)};
"c"+(1.5,-0.3);"c"+(0,0.2)**\crv{"c"+(0.7,-0.3)};
"b"+(5.8,0.3)*{\scriptstyle{C_{\gamma'}}};
"b"+(5.8,-0.3)*{\scriptstyle{C_{\gamma}}};
"b"+(3.8,0)*{\scriptstyle{(A,S)}};
"d"+(-0.8,0);"d"**\dir{-};
"d"+(0,-0.8);"d"**\dir{-};
"d"+(0.6,0.6);"d"**\dir{-};
"d"+(0.2,0)*{\scriptstyle{A}};
"d"+(0.7,-0.3)*{\scriptstyle{w=0}};
"d"+(-0.2,0.4)*{\scriptstyle{v=0}};
"d"+(-0.4,-0.4)*{\scriptstyle{u=0}};
"d"+(-3.8,-1.3)*{C=\psi(X_A)=p_1^{-1}(A)\subset \wt\C^2\times_B\C};
"d"+(-8,-1.3)*{C(2)=X_A\subset \wt\C^3};
"d"+(0,-1.28)*{A\in\wt\C^2};
"d"+(-1.7,0)*{\stackrel{p_1}{\lra}};
"d"+(-6.3,0)*{\stackrel{\psi}{\lra}};
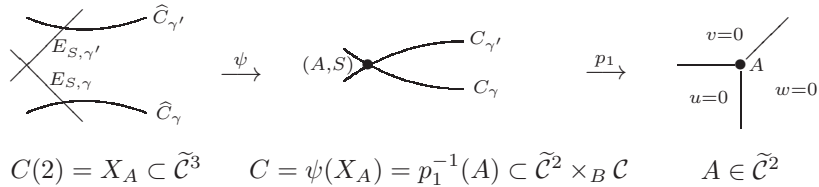
"d"+(-4.5,-2)*{\text{\bf Figure 2. } \text{The local picture of the chain of maps $\wt\C^3\stackrel{\psi}{\ra}\wt\C^2\times_B\C\stackrel{p_1}{\ra}\wt\C^2$.}};
\end{xy}
\]
 Moreover, if the node $S$ of $C$ is such that $S\in C_{\gamma}\cap C_{\gamma'}$, for $(\gamma,\gamma')$ in $\{1,\dots,p\}^2$, we will denote by $E_{S,\gamma}$ (respectively $E_{S,\gamma'}$) the $\psi|_{X_A}$-exceptional component of $X_A$ contracted to $(A,S)$ and intersecting $C_{\gamma}$ (respectively $C_{\gamma'}$), and we put
 \[
 E_S:=E_{S,\gamma}\cup E_{S,\gamma'}.
 \]

Let $A$ be a distinguished point of $\wt\C^2$ as in \eqref{distinguished-type}. 
A \emph{slice of $\wt\C^2$ through $A$} is a section $\lambda\:B\to\wt\C^2$ of $\wt\C^2/B$ sending the 
special point of $B$ to $A$ and such that the pullbacks of $C_{\gamma_1,\gamma_2,\phi}$, $C_{\gamma_1,\gamma'_2,\phi}$ and $C_{\gamma'_1,\gamma_2,\phi}$ via $\lambda$ are all 
prime. 
Form the Cartesian diagram
\begin{equation}\label{diagram}
\begin{CD}
\W @>\theta >> \wt\C^3\\
@V\rho VV @Vp_1\psi VV\\
B @>\lambda >> \wt\C^2
\end{CD}
\end{equation}
By \cite[Proposition 4.6]{CEP} we have that $\rho$ is a smoothing of $X_A$, which we call \emph{the $\lambda$-smoothing of $X_A$}.

\subsection{Abel maps and their extensions}\label{Abel-maps} Keep the notation of Section \ref{Not2}. Let $\sigma$ be a section of $\pi\col\C\ra B$ through its smooth locus. Let $\P$ be a relative invertible sheaf on $\C/B$ of relative degree $2+f$ and $\E$ be a polarization of degree $f$ on $\C/B$. 
 The $(\P,\E)$-Abel map of $\C/B$  is the map
 \begin{equation}
 \alpha^2_{\P,\E}\col\C^2\dashrightarrow\ol{\mathcal J^f}
 \end{equation}
sending a pair $(Q_1,Q_2)$ of points of the fiber $\C_\eta$ of $\pi$ over the generic point $\eta$ of $B$ to the invertible sheaf $\P|_{\C_\eta}(-Q_1-Q_2)$. 
If $\deg_{C_i}\P=\deg_{C_i}\O_\C(2\sigma(B))$, for every $i$ in $\{1,\dots,p\}$, and if $\E$ is a canonical polarization, we set
\begin{equation}\label{AJ}
\alpha^2:=\alpha^2_{\P,\E}
\end{equation}
and call it \emph{the Abel--Jacobi map of $\C/B$.}

The $(\P,\E)$-Abel map of $\C/B$ is defined at least on the open subset $\C_\eta\times_B\C_\eta$ of $\C^2$. To get a natural extension, we resort to the notion of N\'eron model. 
 Let $\mathcal J^f_{\C_\eta}$ be the degree-$f$ Jacobian of $\C_\eta$. 
The \emph{N\'eron model of $\mathcal J^f_{\C_\eta}$} is a $B$-scheme $N(\mathcal J^f_{\C_\eta})$, smooth and separated over $B$, whose generic fiber is isomorphic to  $\mathcal J^f_{\C_\eta}$ and uniquely determined by the following universal property (the \emph{N\'eron mapping property}): for every $B$-smooth scheme $Z$ with generic fiber $Z_\eta$ and for every $\{\eta\}$-morphism $u_\eta\:  Z_\eta\ra J_{\C_\eta}$, there is a unique extension of $u_\eta$ to a morphism $u\:  Z\ra N(\mathcal J^f_{\C_\eta})$. For more details on N\'eron models, we refer the reader to \cite{BLR}.

\begin{Thm}\label{Ner}
The $B$-smooth locus of $\ol{\mathcal J^f}$ is isomorphic to $N(J^f_{\C_\eta})$.
\end{Thm}

\begin{proof}
See \cite{B}, \cite[Theorem A]{K} and  \cite[Theorem 3.1]{MeVi}; see also \cite{CAJM}.
\end{proof}

Let $\dot{\C}$ be the smooth locus of $\pi$ and set $\dot\C^2:=\dot\C\times_B\dot\C$.
Since $\dot\C^2$ is $B$-smooth, combining the N\'eron mapping property and Theorem \ref{Ner}, the map $\alpha^2_{\P,\E}$ induces a morphism
\begin{equation}\label{A-N}
\alpha^2_{\P,\E}\:  \dot{\C}^2\lra \ol{\mathcal J^f},
\end{equation}
which is called the \emph{Abel--N\'eron map}. 
 Although the definition of the Abel--N\'eron map is natural, it is not modular. We will discuss further down how one can reinterpret this problem.

If $M$ is a degree-$f$ invertible sheaf on $C$, then there is a divisor $D$ of $\C$ supported on $C$ such that 
$L\otimes \O_\C(D)|_C$ is an invertible sheaf which is $\sigma(0)$-quasistable with respect to $\E$. Moreover, $D$ only depends on the degree of $L$ on the components of $C$ and it is uniquely determined up to sum-up multiples of the divisor $C$ of $\C$.

For each $(\gamma,\gamma')\in \{1,\dots,p\}^2$ and for smooth points $Q_{\gamma}$ and $Q_{\gamma'}$ of $C$ lying respectively on $C_{\gamma}$ and $C_{\gamma'}$, we let $Z_{\gamma,\gamma'}^{\P,\E}$ be the divisor of $\C$ supported on $C$ such that the invertible sheaf
\[
\P|_C(-Q_\gamma-Q_{\gamma'})\otimes\O_\C(Z_{\gamma,\gamma'}^{\P,\E})|_C
\]
 is $\sigma(0)$-quasistable with respect to $\E$. Write    
 $$Z_{\gamma,\gamma'}^{\P,\E}:=-\sum_{m=1}^p \alpha_{\gamma,\gamma',m}^{\P,\E} C_m, \text{ for } \alpha_{\gamma,\gamma',m}^{\P,\E}\in \ze,$$
  where $\alpha^{\P,\E}_{\gamma,\gamma',m}\ge 0$.  
  If $\deg_{C_i}\P=\deg_{C_i}\O_\C(2\sigma(B))$, for every $i$ in $\{1,\dots,p\}$, and if $\E$ is a canonical polarization, we set  
  \[
  Z_{\gamma,\gamma'}:=Z_{\gamma,\gamma'}^{\P,\E}\;  \text{ and } \; \alpha_{\gamma,\gamma',m}:=\alpha_{\gamma,\gamma',m}^{\P,\E}.
  \]
  For every $(\gamma,\gamma')$ and $(m,n)$ in $\{1,\dots,p\}^2$, we define
\begin{equation}\label{delta-def} 
\delta(\gamma,\gamma',m,n):=\alpha_{\gamma,\gamma',m}-\alpha_{\gamma,\gamma',n}.
\end{equation} 
 Notice that 
 \begin{equation}\label{delta-prop}
 \delta(\gamma,\gamma',m,n)=\delta(\gamma',\gamma,m,n)=-\delta(\gamma',\gamma,n,m).
 \end{equation}

 Let $\phi\:\wt{\C}^2\ra \C^2$ and $\psi\:\wt{\C}^3\ra \wt{\C}^2\times_B\C$ be good partial desingularizations. 
Consider the maps $\rho_1,\rho_2\:\C^3\to\C^2$, where $\rho_i$ is the projection onto the product 
over $B$ of the $i$-th and last factors of $\C^3$, for $i=1,2$, and put $\Delta_i:=\rho^{-1}_i(\Delta)$.  
We let $\wt\Delta_1$ and $\wt\Delta_2$ be the strict transforms of $\Delta_1$ and $\Delta_2$ to $\wt\C^3$. Recall that $\xi$ is the composed map $\xi\:\wt\C^3\stackrel{\psi}{\ra}\wt\C^2\times_B \C\ra\C$, where the second map is the projection onto the last factor. 
 Consider the invertible sheaf  $\L^{\P,\E}_{\psi}$ on $\wt\C^3$ defined as
\[
\L^{\P,\E}_{\psi}:=\xi^*\P\otimes\I_{\wt\Delta_1|\wt\C^3}\otimes \I_{\wt\Delta_2|\wt\C^3}\otimes\O_{\wt\C^3}\left(-\sum_{1\le i,k\le p}\sum_{m\in\{1,\dots,p\}} \alpha_{i,k,m}^{\P,\E} C_{i,k,m,\psi}\right).
\]
If $\deg_{C_i}\P=\deg_{C_i}\O_\C(2\sigma(B))$, for every $i$ in $\{1,\dots,p\}$, and $\E$ is a canonical polarization, then we set 
\begin{equation}\label{Lpsi}
\L_\psi:=\L^{\P,\E}_\psi.
\end{equation}
The relative sheaf $\psi_*\L^{\P,\E}_\psi$ on $\wt\C^2\times_B\C/\wt\C^2$ induces the rational map
\begin{equation}\label{alphaPE}
\alpha^2\circ\phi\col\wt\C^2 \dashrightarrow\ol{\mathcal J^f}
\end{equation}
agreeing with the $(\P,\E)$-Abel map of $\C/B$ over $\C_\eta\times_B\C_\eta$ 
and sending a pair $(Q_\gamma,Q_{\gamma'})$ of smooth points of $C$ lying on $C_{\gamma}$ and $C_{\gamma'}$, to the $\sigma(0)$-quasistable invertible sheaf 
\[
\P|_C(-Q_\gamma-Q_{\gamma'})\otimes \O_\C(Z_{\gamma,\gamma'}^{\P,\E})|_C.
\]

Recall that $p_1\col\wt\C^2\times_B\C\ra \wt\C^2$ denotes the projection onto the first factor. Let $\sigma'$ be the section of $p_1$ obtained as pull-back of the section $\sigma$ of $\pi\col\C\ra B$. Since the restriction of $\psi_*\L^{\P,\E}_\psi$ to $p_1^{-1}(\dot\C^2)$ is a $\sigma'$-quasistable invertible sheaf on $p_1^{-1}(\dot\C^2)/\dot\C^2$, we see that the Abel--N\'eron map is induced by such a restriction. Therefore, to get a modular description of the Abel--N\'eron map, it suffices to describe the integers $\alpha^{\P,\E}_{i,k,m}$ appearing in the definition of $\L^{\P,\E}_\psi$. It is possible to do that for the integers $\alpha_{i,k,m}$ by means of certain subcurves of $C$, as we will see in Theorem \ref{main1} of the next section.

Finally, let us recall an important property whose proof is in \cite[Proposition 4.6]{CEP}.
Fix a slice $\lambda\col B\ra \wt\C^2$ of $\wt\C^2$ through $A$, let $\W\ra B$ be the $\lambda$-smoothing of $X_A$ and $\theta\col \W\ra \wt\C^3$ be the induced map (see Diagram \ref{diagram}). If $(R_1,R_2)$ is a pair of reducible nodes of $C$, it follows that there are relative Cartier divisors $\Gamma_1$ and $\Gamma_2$ on $\W/B$ intersecting transversally $X_A$ respectively in $E_{R_1,\gamma_1}$ and $E_{R_2,\gamma_2}$, and a Cartier divisor $D$ of $\W$ supported on $\psi|_{X_A}$-exceptional components such that 
\begin{equation}\label{diag-type}
  \theta^*\left(\I_{\wt\Delta_1|\wt\C^3}\otimes \I_{\wt\Delta_2|\wt\C^3}\right)\simeq  \O_{\W}(D-\Gamma_1-\Gamma_2).
\end{equation}

\subsection{Nested sets of tails}\label{Nestedsets} Keep the notation of Sections \ref{Not2} and \ref{Abel-maps}.  For every $\gamma$ in $\{1,\dots,p\}$, consider the \emph{set of nested $1$-tails of $C$ associated to $C_\gamma$}
\[
\T^1_\gamma:=\{Z : \text{$Z$ is a $1$-tail such that $C_\gamma\subseteq Z$ and $\sigma(0)\in Z^c$ }\}.
\]
By \cite[Lemma 4.3]{CE}, if $W$ and $W'$ are in $\T^1_\gamma$, then either $W\prec W'$ or $W'\prec W$.

\begin{Lem}\label{2tails}
Fix $(\gamma,\gamma')$ in $\{1,\dots,p\}^2$. If $Z$ and $Z'$ are 2-tails of $C$ such that $C_\gamma\cup C_{\gamma'}\subseteq Z\wedge Z'$ and $\sigma(0)\in Z^c\wedge (Z')^c$, then $Z\wedge Z'$ is a 2-tail of $C$.
\end{Lem}

\begin{proof}
See \cite[Proposition 4.1]{P1}.
\end{proof}

Thanks to Lemma \ref{2tails}, we can define the \emph{set of nested 2-tails of $C$ associated to $(C_{\gamma},C_{\gamma'})$} as
\[
\T^2_{\gamma,\gamma'}=\{W^2_0,\dots,W^2_M\}
\]
where, if we set $W^2_{-1}:=\emptyset$, the tail $W^2_t$ is inductively defined as the 2-tail which is minimal (with respect to inclusion) among the 2-tails $Z$ of $C$ such that $C_\gamma\cup C_{\gamma'}\subseteq Z$, $\sigma(0)\in Z^c$ and $W^2_{t-1}\prec Z$, for every $t$ in $\{0,\dots,M\}$. 

\begin{Lem}\label{3tails}
Fix $(\gamma,\gamma')$ in $\{1,\dots,p\}^2$. If $Z$ and $Z'$ are 3-tails of $C$ that are $\T^2_{\gamma,\gamma'}$-free and such that $C_\gamma\cup C_{\gamma'}\subseteq Z\wedge Z'$ and $\sigma(0)\in Z^c\wedge (Z')^c$, then $Z\wedge Z'$ is a $\T^2_{\gamma,\gamma'}$-free $3$-tail of $C$.
\end{Lem}

\begin{proof}
See \cite[Proposition 4.5]{P1}.
\end{proof}

Thanks to Lemma \ref{3tails}, we can define the \emph{set of nested $3$-tails of $C$ associated to $(C_{\gamma},C_{\gamma'})$} as 
\[
\T^2_{\gamma,\gamma'}=\{W^3_0,\dots,W^3_N\}
\]
where, if we set $W^3_{-1}:=\emptyset$, the tail $W^3_t$ is inductively defined as the 2-tail of $C$ which is minimal (with respect to inclusion) among the 2-tails $Z$ of $C$ that are $\T^2_{\gamma,\gamma'}$-free and such that $C_\gamma\cup C_{\gamma'}\subseteq Z$, $\sigma(0)\in Z^c$ and $W^3_{t-1}\prec Z$, for every $t\in\{0,\dots,N\}$.

Notice that the sets $\T^1_\gamma$, $\T^2_{\gamma,\gamma'}$ and $\T^3_{\gamma,\gamma'}$ are totally ordered sets (with respect to inclusion). For every $(\gamma,\gamma')$ in $\{1,\dots,p\}^2$, we set
\[
\T_{\gamma,\gamma'}=\T^1_\gamma \sqcup \T^1_{\gamma'}\sqcup \T^2_{\gamma,\gamma'}\sqcup \T^3_{\gamma,\gamma'}.
\]

\begin{Thm}\label{main1}
Let $\pi\col \C\ra B$ be a smoothing of a nodal curve $C$ with irreducible components $C_1,\dots,C_p$. For every $(\gamma,\gamma')$ in $\{1,\dots,p\}^2$, we have
\[
\O_\C(Z_{\gamma,\gamma'})|_C\simeq \O_\C\left(-\sum_{Z\in \T_{\gamma,\gamma'}} Z\right)|_C.
\]
\end{Thm}

\begin{proof}
See \cite[Theorem 6.3]{P1}
\end{proof}

\smallskip

We will often use the following results.
 
\begin{Lem}\label{maxT2}
If $Z$ is a 2-tail of $C$ such that $C_\gamma\cup C_{\gamma'}\subseteq Z$ and $\sigma(0)\subseteq Z^c$, then there is a $Z$-terminal tail $W$ in $\T^2_{\gamma,\gamma'}$ such that $W\subseteq Z$.
\end{Lem}

\begin{proof}
See \cite[Corollary 4.2]{P1}.
\end{proof}

\begin{Lem}\label{maxT3} 
  If $Z$ is a 3-tail of $C$ such that $C_\gamma\cup C_{\gamma'}\subseteq Z$ and $\sigma(0)\subseteq Z^c$, then there is a $Z$-terminal tail $W$ such that either $W\in \T^2_{\gamma,\gamma'}$, or $W\in\T^3_{\gamma,\gamma'}$ and $W\subseteq Z$.
\end{Lem}

\begin{proof}
See \cite[Corollary 4.6]{P1}.
\end{proof}

\begin{Lem}\label{free-perf}
Let $Z$ be a tail of a nodal curve $C$. The following properties hold
 \begin{enumerate} [(i)]
\item  \label{free-perf(i)}
if $\term_Z\subset Z'$, for some tail $Z'$ of $C$, then either $Z\subseteq Z'$, or $Z^c\subseteq Z'$;
 \item \label{free-perf(ii)}
if $\#(\term_Z\cap\term_{Z'})=k_Z-1$,  for some tail $Z'$ of $C$, then $(Z,Z')$ is perfect;
 \item \label{free-perf(iii)}
 if $k_Z\ge 2$, then $(Z, Z')$ is free, for every 1-tail $Z'$ of $C$.
\end{enumerate}
\end{Lem}

\begin{proof}
See \cite[Lemma 3.4]{P1}.
\end{proof}

\section{Admissibility}
\noindent
Let $\pi\: \C\ra B$ be a smoothing of a nodal curve $C$. In this Section we will show that, for given good partial desingularizations  
$\phi\:\wt\C^2\ra \C^2$ and $\psi\col \wt\C^3\ra \wt\C^2\times_B\C$, $\psi_*\L_\psi$ is a relatively torsion-free rank-1 sheaf on $\wt\C^2\times_B\C/\wt\C^2$ (see Theorem \ref{main2}).   

\subsection{Comparing nested sets of tails}
Keep the notation of Sections \ref{Not2} and \ref{Abel-maps}.
For a subset $\N$ of the set of nodes of $C$ and for a set $\S$ of subcurves of $C$, we define  
\[
d_{\S, \N}:=\#\{W\in \S : R\in \term_W, \text{ for some } R\in \N\}.
\]
For every $(i,j,k)\in\{1,\dots,p\}$ and $s\in\{1,2,3\}$, we set
\[
\T^s_{i,k|j,k}:=(\T^s_{i,k}\cup \T^s_{j,k})\setminus (\T^s_{i,k}\cap \T^s_{j,k}).
\]
It is easy to see that if $\T^1_{i,k;j,k}$ is nonempty, then it consists exactly of a tail such that $C_i\cap C_j$ is its set of terminal points. 

\begin{Prop}\label{Diff}
Let $C$ be a nodal curve with irreducible components $C_1,\dots C_p$. Let $(i,j,k)$ in $\{1,\dots,p\}^3$, where $i\ne j$ and $C_i\cap C_j$ is nonempty. We have
\begin{equation}\label{diff-rel}
d_{\T_{i,k}\cup \T_{j,k},C_i\cap C_j}\le1.
\end{equation}
For each $s$ in $\{2,3\}$, the set $\T^s_{i,k | j,k}$ is a totally ordered set with respect to inclusion; 
the minimal of the tails of $\T^s_{i,k | j,k}$ containing a tail $Z$ of $\T^s_{i,k | j,k}$ is $Z$-terminal. 
The set $\T^s_{i,k | j,k}$ is nonempty if and only if exactly one of the following conditions hold
\begin{itemize}
\item[(i)]
there is a unique 
tail of $\T^s_{i,k}\cup \T^s_{j,k}$ that does not contain $C_i\cup C_j$; this unique tail is the minimal element of $\T^s_{i,k | j,k}$.
\item[(ii)]
$s=3$ and there is a unique $Z$-terminal tail in $\T^3_{i,k}\cup\T^3_{j,k}$, where $Z$ is the maximal tail of $\T^2_{i,k | j,k}$; this unique tail is the minimal element of $\T^3_{i,k | j,k}$.
\end{itemize}
Finally, if $\T^s_{i,k | j,k}$ is nonempty for some $s$ in $\{2,3\}$, then $\T^1_{i,k | j,k}$ is empty.
\end{Prop}

\begin{proof}
Assume that $W$ and $W'$ are tails respectively of $\T_{i,k}$ and $\T_{j,k}$ terminating in $C_i\cap C_j$, and hence $W\ne W'$. To show that \eqref{diff-rel} holds, we may assume that $k_W\ge2$, otherwise $W=W'$, because any $k$-tail with $k\ge 2$ is $Z$-free, for a 1-tail $Z$ of $C$. 
Notice that $(W,W')$ is not perfect, because 
\[
C_i\subseteq W\wedge (W')^c, \; C_j\subseteq W^c\wedge W' \; \text{ and }  \; 
C_k\subseteq W\wedge W'.
\]
Thus, by Lemma \ref{free-perf}, we have $k_W=k_{W'}=3$, and there is a node $R$ of $C$ such that 
\[
\term_W\cap \term_{W'}=C_i\cap C_j=\{R\}.
\]
In the sequel, we will use the following relation that we proved so far
\begin{equation}\label{byproduct}
d_{\T^2_{i,k}\cup\T^2_{j,k} , C_i\cap C_j}\le 1
\end{equation}

We claim that for each $s\in\{2,3\}$, the set $\T^s_{i,k | j,k}$ is a totally ordered set with respect to inclusion and the minimal tail of $\T^s_{i,k | j,k}$ containing a tail $Z$ of $\T^s_{i,k | j,k}$ is $Z$-terminal. Notice that the claim for $s=2$ implies that \eqref{diff-rel} holds: Indeed, since $W\cup W'$ crosses $R$ and $W\wedge W'$ does not contain $R$, these subcurves are 2-tails, by \cite[Lemmas 3.1 and 3.3]{P1} and by the fact that any $k$-tail with $k\ge2$ is $Z$-free, for a 1-tail $Z$ of $C$. In particular, Lemma \ref{maxT2} implies that there are $(W\cup W')$-terminal tails $X\in \T^2_{i,k}$ and $X'\in \T^2_{j,k}$ that are contained in $W\cup W'$. If we write
\[
\term_{W\cup W'}=\{S,S'\}, \text{ for } S\in \term_W \text{ and } S'\in\term_{W'},
\] 
 it follows that 
 \[
 S'\in \term_X, \; S'\not\in X', \; S\in \term_{X'} \text{ and } S\not\in X 
 \]
 Hence $X\not\in\T^2_{j,k}$, $X'\not\in\T^2_{i,k}$; also, $X$ and $X'$ do not contain each other, contradicting the fact that $\T^2_{i,k | j,k}$ is totally ordered.

 We prove the claim for $s=2$. Write 
 \[
 \T^2_{i,k}=\{W_{2t}\}_{t\ge 0} \; \text{ and } \; \T^2_{j,k}=\{W_{2t+1}\}_{t\ge0},
 \]
  where $W_t\prec W_{t+2}$, for each $t\ge0$. 
  If $\T^2_{i,k | j,k}$ is nonempty, then the minimal tails of $\T^2_{i,k}$ and $\T^2_{j,k}$ are different. Hence, using \eqref{byproduct}, exactly one of these tails do not contain $C_i\cup C_j$ (in particular, notice that (i) holds). 
  We may assume that $W_0$ does not contain $C_i\cup C_j$ and $W_0\in\T^2_{i,k | j,k}$. If $\T^2_{j,k}$ is empty, then $\T^2_{i,k}=\{W_0\}$ and we are done. If not, by \eqref{byproduct} we have $C_i\cup C_j\subseteq W_1$, and hence $W_0\subsetneq W_1$, by the minimal property of $W_0$. We have two possibilities. Assume first that $W_0\prec W_1$. Then we have $\#\T^2_{i,k}\ge2$, with $W_2\subseteq W_1$. Since the other inclusion holds by the minimal property of $W_1$, we get $W_1=W_2$, hence $\T^2_{i,k | j,k}=\{W_0\}$ and we are done. Assume now that $(W_0,W_1)$ is terminal. Of course, we have 
  $W_1\in \T^2_{i,k | j,k}$.  If $\T^2_{i,k}=\{W_0\}$, then $\T^2_{j,k}=\{W_1\}$ and we are done; if $\#\T^2_{i,k}\ge2$, then $W_1\subsetneq W_2$, by the minimal property of $W_1$. Iterating the reasoning, we deduce that
  \[
  \T^2_{i,k | j,k}=\{W_0,W_1,\dots, W_m\}, \text{ for some } m\ge0,
  \]
   where $W_{t-1}\subsetneq W_t$ and $(W_{t-1},W_t)$ is terminal, for each $t\in\{1,\dots, m\}$.

 We turn now to the proof of the claim for $s=3$. Arguing as we did for $s=2$, one can prove the claim when (i) holds for $s=3$ (in this case, we can use that \eqref{diff-rel} holds and that $\T^2_{i,k | j,k}$ is empty). Assume that $\T^3_{i,k | j,k}$ is nonempty and (i) does not holds for $s=3$. By the very definition of $\T^3_{i,k}$ and $\T^3_{j,k}$, we see that  there are tails $W_{t_0}\in \T^2_{i,k | j,k}$ and $W'_0\in\T^3_{i,k | j,k}$ such that $(W_{t_0}, W'_0)$ is terminal.
   Since $(W_{t-1},W_t)$ is terminal for each $t\in\{1,\dots,m\}$ and since the set of terminal points of any tail in $\T^3_{i,k}\cup\T^3_{j,k}$ does not contain $C_i\cap C_j$, there is a unique node of $C$ that can be terminal for both $W_{t_0}$ and $W'_0$, and this point belongs to $\term_{W_m}$; this proves the uniqueness of $W'_0$ and that $W_{t_0}=W_m$. We may assume that $W'_0\in\T^3_{i,k}$ (and hence $W_m\in\T^2_{j,k}$). Let $\{W'_t\}_{t<0}$ be the set of tails of $\T^3_{i,k}$ that are strictly contained in $W'_0$, with $W'_{t-1}\prec  W'_t$, for each $t\le0$. Since any tail in $\T^3_{j,k}$ contains $C_i\cup C_j$, we see that $\{W'_t\}_{t<0}$ is also the set of tails of $\T^3_{j,k}$ that are strictly contained in $W'_0$. Let 
   \[
   \{W'_{2t}\}_{t\ge0} \; \text{ and } \; \{W'_{2t+1}\}_{t\ge0}
   \]
    be respectively the remaining set of tails of $\T^3_{i,k}$ and $\T^3_{j,k}$. With the unique exception of $W'_0$, all the tails of these two sets are $(\T^2_{i,k}\cup\T^2_{j,k})$-free. If either $\T^3_{j,k}$ is empty or $W'_{-1}$ is the maximal tail of $\T^3_{j,k}$, then $W'_0$ is the maximal tail of $\T^3_{i,k}$ and we are done. If not, since $W'_{-1}\prec W'_1$, we have $W'_0\subsetneq W'_1$, by the minimal property defining $W'_0$. We can now proceed as in the case $s=2$ and deduce that 
    \[
    \T^3_{i,k | j,k}=\{W'_0,W'_1,\dots,W'_n\}, \text{ for some } n\ge0,
    \]
     where $W'_{t-1}\subsetneq W'_t$ and $(W'_{t-1},W'_t)$ is terminal, for each $t\in\{1,\dots,n\}$. The proof of the claim is complete.
   
We prove now that, for each $s\in\{2,3\}$, the set $\T^s_{i,k | j,k}$ is nonempty if and only if exactly one between (i) and (ii) holds. Using \eqref{diff-rel}, we see that (i) and (ii) do not hold at the same time for $s=3$. We proved the 'only if' part during the proof of the claim. The 'if' part is clear.

Finally, assume that $\T^s_{i,k | j,k}$ is non empty for some $s\in\{2,3\}$. Using (i) and (ii) we see that there is a 2-tail or a 3-tail of $\T_{i,k}\cup \T_{j,k}$ admitting the nodes of $C_i\cap C_j$ as terminal points. Since $\T^1_{i,k | j,k}$ consists at most of a tail with $C_i\cap C_j$ as terminal point, by \eqref{diff-rel} we see that $\T^1_{i,k | j,k}$ must be empty.
\end{proof}

\begin{Rem}\label{Rem1}
For $s$ in $\{1,2,3\}$, any two tails of $\T^s_{i,k}$ and $\T^s_{j,k}$ contain each other. It is clear if one of the two tails is in $\T^s_{i,k}\cap \T^s_{j,k}$, otherwise we use Proposition \ref{Diff}.
\end{Rem}

\begin{Rem}\label{Rem2}
Suppose that $\T^2_{i,k|j,k}$ is nonempty.
Write $\T^2_{i,k | j,k}=\{W_0,\dots,W_m\}$, where $W_{t-1}\subsetneq W_t$ and $(W_{t-1},W_t)$ is terminal. Let $S_1\in\term_{W_0}$ and $S_2\in\term_{W_m}$ be such that $S_1\not\in \term_{W_2}$ and $S_2\not\in \term_{W_{n-1}}$.
We call $S_1$ and $S_2$ \emph{the difference nodes of $\T^2_{i,k | j,k}$}. Up to switching $S_1$ and $S_2$, we have 
\begin{equation}\label{S1S2}
S_1\in C_i\cap C_j \; \text{ and } \; S_2\in \term_{W'_0},
\end{equation}
 where $W'_0$ is the minimal tail of $\T^3_{i,k | j,k}$ (if nonempty). 
For every $t$ in $\{1,\dots,n-1\}$, we have
\begin{equation}\label{term-2}
\term_{W_t}\subseteq \cup_{W\in\T^2_{i,k}}\term_W \; \text{ and } \; \term_{W_t}\subseteq \cup_{W\in\T^2_{j,k}}\term_W.
\end{equation}
 The two inclusions appearing in \eqref{term-2} do not hold at the same time for $W_0$ and $W_m$: The points for which they fail to hold are exactly the difference nodes of $\T^2_{i,k | j,k}$.
\end{Rem}

\subsection{Admissibility conditions}\label{sec5.2}

Throughout the section, we keep the notation of Sections \ref{Not2} and \ref{Abel-maps}. In particular, recall the sheaf $\L_\psi$ introduced in \eqref{Lpsi}.

We say that $\L_\psi$ is \emph{$\psi$-admissible}, 
if the restriction of $\L_\psi$ to every chain of $\psi$-exceptional components has degree $-1$, $0$ or $1$.
The notion of admissibility is employed in Theorem \ref{main2} to show that $\psi_*\L_\psi$ is a torsion-free rank-1 sheaf. For more details on admissible invertible sheaf, we will refer to \cite[Section 5]{CEP}.

\begin{Lem}\label{adm-cond}
The invertible sheaf  $\L_\psi$ is $\psi$-admissible if and only if the following conditions hold, for every $R_1$ in $C_{\gamma_1}\cap C_{\gamma'_1}$ and $R_2$ in $C_{\gamma_2}\cap C_{\gamma'_2}$, where $(\gamma_1,\gamma'_1)$ and $(\gamma_2,\gamma'_2)$ are in $\{1,\dots,p\}^2$, and  for every $(\alpha,\alpha')$ in $\{\gamma_1,\gamma'_1\}^2$ and  $(\beta,\beta')$ in $\{\gamma_2,\gamma'_2\}^2$. 
\begin{itemize}
\item
If $S\not\in\{R_1,R_2\}$,  for some  $S\in C_m\cap C_n$ and $(m,n)\in\{1,\dots,p\}^2$ with $m\ne n$, and if 
$C_{\alpha,\beta,\phi}\cap C_{\alpha',\beta',\phi}\ne\emptyset$ in $\wt \C^2$, then 
\begin{equation}\label{a}
 |\delta(\alpha,\beta,m,n)-\delta(\alpha',\beta',m,n)|\le 1.
\end{equation}
 \item
 If $R_1\ne R_2$, $\{\alpha,\alpha'\}=\{\gamma_1,\gamma'_1\}$ and $\{\beta,\beta'\}=
 \{\gamma_2,\gamma'_2\}$, then
 \begin{equation}\label{b}
  |\delta(\alpha,\beta,\alpha,\alpha')-\delta(\alpha,\beta',\alpha,\alpha')|\le 1;
  \end{equation} 
  \begin{equation}\label{c}
  |\delta(\alpha,\beta,\beta,\beta')-\delta(\alpha',\beta,\beta,\beta')|\le 1;
  \end{equation}
  \begin{equation}\label{d}
  |\delta(\alpha,\beta,\alpha,\alpha')-\delta(\alpha',\beta,\alpha,\alpha')-1|\le 1;
  \end{equation}
    \begin{equation}\label{e}
    |\delta(\alpha,\beta,\beta,\beta')-\delta(\alpha,\beta',\beta,\beta')-1|\le 1;
    \end{equation}
\begin{equation}\label{f}
|\delta(\alpha,\beta,\alpha,\alpha')-\delta(\alpha',\beta',\alpha,\alpha')-1|\le 1,   \text{ if } C_{\alpha,\beta,\phi}\cap C_{\alpha',\beta',\phi}\ne\emptyset;
\end{equation}
\begin{equation}\label{g}
|\delta(\alpha,\beta,\beta,\beta')-\delta(\alpha',\beta',\beta,\beta')-1|\le 1,   \text{ if } C_{\alpha,\beta,\phi}\cap C_{\alpha',\beta',\phi}\ne\emptyset.
\end{equation}
\item
If $R_1=R_2$ and $\{\alpha,\alpha'\}=\{\gamma_1,\gamma'_1\}$, then 
      \begin{equation}\label{h}
 |\delta(\alpha,\alpha,\alpha,\alpha')-\delta(\alpha,\alpha',\alpha,\alpha')-1|\le 1.
\end{equation}
\end{itemize}
\end{Lem}

\begin{proof}
The proof follows from \cite[Theorem 6.5]{CEP}, using the relations \eqref{delta-prop}.
\end{proof}

\begin{Lem}\label{adia-modif}
 Let $(i,j,k)$ be in $\{1,\dots,p\}^3$, where $C_i\cap C_j$ is nonempty. Then there is a possibly empty subcurve $Y$ of $C$ with $C_i\subseteq Y$ and $C_j\subseteq Y^c$, and such that  
 \[
 \O_\C(Z_{i,k})|_C\simeq \O_\C(Z_{j,k}-Y)|_C.
 \] 
\end{Lem}

\begin{proof}
First, we claim that if there is a nonempty subcurve $Y$ of $C$ such that 
\[
\O_\C(Z_{i,k}-Z_{j,k})|_C\simeq \O_\C(-Y)|_C,
\]
then $C_i\subseteq Y$ and $C_j\subseteq Y^c$. Indeed, for each $m\in\{1,\dots,p\}$ fix a smooth point $Q_m$ of $C$ lying on $C_m$. For each pair $(m,n)\in\{1,\dots,p\}^2$ consider the $\sigma(0)$-quasistable invertible sheaf $N_{m,n}$ on $C$ defined as
\[
N_{m,n}:=\O_C(2\sigma(0)-Q_n-Q_m)\otimes\O_\C(Z_{n,m})|_C.
\]
We have $N_{i,k}=N'\otimes\O_\C(-Y)|_C$, where 
\[
N':=\O_C(2\sigma(0)-Q_i-Q_k)\otimes \O_\C(Z_{j,k})|_C.
\]
Notice that
\[
\beta_{N'}(W)=
\begin{cases} 
\begin{array}{ll}
\beta_{N_{j,k}}(W) & \text{ if either } C_i\cup C_j\subseteq W \text{ or } C_i\cup C_j\subseteq W^c; \\
\beta_{N_{j,k}}(W)+1 & \text{ if } C_i\subseteq W^c \text{ and } C_j\subseteq W; \\
\beta_{N_{j,k}}(W)-1 & \text{ if } C_i\subseteq W \text{ and } C_j\subseteq W^c. 
\end{array}
\end{cases}
\]
Since $Y$ is nonempty, $N'$ is not $\sigma(0)$-quasistable, and hence there is a subcurve $W$ of $C$ such that $\beta_{N'}(W)\le 0$ (where the inequality is strict if $P\not\in W$), with $C_i\subseteq W$ and $C_j\subseteq W^c$. Using that $N_{i,k}=N'\otimes\O_\C(-Y)|_C$ and that $N_{i,k}$ is $\sigma(0)$-quasistable, it follows that $C_i\subseteq Y\wedge W$ and $C_j\subseteq Y^c\wedge W^c$. The proof of the claim is complete.

Suppose that $\T^1_{i,k | j,k}$ is nonempty. Then $\T^1_{i,k | j,k}$ consists of one tail $Y$ terminating in $C_i\cap C_j$. By Lemma \ref{Diff} we have that $\T^s_{i,k | j,k}$ is empty, for each $s\in\{2,3\}$, hence  
\[
\O_\C(\pm(Z_{i,k}-Z_{j,k}))|_C\simeq \O_\C(-Y)|_C.
\]
and hence we are done. Thus, we may assume that $\T^1_{i,k | j, k}$ is empty. 
Write 
\[
\T^2_{i,k | j,k}=\{W_0,\dots,W_m\} \; \text{ and } \; \T^3_{i,k|j,k}=\{W'_0,\dots,W'_n\}
\]
 where $W_{t-1}\subsetneq W_t$, $W'_{t-1}\subsetneq W'_t$, and where $(W_{t-1},W_t)$ and $(W'_{t-1},W'_t)$ are terminal (see Proposition \ref{Diff}). We have 
\[
\O_\C(\pm (Z_{i,k}-Z_{j,k}))|_C\simeq\O_\C\left(\varepsilon\sum_{t=0}^{m} (-1)^{m-t} W_t + \sum_{t=0}^{n} (-1)^{n-t} W'_t \right)|_C,
\]
for some $\varepsilon\in\{-1,1\}$. Since $W_{t-1}\subsetneq W_t$ and $W'_{t-1}\subseteq W'_t$, there are subcurves $X$ and $X'$ of $C$ such that 
\[
  X=\sum_{t=0}^{m} (-1)^{m-t} W_t \; \text{ and } \; X'=\sum_{t=0}^n (-1)^{n-t} W'_t.
\]
If one between $\T^2_{i,k|j,k}$ and $\T^3_{i,k|j,k}$ is empty, then we are done. If not, then it follows from Proposition \ref{Diff} that $W'_0$ is $W_m$-terminal, hence we have $\varepsilon=(-1)^{n+1}$ and
\begin{equation}\label{pm}
\O_\C(\pm (Z_{i,k}-Z_{j,k}))|_C\simeq \O_\C\left((-1)^{n+1}X + \sum_{t=0}^{n-t} (-1)^{n-t} W'_t
\right)|_C.
\end{equation}
Notice that $X$ is a 2-tail admitting the difference nodes $S_1$ and $S_2$ of $\T^2_{i,k|j,k}$ as terminal points. Since $W'_0$ crosses the nodes of the set $C_i\cap C_j$, it follows from \eqref{S1S2} that $S_1$ and $S_2$ are contained in $W'_0$. Since $\sigma(0)$ is not contained in $X$, it follows from  Lemma \ref{free-perf} that $X\subseteq W'_0$. We deduce that the right hand side of Equation \eqref{pm} is isomorphic to $\O_\C(-Y)$, for some subcurve $Y$ of $C$.
\end{proof}

\begin{Thm}\label{main2}
Let $\pi\:\C\ra B$ be a smoothing  of a nodal curve $C$.
Let $\phi\:\wt\C^2\ra \C^2$ and $\psi\:\wt{\C}^3\ra \wt \C^2\times_B\C$ be  good partial desingularizations. Then the invertible sheaf $\L_\psi$ is  $\psi$-admissible. Moreover, 
 $\psi_*\L_\psi$ is a relatively torsion-free, rank-1 sheaf on $\wt\C^2\times_B \C/\wt\C^2$ of relative degree 0, whose formation commutes with base change.
\end{Thm}

\begin{proof} 
Let us check the conditions of Lemma \ref{adm-cond}.  
Let $R_1\in C_{\gamma_1}\cap C_{\gamma'_1}$ and $R_2\in C_{\gamma_2}\cap  C_{\gamma'_2}$, where $(\gamma_1,\gamma'_1),(\gamma_2,\gamma'_2)\in\{1,\dots,p\}^2$. Let 
 $(\alpha,\alpha')\in\{\gamma_1,\gamma'_1\}^2$ and $(\beta,\beta')\in\{\gamma_2,\gamma'_2\}^2$ such that $C_{\alpha,\beta,\phi}\cap C_{\alpha',\beta',\phi}\ne\emptyset$. 
Let $S\in C_m\cap C_n$, for $(m,n)\in\{1,\dots,p\}^2$ and $m\ne n$. We will often use that, for every $(a,b)\in\{1,\dots,p\}^2$, we have 
\begin{equation}\label{deltasum}
\delta(a,b,m,n)=\underset{T\in \term_W}{\sum_{W\in \T_{a,b}}}\epsilon_W,
\end{equation}
where $\epsilon_W=1$ (respectively $\epsilon_W=-1$) if $C_m\subseteq W$ (respectively $C_n\subseteq W$).

Assume that $S\not\in \{R_1,R_2\}$. 
By Lemma \ref{adia-modif},  we see that (\ref{a}) holds if either $\alpha=\alpha'$ or 
$\beta=\beta'$, so we can reduce to the case in which 
\[
\{\alpha,\alpha'\}=\{\gamma_1,\gamma'_1\}\;\text{ and } \;
\{\beta,\beta'\}=\{\gamma_2,\gamma'_2\}.
\]
If $R_1=R_2$, then 
$\{\alpha,\beta\}=\{\gamma_1,\gamma'_1\}=\{\alpha',\beta'\}$,
  (recall that the first blowup performed by $\phi$ is along the diagonal of $\wt\C^2$), hence  (\ref{a}) holds. 
  Moreover, (\ref{a}) holds if $S$ is a separating node of $C$. Indeed, in this case, let  $W$ be a 1-tail of $C$ such that $\term_W=\{S\}$. Since $C_{\gamma_1}\cap C_{\gamma'_1}\ne\emptyset$ and $S\ne R_1$, it follows that $C_{\alpha}\cup C_{\alpha'}$ is contained either in $W$ or  in $W^c$. Similarly,  $C_{\beta}\cup C_{\beta'}$ is contained either in $W$ or  in $W^c$. Since $W$ is $Z$-free, for every tail $Z$ with $k_Z\ge2$, it follows that 
  \[
  \delta(\alpha,\beta,m,n)=\delta(\alpha',\beta',m,n).
  \]
  
Therefore, to prove that \eqref{a} holds, we may assume that $R_1\ne R_2$ and that the subsets of tails of  $\T_{\alpha,\beta}$ and of $\T_{\alpha',\beta'}$ admitting $S$ as terminal point consist  respectively of a tail $W_1$ and of a tail $W_2$ (that are not 1-tails). We may also assume that $W_1\cup W_2$ crosses $S$, otherwise $\delta(\alpha,\beta,m,n)=\delta(\alpha',\beta',m,n)$ and we are done. Notice that
 $(W_1,W_2)$ is not perfect, otherwise $W_1\subseteq W_2^c$, hence 
 \[
 C_{\alpha}\cup C_{\beta}\subseteq W_2^c \; \text{ and } \;  C_{\alpha'}\cup C_{\beta'}\subseteq W_1^c,
 \] 
 from which we would get 
$\term_{W_1}=\term_{W_2}=\{R_1,R_2,S\}$ and $W_2=W_1^c$, which is not possible. Since $(W_1,W_2)$ is not perfect, by  Lemma \ref{free-perf} we can write 
\[
\term_{W_1}=\{S,U_1,U'_1\} \; \text{ and } \;\term_{W_2}=\{S,U_2,U'_2\},
\] 
with 
  $U_1\not\in W_2$, $U'_1\not\in W^c_2$, $U_2\not\in W_1$, $U'_2\not\in W^c_1$ and $\{U_1,U_2\}\cap \{U'_1,U'_2\}=\emptyset$. Notice that $W_1\cup W_2$ and $W_1\wedge W_2$ are 2-tails, with 
  \[
  \term_{W_1\cup W_2}=\{U_1,U_2\} \; \text{ and } \; \term_{W_1\wedge W_2}=\{U'_1,U'_2\}.
  \]

We claim that $W_1\cup W_2$ and $W_1\wedge W_2$ are distinct tails of $\T^2_{\alpha,\beta|\alpha,\beta'}$ and $W_1\wedge W_2$ crosses $R_1$ and $R_2$. Indeed, by Lemma \ref{maxT2} there are $(W_1\cup W_2)$-terminal tails 
\[
X_1\in \T^2_{\alpha,\beta}, \; 
X_2\in \T^2_{\alpha',\beta'}, \; 
X_3\in \T^2_{\alpha,\beta'}, \;
X_4\in \T^2_{\alpha',\beta},
\] 
all contained in $W_1\cup W_2$. Of course, $U_1\not\in \term_{X_1}$ and $U_2\not\in \term_{X_2}$, hence 
\[
U_1\in \term_{X_2},\; U_1\not\in X_1,\; U_2\in \term_{X_1}  \text{ and } U_2\not\in X_2.
\]
 Fix $i\in\{3,4\}$. We have $U_1\in \term_{X_i}$, otherwise $U_1\not\in X_i$ and $U_2\in \term_{X_i}$, which implies that $X_2$ and $X_i$ do not contain each other, contradicting Remark \ref{Rem1}. Analogously, using $X_1$ and $X_i$, it follows that $U_2\in\term_{X_i}$. 
We deduce that 
\[
W_1\cup W_2=X_3=X_4\in \T^2_{\alpha,\beta|\alpha,\beta'}\cap \T^2_{\alpha,\beta|\alpha',\beta}\cap\T^2_{\alpha,\beta'|\alpha',\beta'}\cap \T^2_{\alpha',\beta|\alpha',\beta'}.
\]
 Thus, combining \eqref{diff-rel} with the condition (i) of Proposition \ref{Diff}, we see that 
$C_{\alpha}$, $C_{\alpha'}$, $C_{\beta}$ and $C_{\beta'}$ are contained in $W_1\wedge W_2$. In particular, we can repeat the above argument replacing $W_1\cup W_2$ with $W_1\wedge W_2$ to get  $W_1\wedge W_2\in \T^2_{\alpha,\beta|\alpha,\beta'}$.  Finally, the tails $W_1\cup W_2$ and $W_1\wedge W_2$ are distinct, since 
\[
S\in (W_1\cup W_2)\setminus (W_1\wedge W_2).
\] 

It follows from the claim and Proposition \ref{Diff} that $W_1\wedge W_2$ is a tail of 
$\T^2_{\alpha,\beta|\alpha,\beta'}$ which is neither minimal nor maximal. 
Using \eqref{term-2}, we get 
\[
\term_{W_1\wedge W_2}\subseteq \cup_{W\in \T^2_{\alpha,\beta}}\term_W,
\]
which yields a contradiction because $U'_1\in\term_{W_1}\cap \term_{W_1\wedge W_2}$ and $W_1\in \T^3_{\alpha,\beta}$. 

We see that \eqref{a} always holds. From now on, we may assume that 
\[
\{\alpha,\alpha'\}=\{\gamma_1,\gamma'_1\}\; \text{ and } \; \{\beta,\beta'\}=\{\gamma_2,\gamma'_2\}.
\]
Notice that (\ref{b}) and (\ref{c}) hold by Lemma \ref{adia-modif}. Using \eqref{delta-prop} and again Lemma \ref{adia-modif}, we see that to show \eqref{d} and \eqref{h}, it suffices to prove that
\begin{equation}\label{0}
\delta(\alpha,\beta,\alpha,\alpha')-\delta(\alpha',\beta,\alpha,\alpha')\ge0.
\end{equation}
If $\delta(\alpha,\beta,\alpha,\alpha')<0$, then there is a tail $W$ in $\T_{\alpha,\beta}$ such that $C_{\alpha'}\subseteq W$ and such that $\term_W$ contains $C_{\alpha}\cap C_{\alpha'}$ (hence $C_{\alpha}$ is not contained in $W$). 
The unique possibility is that $W\in \T^1_{\beta}$, with $\alpha\ne \beta$, and that $C_{\alpha}\cap C_{\alpha'}$ is a separating node. We obtain that 
$\delta(\alpha,\beta,\alpha,\alpha')=-1$ and $\delta(\alpha',\beta,\alpha',\alpha)=2$. Similarly, if 
$\delta(\alpha',\beta,\alpha',\alpha)<0$, we obtain that 
 $\delta(\alpha',\beta,\alpha',\alpha)=-1$ and $\delta(\alpha,\beta,\alpha,\alpha')=2$. 
Since
\[
\delta(\alpha,\beta,\alpha,\alpha')-\delta(\alpha',\beta,\alpha,\alpha')=
\delta(\alpha,\beta,\alpha,\alpha')+\delta(\alpha',\beta,\alpha',\alpha),
\]
we see that \eqref{0} holds. One can prove \eqref{e} in a similar fashion.

Finally, again by Lemma \ref{adia-modif}, we have
\[
\delta(\alpha,\beta,\alpha,\alpha')-\delta(\alpha,\beta',\alpha,\alpha')+
\delta(\alpha,\beta',\alpha,\alpha')-\delta(\alpha',\beta',\alpha,\alpha') \le 2.
\]
Therefore, we see that \eqref{f} holds once we show that  
\[
\delta(\alpha,\beta,\alpha,\alpha')-\delta(\alpha',\beta',\alpha,\alpha')\ge0,
\]
and one can prove that this is true with a reasoning similar to the one employed to prove \eqref{0}.  One can prove \eqref{g} in a similar fashion.

The last sentence is consequence of \cite[Proposition 5.2]{CEP}
\end{proof}

\section{The analysis locally around the diagonal}\label{diagonal}
\noindent
We will keep the notation of Sections \ref{Not2} and \ref{Abel-maps}. Consider the map
\[
\alpha^2\circ\phi\col\wt\C^2 \dashrightarrow  \ol{\mathcal J^f}
\]
defined in \eqref{alphaPE} and induced by the relative sheaf $\psi_*\L^{\P,\E}_\psi$ on $\wt\C^2\times_B\C/\wt\C^2$. The goal of this section is to provide conditions ensuring that $\alpha^2\circ\phi$ is defined locally around the locus of $\wt\C^2$ lying over the points $(R,R)$ of $\wt\C^2$, for $R\in \N(C)$. The conditions are always satisfied by the Abel--Jacobi map of $\C/B$ (see Theorem \ref{main3}).

\begin{Lem}\label{compadm} 
Let $C$ be a curve and $\mu\col C(d)\ra C$ be the contraction map, for a positive integer $d$. Let $\E$ be a polarization on $C$ and $P$ be a smooth point of $C(d)$.
Let $\L$ and $\M$ be invertible sheaves on $C(d)$. Let $\mathcal R$ be the 
set of $\mu$-exceptional components. Assume that there is a smoothing $\X\ra B$ of $C(d)$ such that 
\[
\M\otimes\L^{-1}\simeq \O_\X\left(\underset{E\in\mathcal R}{\sum} c_E E \right)|_{C(d)}, \,\,\,  c_E\in\mathbb{Z}.
\]
If $\L$ is $\mu$-admissible and $\M$ is $P$-quasistable with respect to $\mu^*\E$, then $\mu_*\L$ is $\mu(P)$-quasistable with respect to $\E$. 
\end{Lem}

\begin{proof}
Use \cite[Propositions 5.2 and 5.3]{CEP}, observing that an invertible sheaf on $C$ which is $P$-quasistable with respect to $\mu^*\E$ is $\mu$-admissible.
\end{proof}

Recall the identification of $X_A$ with $C(2)$ and of $\psi(X_A)$ with $C$, and of the map $\mu\: C(2)\ra C$ with $\psi|_{X_A}\:X_A\ra \psi(X_A)$. Recall that $\xi$ is the composed map $\xi\:\wt\C^3\stackrel{\psi}{\ra}\wt\C^2\times_B \C\ra\C$, where the second map is the projection onto the last factor.

\begin{Lem}\label{mu*}
Let $A$ be a distinguished point of $\wt\C^2$ in the intersection of the divisors $C_{\gamma_1,\gamma_2,\phi}$, $C_{\gamma_1,\gamma'_2,\phi}$ and $C_{\gamma'_1,\gamma_2,\phi}$ of $\wt\C^2$, for $(\gamma_1,\gamma'_1)$ and $(\gamma_2,\gamma'_2)$ in $\{1,\dots,p\}^2$.  Consider the Cartier divisor $D$ of $\wt\C^3$ given by
\[
D:=-\sum_{m=1}^p \alpha^{\P,\E}_{\gamma_1,\gamma_2,m}(C_{\gamma_1,\gamma_2,m,\psi}+C_{\gamma_1,\gamma'_2,m,\psi}+C_{\gamma'_1,\gamma_2,m,\psi}).
\]
Let $Q_{\gamma_1}$ and $Q_{\gamma_2}$ be smooth points of $\psi(X_A)$ lying on $C_{\gamma_1}$ and $C_{\gamma_2}$. 
Then $\O_{\C^3}(D)$ has degree $0$ on the $\psi|_{X_A}$-exceptional components and the invertible sheaf 
\[
(\psi_*\xi^*\P)|_{\psi(X_A)}(-Q_{\gamma_1}-Q_{\gamma_2})\otimes\psi_*\O_{\wt\C^3}(D)|_{\psi(X_A)}
\]
 is $\sigma(0)$-quasistable with respect to $\E|_{\psi(X_A)}$.
\end{Lem}

\begin{proof}
Fix a slice $\lambda\:B\to\wt\C^2$ through $A$, let $\W\ra B$ be the $\lambda$-smoothing of $X_A$ and $\theta\col \W\ra \wt\C^3$ be the induced map 
 (see Diagram \eqref{diagram}). 
Set $\alpha_m:=\alpha^{\P,\E}_{\gamma_1,\gamma_2,m}$.
 
Fix a node $S$ of $C$ such that $S\in C_m\cap C_n$, for $(m,n)\in\{1,\dots,p\}^2$ and $m\ne n$, and  set 
\[
X_{A,S}:=\wh C_m\cup \wh C_n\cup E_S.
\]
  By \cite[Table 1 of Lemma 3.2]{CEP}, there are $(i,k)$, $(i,l)$, $(j,k)$ in $\{1,\dots,p\}^2$ such that 
\[
\{(i,k),(i,l),(j,k)\}=\{(\gamma_1,\gamma_2),(\gamma_1,\gamma'_2),(\gamma'_1,\gamma_2)\}
\]
and satisfying the following properties
\[
\begin{array}{lll}
 C_{i,k,m,\psi}\cap X_{A,S}=\wh C_m\cup E_S & \text{ and } &   C_{i,k,n,\psi}\cap X_{A,S}=\wh C_n; \\
  C_{i,l, m, \psi}\cap X_{A,S}=\wh C_m\cup E_{S,m} &  \text{ and } &   C_{i,l, n, \psi}\cap X_{A,S}=\wh C_n\cup E_{S,n};\\
   C_{j,k, m,\psi}\cap X_{A,S}=\wh C_m & \text{ and } &C_{j,k,n,\psi}\cap X_{A,S}= \wh C_n\cup E_S. 
\end{array}
\]
Therefore, we obtain
\begin{equation}\label{D1D2}
\theta^*\O_{\wt\C^3}(D)|_{X_A}\simeq\O_{\W}(D_1+D_2)|_{X_A}
\end{equation}
where 
\begin{equation}\label{D1}
\begin{array}{lll}
D_1 := & -\alpha_m \wh C_m-\alpha_m (\wh C_m+E_{S,m})-\alpha_m(\wh C_m+E_S) \\
& -\alpha_n \wh C_n-\alpha_n (\wh C_n+E_{S,n})-\alpha_n(\wh C_n+E_S),\\
\end{array}
\end{equation}
and where $D_2$ is not supported on $X_{A,T}$.

It follows from \eqref{D1D2} and \eqref{D1} that $\O_\W(D_1+D_2)$ has degree $0$ on every $\psi$-exceptional component. Moreover, for each $m\in\{1,\dots,p\}$, we have
\[
\deg_{C_m}\O_\W(D_1+D_2)=\sum_{n\ne m} \#(C_m\cap C_n)\cdot (\alpha_m-\alpha_n).
\]
Therefore, $\psi_*\O_{\wt\C^3}(D)|_{\psi(X_A)}$ and $\O_\C(Z_{\gamma_1,\gamma_2}^{\P,\E})|_C$ are invertible sheaves on $C$ with the same degree on irreducible components of $C$, which proves the last sentence.
\end{proof}

\begin{Lem}\label{ext-red} 
If $\L^{\P,\E}_\psi\otimes\O_{X_A}$ is $\psi|_{X_A}$-admissible and its push-forward via $\psi|_{X_A}$ is $\sigma(0)$-quasistable with respect to $\E|_{\psi(X_A)}$, for a distinguished point $A$ of $\wt\C^2$, then the map $\alpha^2_{\P,\E}\circ\phi\col\wt\C^2 \dashrightarrow  \ol{\mathcal J^f}$ 
induced by the relative sheaf $\psi_*\L^{\P,\E}_\psi$ on $\wt\C^2\times_B\C/\wt\C^2$ is defined on an open subset of $\wt\C^2$ containing $A$. If this condition is satisfied for every distinguished point $A\in\phi^{-1}(R_1,R_2)$ of $\wt\C^2$, where $\{R_1,R_2\}\subseteq \N(C)$, then the map $\alpha^2\circ\phi$ is defined everywhere.
\end{Lem}

\begin{proof}
If $\L^{\P,\E}_\psi\otimes\O_{X_A}$ is $\psi|_{X_A}$-admissible, then its formation commutes with base change (see \cite[Proposition 5.2]{CEP}). Assume also that its push-forward via $\psi|_{X_A}$ is $\sigma(0)$-quasistable with respect to $\E|_{\psi(X_A)}$. Let $p_1\col\wt\C^2\times_B\C\ra \wt\C^2$   and $p_2\col\wt\C^2\times_B\C\ra \C$ be the projections onto the first and last factor and let $\sigma'$ be the section of $p_1$ obtained as pull-back of the section $\sigma$ of $\pi\col\C\ra B$. Since quasistability is an open property (see \cite[Proposition 34]{E01}), $\psi_*\L_\psi|_{p_1^{-1}(U)}$ is $\sigma'|_U$-quasistable with respect to $p_2^*\E$, for some open subset $U$ of $\wt\C^2$ containing $A$, and we conclude that $\alpha^2_{\P,\E}\circ \phi$ is defined on $U$. The second statement follows from \cite[Lemma 6.3]{CEP}.
\end{proof}

\begin{Thm}\label{main3}
Let $\pi\:\C\ra B$ be a smoothing  of a nodal curve $C$ with a section $\sigma$ through its smooth locus. Let $C_1,\dots,C_p$ be the irreducible components of $C$. Let $\P$ be a relative invertible sheaf on $\C/B$ of relative degree $2+f$ and $\E$ be a polarization of degree $f$ on $\C/B$. Let $\phi\:\wt\C^2\ra \C^2$ and $\psi\: \wt\C^3\ra\wt\C^2\times_B \C$ be good partial desingularizations.  
Let $R$ be a node of $C$ such that $R\in C_{\gamma}\cap C_{\gamma'}$, for $(\gamma,\gamma')\in \{1,\dots,p\}^2$ and $\gamma\ne \gamma'$ and 
$A\in\phi^{-1}(R,R)$ be a distinguished point of $\wt\C^2$. Assume that the following conditions hold
\begin{itemize}
\item[(i)]
there is a possibly empty subcurve $Y$ of $C$ with $C_{\gamma}\subseteq Y$ and $C_{\gamma'}\subseteq Y^c$,  and such that
\[
\O_\C(Z_{\gamma,\gamma}^{\P,\E})|_C\simeq \O_\C(Z_{\gamma,\gamma'}^{\P,\E}-Y)|_C;
\]
\item[(ii)]
 $\L^{\P,\E}_\psi\otimes \O_{X_A}$ is $\psi|_{X_A}$-admissible.
\end{itemize}
Then the push-forward of $\L^{\P,\E}_\psi\otimes \O_{X_A}$ via $\psi|_{X_A}$ 
 is $\sigma(0)$-quasistable with respect to $\E|_{\psi(X_A)}$. Hence, the map $\alpha^2\circ\phi\col\wt\C^2\dashrightarrow \ol{\mathcal J^f}$ induced by $\psi_*\L^{\P,\E}_\psi$ is defined on an open subset of $\wt\C^2$ containing $A$. 
The conditions (i) and (ii) are satisfied when $\deg_{C_i}\P=\deg_{C_i}\O_\C(2\sigma(B))$ for every $i\in\{1,\dots,p\}$, and $\E$ is a canonical polarization. 
 \end{Thm}

\begin{proof}
Recall that there is an identification of $X_A$ with $C(2)$ and of $\psi(X_A)$ with $C$. Abusing notation, we use the same symbol for both $\psi$ and $\psi|_{X_A}$. We also use $\E$ to denote $\E|_{\psi(X_A)}$. 
Fix a slice $\lambda\:B\to\wt\C^2$ through $A$, let $\W\ra B$ be the $\lambda$-smoothing of $X_A$ and $\theta\col \W\ra \wt\C^3$ be the induced map 
 (see Diagram \eqref{diagram}). 
 By \eqref{diag-type}, there is a divisor $D_1$ of $\W$ supported on $\psi$-exceptional components such that 
 \begin{equation}\label{strong2}
  \theta^*\left(\I_{\wt\Delta_1|\wt\C^3}\otimes \I_{\wt\Delta_2|\wt\C^3}\right)\otimes \O_{X_A}\simeq \O_\W(D_1)|_{X_A} \otimes N,
  \end{equation}
  where $N$ is a line bundle on $X_A$ with degree $-1$ on $C_{\gamma}$ and $E_{R,\gamma'}$, and degree $0$ on the remaining components of $X_A$.
Set
\[
M:=\theta^*\O_{\wt\C^3}\left(-\sum_{m=1}^p \alpha^{\P,\E}_{\gamma,\gamma',m}(C_{\gamma,\gamma,m,\psi}+C_{\gamma,\gamma',m,\psi}+C_{\gamma',\gamma,m,\psi})\right)
\]
and
\[
M':=\theta^*\O_{\wt\C^3} \left(-\underset{C_m\subseteq Y}{\sum_{m\in\{1,\dots,p\}}}C_{\gamma,\gamma,m,\psi}\right).
\]
Recall the definition of canonical lifting introduced in Section \ref{Not1} and set
\[
\wh Y:=L_0^Y\cup E_R.
\]
 There is a divisor $D_2$ on $\W/B$ supported on $\psi$-exceptional components such that 
 \[
 M'\simeq\O_\W(D_2)\otimes \O_\W(-\wh Y)
 \]
 Set $D:=D_1+D_2$. Since $\O_\C(Z_{\gamma,\gamma'}^{\P,\E})|_C\simeq \O_\C(Z_{\gamma',\gamma}^{\P,\E})|_C$, by Condition (i) we get
 \begin{equation}\label{thetaL}
\theta^*\L^{\P,\E}_\psi \otimes \O_{X_A} \simeq (\O_\W(D)\otimes \theta^*\xi^*\P\otimes M(-\wh Y))|_{X_A}\otimes N.
\end{equation}
Set $\wh\E:=\xi^*\E|_{X_A}$. 
Since $\L^{\P,\E}_\psi\otimes \O_{X_A}$ is $\psi$-admissible by Condition (ii), it turns out from Lemma \ref{compadm} and from Equation \eqref{thetaL} that to conclude our proof, it suffices to show that the invertible sheaf
\[
\wh N:=(\theta^*\xi^*\P\otimes M(-\wh Y))|_{X_A}\otimes N
\] 
on $X_A$ is $\psi^{-1}(\sigma(0))$-quasistable with respect to $\wh\E$.

Let $V$ be a connected subcurve of $X_A$ with connected complementary subcurve. Consider the (possibly empty) subcurves of $C$ 
\[
V':=\psi(V)\wedge Y \; \text{ and } \; V'':=\psi(V)\wedge Y^c.
\]
We have $\psi(V)=V'\cup V''$. Notice that $\deg_{E_{R,\gamma}} N(-\wh Y)=\deg_{E_{R,\gamma'}}N(-\wh Y)=0$. For every node $S$ of $C$ such that $R\ne S\in C_m\cap C_n\cap Y\cap Y^c$, where $m\ne n$, we have
  \begin{equation}\label{ESh}
 \deg_{E_{S,m}}N(-\wh Y)=
\begin{cases} 
\begin{array}{ll}
-1 & \text{ if } C_m \subseteq Y;\\
0 & \text{ otherwise. }
 \end{array}
 \end{cases}
 \end{equation}
Thus, the degree of $\wh N$ on every chain of $\psi$-exceptional components is $-1$ or $0$ (by Lemma \ref{mu*}, $\theta^*\xi^*\P\otimes M$ has degree $0$ on $\psi$-exceptional components). Therefore, if $V$ is contracted by $\psi$, then $\wh N$ is $\psi^{-1}(\sigma(0))$-quasistable over $V$ with respect to $\wh\E$. We may assume that $V$ and $V^c$ are not contracted by $\psi$. In particular, if we set 
\[
E_V:=\bigcup_{S\in V'\cap V''} E_S,
\]
then $E_V\subseteq V$.
We may also assume that $C_{\gamma'}\subseteq V^c$.  

Since $C_{\gamma'}\subseteq V^c$, the subcurve $\psi(V)$ does not cross $R$, and hence $R\not\in V'\cap V''$. Using \eqref{ESh}, 
 we see that $\deg_{E_V}\wh N=-\# V\cap V''$; since $E_V$ is the union of $\#(V'\cap V'')$ chains of rational curves, we get $\beta_{\wh N}(E_V, \wh\E)=0$. 
Consider the subcurves of $X_A$
\[
V_1:=L_2^{V'}\cup L_2^{V''} \; \text{ and } \;  V_2:=L_0^{V'}\cup L_0^{V''}\cup E_V.
\]
Using that $E_V\subset V$, we have $g_{V_1}=g_{V}=g_{V_2}$ and hence
\[\beta_{\wh N}(V_1, \wh\E)\le\beta_{\wh N}(V, \wh\E)\le \beta_{\wh N}(V_2,  \wh\E).
\]
Recall that there is an identification of $\psi(X_A)$ with $C$. Let $Q_{\gamma}$ and $Q_{\gamma'}$ be smooth points of $C$ lying on $C_{\gamma}$ and $C_{\gamma'}$ and define the invertible sheaves $N'$ and $N''$ on $C$ 
\[
N':=(\psi_*\theta^*\P)|_C(-2Q_{\gamma})\otimes \psi_*(M|_{X_A})\otimes \O_\C(-Y)|_C; \\ 
\]
\[
N'':=(\psi_*\theta^*\P)|_C(-Q_{\gamma}-Q_{\gamma'}) \otimes  \psi_*(M|_{X_A}).
\]
It follows from Lemma \ref{mu*} and Condition (i) that $N'$ and $N''$ are $\sigma(0)$-quasistable invertible sheaves on $C$ with respect to $\E$. 
Since $g_{L_0^{V'}}=g_{L_2^{V'}}=g_{V'}$ and since the degree of $\xi^*\E$ on $\psi$-exceptional components is $0$, by construction we get
\[
\beta_{\wh N}(L_0^{V'},\wh\E)=\beta_{N'}(V', \E) \; \text{ and } \; \beta_{\wh N}(L_2^{V'},\wh\E)=\beta_{N''}(V',\E),
\]
and similarly also
\[
\beta_{\wh N}(L_0^{V''},\wh\E)=\beta_{N''}(V'',\E) \; \text{ and } \;\beta_{\wh N}(L_2^{V''}, \wh\E)=\beta_{N'}(V'',\E). 
\]
Notice that  $\beta_{\wh N}(L_2^{V'}\wedge L_2^{V''},\wh\E)=0$, because $L_2^{V'}\wedge L_2^{V''}=E_V$. Therefore, we get
\[
\begin{array}{lll}
\beta_{\wh N}(V,\wh\E) & \ge &  \beta_{\wh N}(V_1, \wh\E) \\
 & = & \beta_{\wh N}(L_2^{V'},\wh\E)+\beta_{\wh N}(L_2^{V''},\wh\E)-\beta_{\wh N}(L_2^{V'}\wedge L_2^{V''},\wh\E)\\
 & = & \beta_{N''}(V',\E)+\beta_{N'}(V'',\E) \\
 & \ge & 0
\end{array}
\]
where the last inequality is strict if $\sigma(0)\in V$.  We also get
\[
\begin{array}{lll}
\beta_{\wh N}(V,\wh\E) & \le &  \beta_{\wh N}(V_2,\wh\E) \\
 & = &  \beta_{\wh N}(L_0^{V'},\wh\E)+\beta_{\wh N}(L_0^{V''}\cup E_V,\wh\E)-\#(V'\cap V'')\\
 & = & \beta_{\wh N}(L_0^{V'},\wh\E)+\beta_{\wh N}(L_0^{V''},\wh\E)-2\,\#(V'\cap V'')\\
  & = & \beta_{N'}(V',\E)+\beta_{N''}(V'',\E) - 2\,\#(V'\cap V'') \\
  & \le & k_{V'}+k_{V''}-2\,\#(V'\cap V'') \\
  & = & k_{\psi(V)} \\
  & = & k_V
\end{array}
\]
where the last inequality is strict if $\sigma(0)\not\in V$ and the last equality holds because $Z$ is connected. 

The last two sentences follow from Lemmas \ref{ext-red}, \ref{adia-modif} and \ref{main2}.
\end{proof}

\begin{Rem}
We warn the reader that the proof of Theorem \ref{main3} can not easily modify to prove a similar result for more general distinguished points.
In the proof we strongly use the obvious fact that $\O_\C(Z^{\P,\E}_{\gamma,\gamma'}-Z^{\P,\E}_{\gamma',\gamma})|_C$ is trivial. If $R_1$ and $R_2$ are nodes of $C$ such that $R_1\in C_{\gamma_1}\cap C_{\gamma'_1}$ and $R_2\in C_{\gamma_2}\cap C_{\gamma'_2}$, in general 
$\O_\C(Z^{\P,\E}_{\gamma_1,\gamma'_2}-Z^{\P,\E}_{\gamma'_1,\gamma_2})|_C$ is not trivial. Thus, one should be able to describe such an invertible sheaf, possibly giving rise to difficult combinatorial issues. For this reason, in the next sections we will use a different approach.
 \end{Rem}

    \section{On some properties of distinguished points}
    \noindent    
Let $\pi\: \C\ra B$ be a smoothing of a nodal curve $C$ and 
$\phi\:\wt\C^2\ra \C^2$ be a good partial desingularization.          
In this section, we will introduce the notions of quasistable and synchronized distinguished points of $\wt\C^2$ via purely combinatorial conditions on $C$ and $\phi$. The key ingredient in the proof of Theorem \ref{main4} will be the following result: if a distinguished point is quasistable, then it is synchronized. The main goal of this section will be Propositions \ref{PropA} and \ref{PropB}, giving necessary conditions for the proof of this implication.

\subsection{More notation and terminology}\label{Not3} 

Keep the notation of Sections \ref{Not2} and \ref{Abel-maps}. As depicted in Figure 3, we will let $A$ be a distinguished point of $\wt\C^2$, with 
\[
A\in C_{\gamma_1,\gamma_2,\phi}\cap C_{\gamma_1,\gamma'_2,\phi}\cap  
 C_{\gamma'_1,\gamma_2,\phi},
\]
where $(\gamma_1,\gamma'_1)$ and $(\gamma_2,\gamma'_2)$ are in $\{1,\dots,p\}^2$. We will set
 \[
 \P_A:=\{(\gamma_1,\gamma_2),(\gamma_1,\gamma'_2),(\gamma'_1,\gamma_2)\}.
 \]
 Throughout the section, we will assume that 
 \[
 \phi(A)=(R_1,R_2), \text{ with } R_1,R_2\in \N(C) \; \text{ and } \;R_1\ne R_2.
 \]
Notice that $R_1\in C_{\gamma_1}\cap C_{\gamma'_1}$ and $R_2\in C_{\gamma_2}\cap C_{\gamma'_2}$. 
Recall the identification of $X_A$ with $C(2)$ and of $\psi(X_A)$ with $C$, and of the map $\mu\: C(2)\ra C$ with $\psi|_{X_A}\:X_A\ra \psi(X_A)$. 
\[
\begin{xy} <30pt,0pt>:
(-4.5,0.4)*{}="b";
(0,1.4)*{}="c";
(0,0.5)*{\bullet}="d";
"b"+(-0.2,-0.2);"b"+(0.7,0.7)**\dir{-};
"b"+(-0.2,0.2);"b"+(0.7,-0.7)**\dir{-};
"b"+(1.5,0.6);"b"+(0,0.6)**\crv{"b"+(0.7,0.3)};
"b"+(1.5,-0.6);"b"+(0,-0.6)**\crv{"b"+(0.7,-0.3)};
"b"+(0.72,0.2)*{\scriptstyle{E_{R_t,\gamma'_t}}};
"b"+(0.67,-0.2)*{\scriptstyle{E_{R_t,\gamma_t}}};
"b"+(1.8,0.6)*{\scriptstyle{\wh C_{\gamma'_t}}};
"b"+(1.8,-0.6)*{\scriptstyle{\wh C_{\gamma_t}}};
"d"+(-0.8,0);"d"**\dir{-};
"d"+(0,-0.8);"d"**\dir{-};
"d"+(0.6,0.6);"d"**\dir{-};
"d"+(0.2,0)*{\scriptstyle{A}};
"d"+(1,-0.2)*{\scriptstyle{C_{\gamma'_1,\gamma_2,\phi}}};
"d"+(-0.2,0.4)*{\scriptstyle{C_{\gamma_1,\gamma'_2,\phi}}};
"d"+(-0.6,-0.4)*{\scriptstyle{C_{\gamma_1,\gamma_2,\phi}}};
"b"+(1,-1.1)*{C(2)=X_A\subset \wt\C^3};
"d"+(0,-1.2)*{A\in\wt\C^2};
"d"+(-1.8,0)*{\stackrel{p_1\circ\psi}{\lra}};
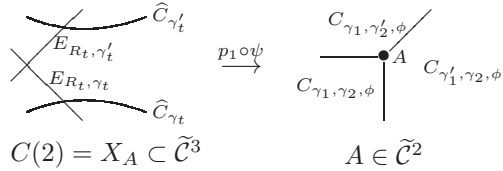
"d"+(-1.7,-2)*{\text{\bf Figure 3. } \text{The local picture of $A$ and the fiber $X_A=p_1\circ\psi^{-1}(A)$}.}
\end{xy}
\]

For each $t$ in $\{1,2\}$, let $\wh\T^1_{R_t}$ be the nested sets of 1-tails of $X_A=C(2)$ associated to 
 $E_{R_t,\gamma_t}$. We will set 
\[
\wh\T^1_A:=\wh\T^1_{R_1}\sqcup\wh{\T}^1_{R_2}.
\] 
 Let $\wh\T^s_A$ be the nested set of $s$-tails of $X_A=C(2)$ associated to $(E_{R_1,\gamma_1}, E_{R_2,\gamma_2})$, for each $s$ in $\{2,3\}$. 
We will set
\[
\wh\T_A:=\wh\T^1_A\sqcup \wh\T^2_A\sqcup \wh\T^3_A.
\]
Moreover, let $\T^s_{\alpha,\beta}$ be the nested sets of tails of $\psi(X_A)=C$ associated to $(C_\alpha,C_\beta)$, for each $(\alpha,\beta)$ in $\Pa_A$ and $s$ in $\{1,2,3\}$. We will set
 \[
 \begin{array}{lll}
 \T^s_A:=\T^s_{\gamma_1,\gamma_2} \sqcup \T^s_{\gamma_1,\gamma'_2} \sqcup\T^s_{\gamma'_1,\gamma_2} \quad\text{ and }\quad \T_A:=\T^1_A\sqcup \T^2_A\sqcup \T^3_A.
 \end{array}
 \]
 
For each $s$ in $\{2,3\}$, we define $m_s:=\#\wh\T^s_A-1$, and 
we will write 
\[
\wh\T^s_A=\{Y^s_0,\dots,Y^s_{m_s}\}.
\]   
Moreover, we will write
\[
\begin{array}{lll}
\T^s_{\gamma_1,\gamma_2}=\{W^s_0,W^s_3,\dots\}, & 
\T^s_{\gamma_1,\gamma'_2}=\{W^s_1,W^s_4,\dots\}, & \T^s_{\gamma'_1,\gamma_2}=\{W^s_2,W^s_5,\dots\},
\end{array}
\]
when these sets are nonempty. We have 
   $Y^s_{t-1}\prec Y^s_t$ for every $t$ in $\{1,\dots,m_s\}$, and 
 $W^s_{t-3}\prec W^s_t$, for every subset $\{W^s_{t-3},W^s_t\}$ of $\T^s_A$.

    \subsection{Quasistable and synchronized distinguished points}\label{quas-sync-dist}
    Keep the notation of Section \ref{Not3}.  We say that  $A$ is \emph{quasistable} if the following two conditions hold
\[
\begin{array}{l}
\#(\{R_1,R_2\}\cap (\cup_{W\in\T^2 _{\gamma_1,\gamma_2}\cup \T^3_{\gamma_1,\gamma_2}}\term_W))\le 1; \\
\#(\{R_1,R_2\}\cap (\cup_{W\in\T^2_{\gamma'_1,\gamma'_2}\cup \T^3_{\gamma'_1,\gamma'_2}}\term_W))\le 1.
\end{array}
\]

    We say that $A$ is \emph{$s$-synchronized}, for $s\in\{1,2,3\}$,  if the set function sending a tail $Y$ of $C(2)$  to the tail $\mu(Y)$ of $C$ induces a bijection
 \[
 \wh{\T}^s_A\lra \T^s_A.
 \]
  Also, $A$ is  \emph{synchronized} if  
 $A$ is  $s$-synchronized for each $s\in\{1,2,3\}$.      

\smallskip

Let us collect some easy consequences of the definitions.
   If $A$ is $s$-synchronized, then for every node $S$ of $C$ we have
 \begin{equation}\label{dtA}
 d_{\T^s_A,S}= \underset{\wh{S}\in \mu^{-1}(S)}{\sum} d_{\wh{\T}^s_A,\wh{S}}.
\end{equation}
In particular, if $A$ is 2-synchronized, then for every $t$ in $\{0,\dots,m_3\}$ we have
 \begin{equation}\label{ne3}
  d_{\T^2_A,S}\ne 3, \text{ for every } S\in\term_{\mu(Y^3_t)}.
  \end{equation}
  
The following property holds for every $Y$ in $\wh\T^s_A$ and $Y'$ in $\wh\T^{s'}_A$ with $\{s,s'\}=\{2,3\}$
\begin{equation}\label{RtTerm}
\text{ if } C_{\gamma'_t}\subseteq \mu(Y), \text{ for some } t\in \{1,2\}, \text{ then } R_t\not\in \term_{\mu(Y)}\cap \term_{\mu(Y')}.
\end{equation}
 Indeed, suppose that $C_{\gamma'_t}\subseteq \mu(Y)$ and $R_t\in \term_{\mu(Y)}\cap \term_{\mu(Y')}$. By Lemma \ref{free-perf} the pair $(\mu(Y),\mu(Y'))$ is perfect. Since by definition we have $E_{R_t,\gamma_t}\subseteq Y\wedge Y'$, we see that $C_{\gamma'_t}\subseteq \mu(Y')^c$. Thus, the unique possibility is that  $\mu(Y)$ and $\mu(Y')^c$ contain each other, hence $\{R_1,R_2\}$ is contained in $\term_{\mu(Y)}$ and in $\term_{\mu(Y')}$, contradicting the fact that $\{k_{\mu(Y)},k_{\mu(Y')}\}=\{2,3\}$.
 
 Suppose now that $A$ is quasistable and there are $W$ in $\T^2_{\alpha,\beta}$ and $W'$ in $\T^3_{\alpha',\beta'}$ such that  $R_1\in\term_W\cap \term_{W'}$, for $(\alpha,\beta)$ and $(\alpha',\beta')$ in $\P_A$. Then we have
\begin{equation}\label{alphabeta}
\begin{array}{lll}
(\alpha,\beta)=(\gamma_1,\gamma'_2) &  \text{and } &  \term_W=\{R_1,R_2\} \\
(\alpha',\beta')=(\gamma_1,\gamma_2) &   \text{and } & \term_{W'}\cap \{R_1,R_2\}=\{R_1\}.
\end{array}
\end{equation}
Indeed, first notice that $(\alpha,\beta)\ne(\alpha',\beta')$, because $R_1\in\term_W\cap \term_{W'}$, and hence $C_{\gamma_1}\subseteq W\cup W'$. Arguing as in the previous paragraph, we see that $C_{\gamma_1}\subseteq W \cap W'$ and that $W$ and $W'$ contain each other. In particular, we have 
\[
\{(\alpha,\beta),(\alpha',\beta')\}=\{(\gamma_1,\gamma_2),(\gamma_1,\gamma'_2)\}.
\] 
If $W$ crosses $R_2$, then by Lemma \ref{maxT2} there are tails both in 
$\T^2_{\gamma_1,\gamma_2}$ and in $\T^2_{\gamma_1,\gamma'_2}$ admitting $R_1$ as terminal point, which is a contradiction, because $R_1\in\term_{W'}$. Thus, we have $\term_W=\{R_1,R_2\}$. Since $A$ is quasistable, we see that $(\alpha,\beta)=(\gamma_1,\gamma'_2)$; since $W$ and $W'$ contain each other, the tail $W'$ crosses $R_2$.

    \smallskip
    
    Let $s$ be in $\{2,3\}$. Define $\zeta_s$ as follows: If $s=2$, set $\zeta_2:=0$; if $s=3$, we let $\zeta_3$ be the maximal positive integer greater or equal to 3 (if it exists) such that 
     \begin{equation}\label{zeta-cond}
     W^3_{t-3}=W^3_{t-2}=W^3_{t-1}, \text{ for every } t\in\{3,\dots,\zeta_s\}  \text{ with } t\equiv0(3),
     \end{equation}
     and set $\zeta_3:=0$ otherwise. Notice that $\zeta_s\equiv 0(3)$. 
 Fix $(\alpha,\beta)\in \P_A$. We define
\[
t_0:=
\begin{cases}
\begin{array}{ll}
0 & \text{ if } (\alpha,\beta)=(\gamma_1,\gamma_2); \\ 
1 & \text{ if } (\alpha,\beta)=(\gamma_1,\gamma'_2); \\
2 & \text{ if } (\alpha,\beta)=(\gamma'_1,\gamma_2).
\end{array}
\end{cases}
\]
In what follows, if either $\T^s_{\alpha,\beta}=\emptyset$ (and hence $\zeta_s=0$), or $\zeta_s\ge 3$  and $W^s_{\zeta_s+t_0-3}$ is the maximal element of $\T^s_{\alpha,\beta}$,  
then we will set
 \[
 W^s_{\zeta_s+t_0}:=C.
 \]

\begin{Lem}\label{LemB}
Assume that $A$ is quasistable. The following properties hold.
\begin{itemize}
\item
[(i)]
If $d_{\T^2_A,R_1}\ne 0$ and $S$ is a node of $C$ such that $d_{\T^2_A,S}\not\equiv 0(3)$, then either $S\in\{R_1,R_2\}$ or $S$ is the difference node of $\T^2_{\gamma_1,\gamma_2|\gamma'_1,\gamma_2}$ distinct from $R_1$.
\item[(ii)]
If $d_{\T^2_A,R_1}\ne 0$, $d_{\T^3_A,R_1}\ne 0$, then 
\[
\T^2_{\gamma_1,\gamma_2}=\T^2_{\gamma'_1,\gamma_2}=\T^2_{\gamma_1,\gamma'_2}\setminus\{W^2_1\}\;\text{ and } \;\term_{W^2_1}=\{R_1,R_2\},
\]
where $W^2_1$ is the minimal tail of $\T^2_{\gamma_1,\gamma'_2}$.
\end{itemize}
\end{Lem}

\begin{proof}
Let $S$ be a node of $C$ such that $d_{\T^2_A,S}\not\equiv0(3)$. 
If $d_{\T^2_A,R_1}\ne0$ and $d_{\T^2_A,R_2}=0$, then $S$ is a difference node of $\T^2_{\gamma_1,\gamma_2|\gamma'_1,\gamma_2}$, by Proposition \ref{Diff}(i). If $d_{\T^2_A,R_1}\ne0$ and $d_{\T^2_A,R_2}\ne0$, then there is a 2-tail $W$ with $R_1$ and $R_2$ as terminal points. Indeed, if $W_t$ is a tail of $\T^2_A$ such that $\{R_1,R_2\}\cap \term_{W_t}=\{R_t\}$, for each $t\in\{1,2\}$, then $W_1\wedge W_2$ is a 2-tail as required. Since $A$ is quasistable, we have 
\[
W=W^2_u, \text{ for some } u\in\{1,2\},
\] 
where $W^2_u$ is the minimal tail of $\T^2_{\gamma_u,\gamma'_{3-u}}$.
 Since $W^2_{u+3}$ crosses $R_1$ and $R_2$, we get 
\begin{equation}\label{W2u}
W^2_{3-u}\subseteq W^2_{u+3}.
\end{equation}
If a tail $W'$ of $\T^2_A$ admits $R_u$ as a terminal point, then $C_{\gamma_u}\subseteq W'$, otherwise by Lemma \ref{free-perf} the pair $(W^2_u,W')$ would be terminal, with the unique possibility that  $W^2_u$ and $(W')^c$ contain each other, implying $W^2_u=(W')^c$, a contradiction. Similarly, if $R_{3-u}\in\term_{W'}$, for some $W'\in\T^2_A$, then $C_{\gamma'_{3-u}}\subseteq W'$. Thus, the other inclusion in \eqref{W2u} holds and hence
\begin{equation}\label{minW2u}
\T^2_{\gamma'_u,\gamma_{3-u}}=\T^2_{\gamma_u,\gamma'_{3-u}}\setminus\{W^2_u\},
\end{equation}
which concludes the proof of the first part item (i). 

Suppose now that $d_{\T^2_A,R_1}\ne 0$ and $d_{\T^3_A,R_1}\ne 0$. Let $W\in\T^2_{\alpha,\beta}$ and $W'\in\T^3_{\alpha',\beta'}$ such that $R_1\in\term_W\cap \term_{W'}$, for $(\alpha,\beta)$ and $(\alpha',\beta')$ in $\P_A$. The properties stated in \eqref{alphabeta} hold. Arguing as in the first part of the proof, we get that \eqref{minW2u} holds with $u=1$. Moreover, $W^2_0$ and $W^2_2$ cross both $R_1$ and $R_2$, the tail $W^2_0$ because $(\alpha',\beta')=(\gamma_1,\gamma_2)$ and by \eqref{diff-rel}, and $W^2_2$ by \eqref{minW2u}, and hence
 $\T^2_{\gamma_1,\gamma_2}=\T^2_{\gamma'_1,\gamma_2}$.
\end{proof}

\begin{Lem}\label{LemD}
Assume that $A$ is quasistable and 2-synchronized. Assume that $d_{\T^2_{\gamma_t,\gamma'_{3-t}},R_1}\ne 0$, for some $t$ in $\{1,2\}$, and that $d_{\T^2_A,R_2}=0$. Consider the set
\[
\R=\{W\in\T^3_A : R_1\not\in\term_W \text{ and } R_2\in\term_W\}.
\] 
Assume that at least one of the following properties holds
\begin{itemize}
\item[(i)]
there is $Y\in\wh\T^3_A$ such that 
$R_1\not\in\term_{\mu(Y)}$ and $R_2\in\term_{\mu(Y)}$;
\item[(ii)]
$\R$ is nonempty.
\end{itemize}
Then $\R$ admits a minimal tail $W$, which is the minimal tail of $\T^3_{\gamma_t,\gamma'_{3-t}}$, and if $Y$ is as in (i), then $W\subseteq \mu(Y)$. If $S$ is a node of $C\setminus\{R_1\}$ such that $d_{\T^2_A,S}\not\equiv0(3)$, then
\[
S\in W\;  \text{ and }\; d_{\T^2_{\gamma_t,\gamma'_{3-t}},S}=0.
\]
\end{Lem}

 \begin{proof}
 Let $W'\in\T^2_{\gamma_t,\gamma'_{3-t}}$ such that $R_1\in\term_{W'}$.
  Since $d_{\T^2_A,R_1}\ne0$ and $d_{\T^2_A,R_2}=0$ and $A$ is quasistable, by Proposition \ref{Diff}(i) we have 
  \begin{equation}\label{equalityT2}
  \T^2_{\gamma_1,\gamma_2}=\T^2_{\gamma_1,\gamma'_2} \ne
  \T^2_{\gamma'_1,\gamma_2}=\T^2_{\gamma'_1,\gamma'_2}.
  \end{equation}  
  There is exactly a node $S$ of $C\setminus\{R_1\}$ such that $d_{\T^2_A,S}\not\equiv0(3)$; such a node is the difference node of $\T^2_{\gamma_1,\gamma'_2|\gamma'_1,\gamma_2}$ distinct form $R_1$.  
In what follows, we denote by $Z$ either a tail $\mu(Y)$, for $Y$ as in (i), or a tail of $\R$.

We claim that $S$ is contained in $Z$ and $d_{\T^2_{\gamma_t,\gamma'_{3-t}},S}=0$. Set $X_0:=W'$. Write 
\[
\term_{X_0}=\{T_0,T_1\}, \text{ where } T_0:=R_1, \text{ and } \; \term_Z=\{R_2,U,V\}.
\] 
 Notice that $Z\wedge X_0$ admits $R_2$ as terminal point,  hence $Z\wedge X_0$
is not a 2-tail, by \eqref{equalityT2}. Moreover, $Z$ is not $X_0$-terminal, otherwise $\term_{X_0}\subseteq Z$, hence $X_0\subseteq Z$ (see  Lemma \ref{free-perf}), a contradiction. Thus, by \cite[Lemma 3.3(iii)]{P1} we see that $Z\cup X_0$ is a 2-tail crossing both $T_0$ and $R_2$. We may assume that $\term_{Z\cup X_0}=\{T_1,U\}$, and that $X_0$ crosses $V$. By Lemma \ref{maxT2}, there is a $(Z\cup X_0)$-terminal tail $X_1$ in $\T^2_{\gamma'_t,\gamma_{3-t}}$ such that $X_1\subseteq Z\cup X_0$. We may assume that $T_1\in \term_{X_1}$, otherwise $U\in\term_{X_1}$ and Proposition \ref{Diff}(i) would imply that   
\[
X_1\in \T^2_{\gamma_1,\gamma_2}\cap \T^2_{\gamma_1,\gamma'_2}\cap \T^2_{\gamma'_1,\gamma_2},
\]
 hence $Z\not\in\T^3_A$. Also, by \eqref{ne3}, we would have $Z\ne\mu(Y)$, for every $Y\in\wh\T^3_A$, a contradiction in any case. 
Write $\term_{X_1}=\{T_1,T_2\}$. Of course, $T_2$ is not in $X_0$, because $X_0\subsetneq X_1$, and hence $T_2$ is in $Z$, because $X_1\subset Z\cup X_0$. If $S$ is equal to $T_2$, then we are done. If not, then $Z$ crosses $T_2$ (otherwise, $d_{\T^2_A,T_2}\not\equiv0(3)$ and $T_2=S$) and there is $X_2\in \T^2_{\gamma_t,\gamma'_{3-t}}$ such that $T_2\in\term_{X_2}$ and 
$X_1\subsetneq X_2$. Write $\term_{X_2}=\{T_2,T_3\}$. 
Arguing as before, $Z\cup X_2$ is a 2-tail crossing both $T_2$ and $R_2$, with $T_3$ and $U$ as terminal points. The claim follows just by iterating the reasoning.

If $t=1$, then $C_{\gamma'_2}\subseteq Z$, otherwise we would have
\[
C_{\gamma_2}\subseteq Z\; \text{ and } \; W'\in\T^2_{\gamma_1,\gamma'_2}=\T^2_{\gamma_1,\gamma_2};
\] 
by the claim and by Lemma \ref{maxT3}, $Z$ would contain a tail of $\T^3_{\gamma_1,\gamma_2}$ with $R_2$ as terminal point, contradicting that $A$ is quasistable. Similarly, if $t=2$, then $C_{\gamma_2}\subseteq Z$, otherwise we would get the same contradiction by noticing that the properties
\[
C_{\gamma'_2}\subseteq Z\; \text{ and } \; W'\in\T^2_{\gamma'_1,\gamma_2}=\T^2_{\gamma'_1,\gamma'_2}
\] 
imply that $Z$ contains a tail of $\T^3_{\gamma'_1,\gamma'_2}$ with $R_2$ as terminal point.

Thus, using the claim and Lemma \ref{maxT3}, we see that $Z$ contains a tail $W$ in $\T^3_{\gamma_t,\gamma'_{3-t}}$. Notice that $W$ admits $R_2$ as terminal point, hence $W$ is the minimal tail of $\T^3_{\gamma_t,\gamma'_{3-t}}$. Since $d_{\T^2_{\gamma_t,\gamma'_{3-t}},R_1}\ne0$, the tail $W$ crosses $R_1$, hence $W\in\R$. Since any $Z$ contains $W$, we see that $W$ is the minimal tail of $\R$ and $W\subseteq \mu(Y)$, for every $Y$ as in (i).
\end{proof}

\begin{Prop}\label{PropA}
Fix $s$ in $\{2,3\}$. Assume that $A$ is quasistable and, if $s=3$, also 2-synchronized. Then there is a permutation $\tau_s$ of $\{0,1,2\}$ such that
\[
W^s_{\zeta_s+\tau_s(0)}\subseteq W^s_{\zeta_s+\tau_s(1)}\subseteq W^s_{\zeta_s+\tau_s(2)}
\]
and satisfying the following properties
\begin{itemize}
\item[(i)]
if $s=2$, or $s=3$ with $d_{\T^2_A,R_1}=d_{\T^2_A,R_2}=0$, then $\tau_s(1)=0$;
\item[(ii)]
if $d_{\T^2_A,R_1}\ne0$ and $d_{\T^3_A,R_1}\ne0$, then $\zeta_3=\tau_3(0)=0$.\end{itemize}
\end{Prop}

\begin{proof}
By Remark \ref{Rem1}, for every $t\in\{1,2\}$ one of the following inclusions holds 
\begin{equation}\label{containing}
\begin{array}{lll}
 W^s_{\zeta_s+t}\subseteq W^s_{\zeta_s} \; \text{ or } \; W^s_{\zeta_s}\subseteq W^s_{\zeta_s+t}. 
\end{array}
\end{equation}

Suppose that $s=2$, or $s=3$ with $d_{\T^2_A,R_1}=d_{\T^2_A,R_2}=0$. 
If $s=3$, using that $A$ is quasistable, we have 
\begin{equation}\label{nodiff}
\T^2_{\gamma_1,\gamma_2}=\T^2_{\gamma_1,\gamma'_2}=\T^2_{\gamma'_1\gamma_2}=\T^2_{\gamma'_1,\gamma'_2}.
\end{equation}
In particular, $\zeta_s=0$.
If $W^s_t\subseteq W^s_0\subseteq W^s_{3-t}$, for some $t\in\{1,2\}$, of course we are done. Thus, by \eqref{containing}, we may assume that one of the following properties hold 
\begin{equation}\label{eitheror}
W^s_0\subseteq W^s_1\wedge W^s_2\; \text{ or } \; W^s_1\cup W^s_2\subseteq W^s_0.
\end{equation}
 If $W^s_0=W^s_t$, for some $t\in\{1,2\}$, again we are done, hence we may also assume that $W^s_0\ne W^s_t$, for every $t\in\{1,2\}$. 
 
 If the left hand side of \eqref{eitheror} holds, by Lemma \ref{Diff}(i) we have that $W^s_0$ admits $\{R_1,R_2\}$ as terminal points, contradicting the fact that $A$ is quasistable. 
 If the right hand side of \eqref{eitheror} holds, again by Lemma \ref{Diff}(i) we have  
 \[
 R_1\in\term_{W^s_2}\; \text{ and } \; R_2\in\term_{W^s_1}. 
 \] 
 For every $t\in\{1,2\}$, the tail $W^s_t$ crosses $R_t$, otherwise either $(W^s_1,W^s_2)$ is perfect, implying that  $W^s_1\subseteq (W^s_2)^c$ and hence $s=3$, or $s=3$ and $W^s_1\wedge W^s_2$ is nonemtpy. In both cases, $s=3$ and $W^s_1\cup W^s_2$ is a 2-tail, hence, by Lemma \ref{maxT2} we would have a $(W^s_1\cup W^s_2)$-terminal tail in $\T^2_A$, and using \eqref{nodiff} we would get a contradiction. Since $W^s_t$ crosses $R_t$ and contains $R_{3-t}$ as terminal point for every $t\in\{1,2\}$, it follows from \eqref{nodiff} and from Lemmas \ref{maxT2} and \ref{maxT3} that there are tails in $\T^s_{\gamma'_1,\gamma'_2}$ containing $R_1$ and $R_2$ as terminal points, which is not possible because $A$ is quasistable.

Suppose now that $s=3$ and $d_{\T^2_A,R_1}\ne0$, with $R_1\in \term_W$, for $W\in\T^2_A$. By Lemma \ref{LemB}(i) there is at most one node $S$ of $C\setminus\{R_1,R_2\}$ such that $d_{\T^2_A,S}\not\equiv0(3)$. In particular, there are distinct pairs $(\alpha_1,\beta_1),(\alpha_2,\beta_2)\in\P_A$ such that $d_{\T^2_{\alpha_1,\beta_1},S}=d_{\T^2_{\alpha_2,\beta_2},S}$, for every node $S\in C\setminus\{R_1,R_2\}$. If $d_{\T^3_A,R_1}=d_{\T^3_A,R_2}=0$, we see that 
\[
W^s_{\zeta_s+t_1}=W^s_{\zeta_s+t_2}, \text{ for distinct } t_1,t_2\in\{0,1,2\},
\]
and combining this equality with \eqref{containing} we are done. If $d_{\T^3_A,R_1}\ne0$, then $\zeta_s=0$. By Lemma \ref{LemB}(ii) we have $d_{\T^2_A,S}\equiv 0(3)$, for every node $S\in C\setminus\{R_1,R_2\}$. Let $W'\in\T^3_A$ be such that $R_1\in\term_{W'}$. By \eqref{alphabeta} we have 
 \[
 \begin{array}{lll}
 W\in\T^2_{\gamma_1,\gamma'_2} &  \text{ and } &  \term_{W}=\{R_1,R_2\}; \\ 
 W'=W^s_0\in\T^3_{\gamma_1,\gamma_2} &  \text{ and} &   \term_{W'}\cap \{R_1,R_2\}=\{R_1\}.
 \end{array}
 \] 
 Thus, $W^s_1$ crosses $R_1$ and $R_2$; also, using Lemma \ref{Diff}(i), we see that
 $W^s_0\subseteq W^s_2$ and $R_1\not\in\term_{W^s_2}$, hence $W^s_2$ crosses $R_1$ and $R_2$ as well. We deduce that $W^s_1=W^s_2$ and, using again  \eqref{containing}, we are done. From now on, we may assume $s=3$ and
 \[
 d_{\T^2_A,R_1}\ne0, \; d_{\T^3_A,R_2}\ne 0\; \text{ and } \; d_{\T^3_A,R_1}=d_{\T^2_A,R_2}=0,
 \] 
and hence that $\T^2_{\gamma_1,\gamma_2}=\T^2_{\gamma_1,\gamma'_2}$.
 
 We now use Equation \eqref{containing}, Remark \ref{Rem2} and Lemma \ref{LemD}.  Suppose $d_{\T^2_{\gamma_1,\gamma'_2},R_1}\ne0$. We have $R_2\in\term_{W^s_1}$,  hence $W^s_1\subseteq W^s_0$, and also $d_{\T^2_{\gamma_1,\gamma_2},S}=d_{\T^2_{\gamma_1,\gamma'_2},S}=0$, for every node $S\in C\setminus\{R_1\}$ such that $d_{\T^2_A,S}\not\equiv0(3)$. Thus, if $W^s_2$ crosses $R_2$, then $W^s_0\subseteq W^s_2$. If $R_2\in\term_{W^s_2}$, then $W^s_1\subseteq W^s_2$ (recall that $W^s_1$ is the minimal tail of the set $\R$ in Lemma \ref{LemD}).
 Suppose $d_{\T^2_{\gamma'_1,\gamma_2},R_1}\ne0$. We have $R_2\in\term_{W^s_2}$. 
  If $R_2\in \term_{W^s_t}$, for some 
 $t\in\{0,1\}$, then we have $W^s_2\subseteq W^s_t\subseteq W^s_{1-t}$. If $W^s_0$ and $W^s_1$ cross $R_2$, then, using that $\T^2_{\gamma_1,\gamma_2}=\T^2_{\gamma_1,\gamma'_2}$, we have $W^s_0=W^s_1$. 
\end{proof}

\begin{Lem}\label{LemC}
Fix $s\in\{2,3\}$. 
Assume that $A$ is quasistable and, if $s=3$, also 2-synchronized.
If $Y$ is a tail in $\wh\T^s_A$, then $\mu(Y)$ contains a tail of $\T^s_A$.
\end{Lem}

\begin{proof}
If $s=2$, then $\mu(Y)$ is a 2-tails and, using that $A$ is quasistable, we have $C_{\alpha}\cup C_\beta\subseteq \mu(Y)$, for some $(\alpha,\beta)\in\P_A$. By Lemma \ref{maxT2}, $\mu(Y)$ contains a tail in $\T^2_A$.

Thus, we may assume $s=3$. 
If $\mu(Y)$ crosses $R_1$ and $R_2$, then we are done, because $\mu(Y)$ is $\T^2_{\alpha,\beta}$-free for some $(\alpha,\beta)$ in $\Pa_A$  (just combine \eqref{ne3} and Lemma \ref{LemB}(i)). If $d_{\T^2_A,R_1}=d_{\T^2_A,R_2}=0$, then the fact that  $A$ is quasistable implies that
     \[
     \T^2_{\gamma_1,\gamma_2}=\T^2_{\gamma_1,\gamma'_2}=\T^2_{\gamma'_1,\gamma_2}=\T^2_{\gamma'_1,\gamma'_2},
     \]
hence $\mu(Y)$ is $(\T^2_A\cup \T^2_{\gamma'_1,\gamma'_2})$-free (see \eqref{dtA} and \eqref{ne3}). Using that $A$ is quasistable, $\mu(Y)$ contains $C_\alpha\cup C_\beta$, for some $(\alpha,\beta)\in\P_A$, then $\mu(Y)$ contains a tail of $\T^3_A$. 

Suppose that $R_1\in \term_W\cap \term_{\mu(Y)}$, for some $W\in\T^2_A$. Since $A$ is 2-synchronized, $W$ is the image via $\mu$ of a tail in $\wh\T^2_A$. Using \eqref{RtTerm} and the definition of $\wh\T^s_A$, for every $Y'\in\wh\T^2_A$ such that $R_1\in\term_{\mu(Y')}$, we have
\[
C_{\gamma_1}\subseteq \mu(Y)\wedge \mu(Y') \text{ and } E_{R_1,\gamma_1}\subseteq Y\wedge Y' ,
\]
and hence $d_{\T^2_A,R_1}\leq1$ and $d_{\T^2_A,R_1}=1$ (see also \eqref{dtA}). 
Since $d_{\T^2_A,R_1}=1$  and $C_{\gamma_1}\subseteq W$, using that $A$ is quasistable we have
$d_{\T^2_{\gamma_1,\gamma_2},R_1}=d_{\T^2_{\gamma'_1,\gamma_2},R_1}=0$, and
\[ 
W=W^2_1\in\T^2_{\gamma_1,\gamma'_2},\; \text{ with } \; \term_{W^2_1}=\{R_1,R_2\}.
\]
  By \eqref{diff-rel}, we get $d_{\T^2_{\gamma_1,\gamma_2},R_2}=0$ and we see that 
  the minimal tail of $\T^2_{\gamma_1,\gamma_2}$ is $W^2_1$-free.
  By Proposition \ref{Diff}, we deduce that
 \[
 \T^2_{\gamma_1,\gamma_2}=\T^2_{\gamma'_1,\gamma_2}=\T^2_{\gamma_1,\gamma'_2}\setminus\{W^2_1\}.
 \]
As a consequence of \eqref{ne3}, we have that $\mu(Y)$ is $\T^2_{\gamma_1,\gamma_2}$-free. 
 Using again \eqref{RtTerm}, we obtain that $\mu(Y)$ crosses $R_2$, then $C_{\gamma_2}\subseteq \mu(Y)$.  By Lemma \ref{maxT3}, $\mu(Y)$ contains a tail in $\T^3_{\gamma_1,\gamma_2}$, and we are done.
 
In the remaining cases, we are done by Lemma \ref{LemD}. 
 \end{proof}

    \begin{Lem}\label{LemA}
    Assume that $A$ is quasistable and 2-synchronized. If $\zeta_3\ge3$, then $\#\wh\T^3_A\ge\zeta_3$ and, for every $t$ in $\{3,\dots,\zeta_s\}$ with $t\equiv 0 (3)$, we have
    \[
    \mu(Y^3_{t-3})=\mu(Y^3_{t-2})=\mu(Y^3_{t-1})=W^3_{t-3}=W^3_{t-2}=W^3_{t-1}.
    \]
    \end{Lem}
    
   \begin{proof}
 Set $W_{\cdot}:=W^3_{\cdot}$, $Y_{\cdot}:=Y^3_{\cdot}$ and $\zeta:=\zeta_3$.
  By Equations \eqref{dtA} and \eqref{zeta-cond}, the canonical liftings of $W_t$ are $\wh \T^2_A$-free for every $t<\zeta$, and fit in a chain
   \[
   L_0^{W_0}\prec L_1^{W_1}\prec L_2^{W_2}\prec L_0^{W_3}\prec \cdots \prec L_1^{W_{\zeta-2}}\prec L_2^{W_{\zeta-1}}.
   \]
  We see that $\#\wh\T^3_A\ge \zeta$ and, for every $t\in\{0,\dots,\zeta-1\}$, there is $i\in\{0,1,2\}$ such that $Y_t\subseteq L_i^{W_t}$, hence $\mu(Y_t)\subseteq W_t$.  
       By contradiction, assume that $\mu(Y_u)\neq W_u$ for some $u<\zeta$, and let $u$ be the minimum for which this condition holds.  
       If $u>0$, we have  
       $W_{u-1}=\mu(Y_{u-1})\subseteq\mu(Y_u)\subseteq W_u$, then by \eqref{zeta-cond} we get $u\equiv 0(3)$ and hence
      \begin{equation}\label{equalities}
      W_{u-3}=W_{u-2}=W_{u-1}=\mu(Y_{u-3})=\mu(Y_{u-2})=\mu(Y_{u-1})\prec \mu(Y_u). 
      \end{equation}
  
  To get a contradiction, it suffices that $W_{u+j}\subseteq\mu(Y_u)$, for some $j\in\{0,1,2\}$, because $W_{u+j}=W_u$ for every $j\in\{0,1,2\}$. Indeed, if $\mu(Y_u)$ crosses $R_1$ and $R_2$, then it is $\T^2_{\alpha,\beta}$-free for every $(\alpha,\beta)\in\Pa_A$ (use \eqref{zeta-cond} and that $A$ is 2-synchronized), and hence we are done by \eqref{equalities}. On the other hand, if $R_1\in\term_{\mu(Y_u)}$, then $u=0$; we need only show that $\mu(Y)$ contains a tail of $\T^3_A$, which is true by Lemma \ref{LemC}.
       \end{proof}

\begin{Prop}\label{PropB}
 Fix $s$ in $\{2,3\}$ and $t$ in $\{0,1,2\}$. Assume $A$ quasistable and, if $s=3$, also 2-synchronized. Let $\tau_s$ be as in Proposition \ref{PropA}.  If  $\#\wh\T^s_A>\zeta_s+t$ and $\mu(Y^s_{\zeta_s+t'})=W^s_{\zeta_s+\tau_s(t')}$ for each $t'$ in $\{0,\dots,t-1\}$, then $W^s_{\zeta_s+\tau_s(t)}\subseteq \mu(Y^s_{\zeta_s+t})\ne C$.
\end{Prop}

\begin{proof}
Recall that by Lemma \ref{PropA} we have
\[
W^s_{\zeta_s+\tau_s(0)}\subseteq W^s_{\zeta_s+\tau_s(1)}\subseteq W^s_{\zeta_s+\tau_s(2)}.
\]
If $t=\zeta_s=0$ and $\wh\T^s_A\ne\emptyset$, we need only show that $\mu(Y^s_0)$ contains a tail of $\T^s_A$, which is true by Lemma \ref{LemC}.
We have two remaining cases. 

\smallskip

\emph{\bf Case 1:} $\zeta_s\ge 3$ (hence $s=3$). 
  By Lemma \ref{LemA}, we have $\#\wh\T^s_A\ge \zeta_s$ and
  \[
  \mu(Y^s_{\zeta_s-3})=\mu(Y^s_{\zeta_s-2})=\mu(Y^s_{\zeta_s-1})=W^s_{\zeta_s-3}=W^s_{\zeta_s-2}=W^s_{\zeta_s-1}.
  \]
 We see that  $W^s_{\zeta_s-t_1}\prec \mu(Y^s_{t_2})$ and $\mu(Y^s_{t_2})$ crosses $R_1$ and $R_2$, for every $t_1\in\{0,1,2\}$ and for every $Y^s_{t_2}\in \T^s_A$ with $t_2\ge \zeta_s$.
 
 Suppose $\#\wh\T^s_A>\zeta_s$. Since $W_{\zeta_s+\tau_s(0)}$ crosses $R_1$ and $R_2$ and is not contained in $\cap_{(\alpha,\beta)\in\P_A} \T^s_{\alpha,\beta}$, by Remark \ref{Rem2} 
 there is a terminal point $S$ of  $W_{\zeta_s+\tau(0)}$ such that $S\in C\setminus\{R_1,R_2\}$ and $d_{\T^2_A,S}\not\equiv0(3)$. By Lemma \ref{LemB}(i), the last two properties uniquely determined  $S$. Using \eqref{ne3} we see that $\mu(Y^s_{\zeta_s})$ is $\T^2_{\alpha,\beta}$-free for some $(\alpha,\beta)\in\P_A$, hence it contains $W^s_{\zeta_s+\tau_s(0)}$, concluding the case $t=0$.

 Suppose $\#\T^s_A>\zeta_s+1$ and $\mu(Y^s_{\zeta_s})=W_{\zeta_s+\tau(0)}$. Using \eqref{dtA}, we see that either $d_{\T^2_A,S}=1$ and hence $\mu(Y^s_{\zeta_s+1})$ is $\T^2_{\alpha,\beta}$-free for two distinct pair $(\alpha,\beta)$ in $\P_A$, or $d_{\T^2_A,S}=2$ and hence $\mu(Y^s_{\zeta_s+1})$ crosses $S$ and it is $\T^2_A$-free. In any cases, we have 
 \begin{equation}\label{Y1}
 W_{\zeta_s+\tau(0)}\subseteq W_{\zeta_s+\tau(1)}\subseteq\mu(Y^s_{\zeta_s+1})\ne C
 \end{equation}
 
 Finally, suppose $\#\T^s_A>\zeta_s+2$ and $\mu(Y^s_{\zeta_s})=W_{\zeta_s+\tau(0)}$. Using again \eqref{dtA}, we get that $\mu(Y^s_{\zeta_s+2})$ crosses $S$, and hence it is $\T^2_A$-free, from which we get
 \begin{equation}\label{Y2}
  W_{\zeta_s+\tau(0)}\subseteq W_{\zeta_s+\tau(1)}\subseteq W_{\zeta_s+\tau(2)}\subseteq\mu(Y^s_{\zeta_s+2})\ne C.
 \end{equation}

\emph{\bf Case 2:} $t>\zeta_s=0$. By the given hypothesis, we have $\mu(Y^s_0)=W:=W_{\tau(0)}$.  Suppose that $\#\T^s_A>t$. 
If $\mu(Y^s_0)$ crosses $R_1$ and $R_2$, we have two cases. In the first case, we have $s=2$, hence $\mu(Y^s_0)$ contains a tail of $\T^2_{\alpha,\beta}$, for every $(\alpha,\beta)\in\P_A$. In the second case, we have $s=3$, and we can conclude arguing exactly as in the last two paragraphs of Case 1. Therefore, we may assume that  $\mu(Y^s_0)$, and hence $W$, admits a terminal point in the set $\{R_1,R_2\}$. 

Notice that $C_{\gamma'_u}\subseteq W$, for some $u\in\{1,2\}$, otherwise we would have 
\[
W\in \T^s_{\gamma_1,\gamma_2}\; \text{ and } \; \{R_1,R_2\}\subseteq \term_W,
\]
 contradicting the fact that $A$ is quasistable. In particular, since $E_{R_u,\gamma_u}\subseteq Y^s_0$, there are $t+1$ distinct pairs $(\alpha,\beta)$ in $\P_A$ such that $C_\alpha\cup C_\beta\subseteq \mu(Y^s_t)$. 
If either $s=2$, or $s=3$ with $d_{\T^2_A,R_1}=d_{\T^2_A,R_2}=0$, we see that $\mu(Y^s_t)$ contains $t+1$ tails of $\T^2_A$, hence Equations \eqref{Y1} and \eqref{Y2} hold respectively when $t=1$ and $t=2$, and for $\zeta_s=0$. Therefore, we may also assume that $s=3$ and $d_{\T^2_A,R_1}\ne0$.

Suppose $R_1\in\term_W$. By Lemma \ref{PropA}(ii), we have $W\in\T^3_{\gamma_1,\gamma_2}$. Since $A$ is quasistable and $W=\mu(Y^s_0)$, we see that $\mu(Y^s_0)$ crosses $R_2$. We now use \eqref{dtA} and Lemma \ref{LemB}(ii): We deduce that $\mu(Y^s_t)$ crosses $R_1$ and $R_2$ and $d_{\T^2_A,S}\equiv0(3)$, for every node $S\in C\setminus\{R_1,R_2\}$, hence that $\mu(Y^s_t)$ is $\T^2_A$-free. Again, we see that Equations \eqref{Y1} and \eqref{Y2} hold respectively when $t=1$ and $t=2$, and for $\zeta_s=0$. 

Suppose $\term_W\cap \{R_1,R_2\}=\{R_2\}$ and $d_{\T^2_A,R_2}=0$. By Lemma \ref{Diff}(i), we have 
\begin{equation}\label{=}
\T^2_{\gamma_1,\gamma_2}=\T^2_{\gamma_1,\gamma'_2}.
\end{equation} By Lemmas \ref{LemB}(i) and \ref{LemD} there is a unique node $S\in C\setminus\{R_1,R_2\}$ such that $d_{\T^2_A,S}\not\equiv0(3)$ and $S\in W$.
 If $S\not\in \term_W$, then $\mu(Y^s_t)$ is $\T^2_A$-free,  and we conclude as in the previous paragraph. Thus, we may assume that $S\in\term_W$.  
If either $d_{\T^2_A,S}=2$ or $t=2$, then $\mu(Y^s_t)$ crosses $S$ and hence it is $\T^2_A$-free (we use again \eqref{dtA}), and  we are done. 
If $d_{\T^2_A,S}=t=1$, then Equation \eqref{=} implies that $d_{\T^2_{\gamma'_1,\gamma_2},S}\ne0$, hence by Lemma \ref{LemD} we get $W\in\T^3_{\gamma_1,\gamma'_2}$.
Since 
\[
E_{R_2,\gamma_2}\subseteq Y^s_0 \; \text{ and } \; W=\mu(Y^s_0),
\]
 it follows that $\mu(Y^s_1)$ crosses $R_1$ and $R_2$. Since $d_{\T^2_A,S}=1$, we conclude that $\mu(Y_1)$ is $\T^2_{\alpha,\beta}$-free for two distinct pairs $(\alpha,\beta)$ in $\P_A$, and hence Equation \eqref{Y1} holds. 
\end{proof}

  \section{The resolution of the degree-2 Abel-Jacobi map}
  \noindent
Let $\pi\: \C\ra B$ be a smoothing of a nodal curve $C$ and  
$\phi\:\wt\C^2\ra \C^2$ be a good partial desingularization.     
From now on, we only consider the Abel--Jacobi map of $\C/B$
\[
\alpha^2\:\C^2\dashrightarrow \J
\] 
introduced in \eqref{AJ}. Our goal is to find a resolution of $\alpha^2$ (see Theorem \ref{main5}). 
The strategy consists in finding the $\phi$'s for which the map $\alpha^2\circ\phi\col\wt\C^2\dashrightarrow\J$ induced by $\psi_*\L_\psi$ is defined everywhere. Thanks to Lemma \ref{ext-red} and Theorem \ref{main3},  we need only deal with distinguished points of $\wt\C^2$ lying over points of $\C^2$ consisting of pairs of distinct reducible nodes of $C$.  
To produce a map $\phi$ achieving our goal, we provide a combinatorial characterization of the blowups of $\C^2$ for which $\alpha^2\circ\phi$ is defined on an open subset of $\wt\C^2$ containing such a distinguished point (see Theorem \ref{main4}).
 
Throughout this section, we will keep the notation of Section \ref{Not3}.
    
    \subsection{Comparing properties of distinguished points}
    
   In this subsection, we will prove the key result stating that if a distinguished point is quasistable, then it is synchronized (see Proposition \ref{qs->s}). Actually, the two notions turn out to be equivalent (see Theorem \ref{main4}). 
    
\begin{Lem}\label{1sync}
The distinguished point $A$ is 1-synchronized.
\end{Lem}

\begin{proof}
For each $t\in\{1,2\}$, write $\wh\T^1_{R_t}=\{Y_{t,0},Y_{t,1},\dots\}$, where $Y_{t,0}\prec Y_{t,1}\prec \cdots$. We have
\[
\T^1_A=\T^1_{\gamma_1}\sqcup \T^1_{\gamma_2}\sqcup \T^1_{\gamma_1}\sqcup\T^1_{\gamma'_2}\sqcup\T^2_{\gamma'_1}\sqcup\T^1_{\gamma'_2}.
\]
 If $R_t$ is not a separating node, then $\T^1_{\gamma_t}=\T^1_{\gamma'_t}$, and $\wh{\T}^1_{R_t}$ consists of the canonical liftings of the tails in $\T^1_{\gamma_t}$. Assume that $R_t$ is a separating node and let $W$ be the 1-tail of $C$ terminating in $R_t$ and containing $\sigma(0)$. Then $\T^1_{\gamma_t}$ and $\T^1_{\gamma'_t}$ differ exactly by $W$. Let $\S$ be the subset of $\T^1_{\gamma_t}\cup \T^1_{\gamma'_t}$ of the tails different from $W$. The canonical liftings of the tails in $\S$ forms exactly the subset $\wh\S$ of $\wh\T^1_{R_t}$ consisting of the tails whose image via $\mu\: C(2)\ra C$ crosses $R_t$. 
 If $W$ is in $\T^1_{\gamma_t}$, then $\wh\T^1_{R_t}\setminus\wh\S=\{Y_{t,0},Y_{t,1}\}=\{L_1^W, L_2^W\}$; if $W$ is in $\T^1_{\gamma'_t}$, then $\wh\T^1_{R_t}\setminus\wh\S=\{Y_{t,0}\}=\{L_2^W\}$. 
\end{proof}

\begin{Prop}\label{qs->s}
    If $A$ is quasistable, then $A$ is synchronized. 
\end{Prop}

\begin{proof}
Throughout the proof, fix $s\in\{2,3\}$. When no confusion may arise, we set $W_{\cdot}:=W^s_{\cdot}$ and $Y_{\cdot}:=Y^s_{\cdot}$. Let $\zeta:=\zeta_s$  be as before Lemma \ref{LemB} and $\tau:=\tau_s$ be the permutation of $\{0,1,2\}$ of Proposition \ref{PropA} such that
   \begin{equation}\label{chain}
   W_{\zeta+\tau(0)}\subseteq W_{\zeta_s+\tau(1)}\subseteq W_{\zeta+\tau(2)}.
   \end{equation}
      Let $u$ be the maximal number in $\{0,1,2\}$ (if it exists) for which $W_{\zeta_s+\tau_s(u)}\ne C$, otherwise put $u:=-1$. Recall that we set $m_s:=\#\wh\T^s_A-1$. 

    By Lemma \ref{1sync}, we need only show that $A$ is $s$-synchronized. 
     We prove at the same time the cases $s=2$ and $s=3$; thus, during the proof for $s=3$, we may assume that $A$ is 2-synchonized.

   \smallskip
   
   \emph{\bf Step 1}.  
      We claim that $A$ is $s$-synchronized if the following properties hold
   \begin{itemize}
     \item[(P1)]
       $m_s\ge \zeta+u$ and, if $u\ge0$, then 
$\mu(Y^s_{\zeta_s+t})=W^s_{\zeta_s+\tau_s(t)}$, for every $t\in\{0,\dots u\}$;
   \item[(P2)]
  $\{\mu(Y^3_t)\}_{t>\zeta_3+2}$ and $\{W^3_t : W^3_t\ne C\}_{t>\zeta_3+2}$
 are $\T^2_A$-free.
 \end{itemize}
      
By Proposition \ref{PropB}, we see that Condition (P1) implies that, if $u\le 1$, then $m_s=\zeta+u$. 
   If $\mu$ is the set function sending a tail $Y$ of $C(2)$ to the tail $\mu(Y)$ of $C$, we need only show that we have a bijection
     \[
  \mu\:\{Y_t\}_{t\ge \zeta}\lra \{W_t : W_t\ne C\}_{t\ge \zeta}.
  \]
   Indeed, this implies that $A$ is $s$-synchronized when $\zeta=0$ and, by Lemma \ref{LemA}, also that $A$ is $s$-synchronized when $\zeta\ge3$. 
 For an integer $t\ge0$, write 
   $t=3a_t+b_t$, for integers $a_t\ge0$ and $b_t\in\{0,1,2\}$, and set $v_t:=3a_t+\tau(b_t)$.

We begin by proving that $\mu(Y_t)=W_{v_t}$,  for every $t$ in $\{\zeta,\dots,m_s\}$. We proceed by induction. We may assume that $u\ge0$, otherwise $m_s=\zeta-1$, and we have nothing to prove.
The statement holds for  $t$ in $\{\zeta,\dots,\zeta+u\}$ by Condition (P1). We may assume that $u=2$, otherwise $m_s=\zeta+u$ and we are done. 
   Let $t$ be in $\{\zeta+3,\dots, m_s\}$. By Condition (P2), the tail $\mu(Y_t)$ is $\T^2_A$-free. Since we have 
\[
   \mu(Y_{t-3})=W_{v_t-3}\prec \mu(Y_t)
\]
  (the equality by induction), it follows that $W_{v_t-3}\prec W_{v_t}\subseteq \mu(Y_t)$. We also have
   \begin{equation}\label{Due}
   Y_{t-3}\prec L_0^{W_{v_t}}\prec L_1^{W_{v_t}}\prec L_2^{W_{v_t}}.
   \end{equation}
    Notice that $v_t\ge \zeta+3$, 
   because $\zeta\equiv 0(3)$.  It follows from Condition (P2) that, if $s=3$, the canonical $W_{v_t}$-liftings are $\wh{\T}^2_A$-free, since in this case $A$ is 2-synchronized. 
  By \eqref{Due} we get  $Y_t\subseteq L_2^{W_{v_t}}$, and we deduce that $\mu(Y_t)\subseteq W_{v_t}$ and hence $\mu(Y_t)=W_{v_t}$.  
    
Up to switching the $W_t$'s with the $W_{v_t}$'s, we can assume that $\mu(Y_t)=W_t$, for every $t\in\{\zeta,\dots,m_s\}$.  By contradiction, suppose that $W_{t_0}\ne C$ is not in the image of $\mu$, for some $t_0>m_s$, with $t_0$ the minimum integer satisfying this condition. Notice that $t_0>\zeta+2$, otherwise $u\le 1$ and $W_{t_0}=C$.
    Thus, we have $t_0>2$ and $\mu(Y_{t_0-3})=W_{t_0-3}\prec W_{t_0}$ (the equality by the minimal property of $t_0$), hence 
     \[
     Y_{t_0-3}\prec L_0^{W_{t_0}}\prec L_1^{W_{t_0}}\prec L_2^{W_{t_0}}.
     \]
By Condition (P2),  the canonical liftings of $W_{t_0}$ are $\wh\T^2_A$-free for $s=3$ (we use again that $A$ is 2-synchronized for $s=3$), hence $m_s\geq t_0$, which is a contradiction.

   \smallskip
   
   \emph{\bf Step 2}. 
     We show that Condition (P1) holds.   We may assume $W_{\zeta_s+\tau(0)}\ne C$, otherwise we are done, because by Proposition \ref{PropB} we would have $\#\wh\T^s_A=\zeta$.
     
         Suppose first $s=2$, or $s=3$ with $d_{\T^2_A,R_1}=d_{\T^2_A,R_2}=0$ (hence $\zeta=0)$. 
   Define
   \[
   Y:=L_0^{W_{\tau(0)}}\cup E_{R_{\tau(0)},\gamma_{\tau(0)}}\cup E_{R_{3-\tau(0)}}, 
   \]
   \[
Y':=L_1^{W_{\tau(1)}}\cup E_{R_1} \cup E_{R_2}   \; \text{ and } \;   Y'':=
L_2^{W_{\tau(2)}}. 
\]
By Proposition \ref{PropA}(i), we have $\tau(1)=0$. In the sequel, we often use Proposition \ref{PropB} without mentioning. Since $Y$ is a proper subcurve, we have $\T^s_A\ne\emptyset$, with $Y_0\subseteq Y$,  hence $\mu(Y_0)\subseteq W_{\tau(0)}$ and $\mu(Y_0)=W_{\tau(0)}$. 
We may assume $W_{\tau(1)}\ne C$, otherwise $\wh\T^s_A=\{Y^s_0\}$, and we are done. 
 Since $Y\prec Y'$, we have $\#\wh\T^s_A>1$, with $Y^s_1\subseteq Y$, hence $\mu(Y^s_1)\subseteq W_{\tau(1)}$ and $\mu(Y^s_1)=W_{\tau(1)}$. 
We may assume $W_{\tau(2)}\ne C$, otherwise $\wh\T^s_A=\{Y^s_0,Y^s_1\}$, and we are done. Since $Y'\prec Y''$, we have $\#\wh\T^s_A>2$, with $Y^s_2\subseteq Y''$, hence $\mu(Y^s_2)\subseteq W_{\tau(2)}$ and $\mu(Y^s_2)=W_{\tau(2)}$.

Suppose now $s=3$, $d_{\T^2_A,R_1}\ne0$ and $d_{\T^3_A,R_1}\ne 0$. By Proposition \ref{PropA}(ii), we have $\zeta=\tau(0)=0$. It follows from \eqref{chain} that $R_1$ is a terminal point of $W_0$ and hence, since $A$ is quasistable, we see that $W_0$, and hence $W_{\tau(1)}$, crosses $R_2$. Define 
  \[
  Y:=L_0^{W_0}\cup E_{R_1}, \; Y':=L_1^{W_{\tau(1)}} \; \text{ and } Y'':=L_2^{W_{\tau(2)}}.
  \]
   By \eqref{dtA} and Lemma \ref{LemB}(ii), the tails $Y$, $Y'$ and $Y''$ are $\wh{\T}^2_A$-free, and $Y\prec Y'\prec Y''$. We can conclude as in the previous paragraph using Proposition \ref{PropB}.
   
  Finally, suppose $s=3$ and that the following properties hold
   \[
   d_{\T^2_A,R_1}\ne0,\; d_{\T^3_A,R_1}=0 \; \text{ and that } \; d_{\T^3_A,R_2}\ne 0 \Rightarrow d_{\T^2_A,R_2}=0.
   \]
   The tail $W_{\zeta+\tau(0)}$ is not contained in $\cap_{(\alpha,\beta)\in\Pa_A}\T^3_{\alpha,\beta}$ 
 by the definition of $\zeta$ and by \eqref{chain}. 
  There is exactly one node $S\in C\setminus\{R_1,R_2\}$ such that $d_{\T^2_A,S}\not\equiv0(3)$. Indeed, if $d_{\T^3_A,R_2}\ne0$, then $d_{\T^2_A,R_2}=0$ and $S$ is the difference node of $\T^2_{\gamma_1,\gamma_2|\gamma'_1,\gamma_2}$ distinct from $R_1$. On the other hand, if $d_{\T^3_A,R_2}=0$, then by Remark \ref{Rem2} there is a node $S\in C\setminus\{R_1,R_2\}$ such that $S\in\term_{W_{\zeta+\tau(0)}}$ and $d_{\T^2_A,S}\not\equiv0(3)$ and $S$ is uniquely determined by Lemma \ref{LemB}(i). 
 In any case, by Lemma \ref{LemD} we have $S\in W_{\zeta+\tau(0)}$. Let $m\in\{1,\dots,p\}^2$ be such that $S\in C_m$ and $C_m\subseteq W_{\zeta+\tau(0)}$, and define  
    \[
    Y:=
  \begin{cases}
  \begin{array}{ll}
     L_0^{W_{\zeta+\tau(0)}}\cup E_{R_2,\gamma_2}\cup E_{S,m} & \text{if } \tau(0)\ne1\text{ and }d_{\T^2_A,S}=1;\\
  L_0^{W_{\zeta+\tau(0)}}\cup E_{R_2,\gamma_2}\cup  E_S   & \text{if } \tau(0)\ne1\text{ and } d_{\T^2_A,S}=2;\\
  L_0^{W_{\zeta+\tau(0)}}\cup E_{R_2}\cup E_{S,m}  
& \text{if } \tau(0)=1 \text{ and }d_{\T^2_A,S}=1;\\
L_0^{W_{\zeta+\tau(0)}}\cup E_{R_2}\cup E_S & \text{if } \tau(0)=1 \text{ and }d_{\T^2_A,S}=2;
   \end{array}
  \end{cases}
  \]
  \[
  Y':=L_1^{W_{\zeta+\tau(1)}}\cup E_{R_2}\cup E_S \quad \text{ and } \quad Y'':=L_2^{W_{\zeta+\tau(2)}}.
\]
   By \eqref{dtA}, the tails $Y,$ $Y'$ and $Y''$ are $\wh\T^2_A$-free. Since $Y\prec Y'\prec Y''$, we can conclude as before using Lemma \ref{PropB}.
     
  \smallskip
  
 \emph{\bf Step 3}.   We show that Condition (P2) holds. We are in the case $s=3$, hence we may assume that $A$ is 2-synchronized. 
   Let $S$ be a node of $C$ such that $d_{\T^2_A,S}\ne0$. We need to prove that $S$ is not a terminal point of a tail in the sets 
   \[
   \{\mu(Y^3_t)\}_{t>\zeta+2} \; \text{ and } \{W_t : W_t\ne C\}_{t>\zeta+2}.
   \]
    If $d_{\T^2_A,S}=3$, this is obvious for the second set, and by \eqref{ne3} also for the first one. Thus, we may assume $d_{\T^2_A,S}\not\equiv 0(3)$ and, by Lemma \ref{Diff}(i), also that $d_{\T^2_A,R_1}\ne0$.
     Set $W:=W_{\zeta+\tau(0)}$. We can assume $W\ne C$ (i.e. $u\ge0$), otherwise the above sets are empty by Proposition \ref{PropB}, and we have nothing to prove.  
    Since   
\[
       \mu(Y_{\zeta+t})=W_{\zeta+\tau(t)}\subseteq W_{\zeta+\tau(t+1)} \text{ for every } t\in\{0,\dots u\},
\]
 we see that $W\prec \mu(Y_{t_1})\wedge W_{t_2}$, for every $t_1>\zeta+2$ and $t_2>\zeta+2$. Thus, it suffices that $S\in W$. Of course, we may assume that $S\in C\setminus\{R_1,R_2\}$; 
 it follows from Lemma \ref{LemB}(i) that $S$ is uniquely determined by the condition $d_{\T^2_A,S}\not\equiv0(3)$.
    
If $d_{\T^3_A,R_1}=d_{\T^3_A,R_2}=0$, then, since $W$ is not in the intersection $\cap_{(\alpha,\beta)\in\P_A} \T^3_{\alpha,\beta}$, 
we see that $W$ is the minimal tail of $\T^3_{\alpha,\beta | \alpha',\beta'}$, for some $(\alpha,\beta)$ and $(\alpha',\beta')$ in $\P_A$, hence by Remark \ref{Rem2} we have $S\in \term_{W}\subseteq W$. If $d_{\T^3_A,R_1}\ne 0$, then $S$ is in $\{R_1,R_2\}$ by Lemma \ref{LemB}(ii), hence $S\in W$. If $d_{\T^2_A,R_2}=d_{\T^3_A,R_1}=0$ and $d_{\T^3_A,R_2}\neq0$, then $R_2\in\term_W$, hence by Lemma \ref{LemD} we have $S\in W$.
   \end{proof}

\subsection{From the local to the global resolution}
\noindent
Let $\pi\col \C\ra B$ be a smoothing of a nodal curve $C$. 
We are ready to state and prove the main results of the paper. First, in Theorem \ref{main4}, we will give a combinatorial  criterion for the existence of a local resolution for the degree-2 Abel--Jacobi map $\alpha^2$ of $\C/B$. Second, in Theorem \ref{main5}, we will show that the blowup $\C^2$ along products of 2-tails and 3-tails of $C$ fulfills the local criterion and hence gives rise to a global algebraic resolution of $\alpha^2$.

\begin{Thm}\label{main4}
Let $\pi\:\C\ra B$ be  a smoothing  of a nodal curve $C$ with a section $\sigma$ through its smooth locus. Let $\phi\:\wt\C^2\ra \C^2$  and $\psi\:\wt\C^3\ra\wt \C^2\times_B\C$ be  good partial desingularizations. Let $R_1,R_2$ be in $\N(C)$, with $R_1\ne R_2$. Let $A_1$ and $A_2$ be the distinguished point of $\wt\C^2$ in $\phi^{-1}(R_1,R_2)$. 
The following properties are equivalent 
\begin{itemize}
\item[(i)]
 $A$ is quasistable, for every $A$ in $\{A_1,A_2\}$;
\item[(ii)]
$A$ is synchronized, for every $A$ in $\{A_1,A_2\}$;
\item[(iii)]
  $\psi|_{X_A*}(\L_\psi|_{X_A})$ is $\sigma(0)$-quasistable, for every $A$ in $\{A_2,A_2\}$;
 \item[(iv)]
 $\psi|_{X_A*}(\L_\psi|_{X_A})$ is simple, for every $A$ in $\{A_1,A_2\}$;
\item[(v)]
the map $\alpha^2\circ\phi\col\wt\C^2\dra \J$ induced by $\psi_*\L_\psi$ is defined on an open subset of $\wt\C^2$ containing $\{A_1, A_2\}$.
     \end{itemize}
\end{Thm}

    \begin{proof} 
By Proposition \ref{qs->s}, item (i) implies item (ii). Using \cite[Proposition 1]{E01}, item (iii) implies item (iv). By Theorem \ref{main2}, $\L_\psi|_{X_A}$ is $\psi|_{X_A}$-admissible and its formation commutes with base change: By Lemma \ref{ext-red} and by the fact that $\J$ is a fine moduli scheme (see Section \ref{Jacobian}), items (iii) and (v) are equivalent.
    
   From now on, we will denote by $A$ one of the two distinguished points in $\{A_1,A_2\}$. 
 Recall the identification of $X_A$ with $C(2)$ and of $\psi(X_A)$ with $C$, and the identification of the contraction map $\mu\: C(2)\ra C$ with $\psi|_{X_A}\: X_A\ra \psi(X_A)$. As usual, we let $(\gamma_1,\gamma_2)$ and $(\gamma'_1,\gamma'_2)$ in $\{1,\dots,p\}^2$ be such that
  \[
  A\in C_{\gamma_1,\gamma_2,\phi}\cap C_{\gamma_1,\gamma'_2,\phi}\cap C_{\gamma'_1,\gamma_2,\phi}.
  \]
Also, we fix a slice $\lambda\:B\to\wt\C^2$ of $\wt\C^2$ through $A$. We let $\W\ra B$ be the $\lambda$-smoothing of $X_A$ and $\theta\col \W\ra \wt\C^3$ be the induced map 
 (see Diagram \eqref{diagram}). 
 By \eqref{diag-type}, we have
 \begin{equation}\label{First-Rel}
  \theta^*\left(\I_{\wt\Delta_1|\wt\C^3}\otimes \I_{\wt\Delta_2|\wt\C^3}\right)\simeq  \O_{\W}(D_1-\Gamma_1-\Gamma_2),
  \end{equation}
  where $\Gamma_1$ and $\Gamma_2$ are relative Cartier divisors on $\W/B$ intersecting transversally $X_A$ respectively in $E_{R_1,\gamma_1}$ and $E_{R_2,\gamma_2}$, and $D_1$ is 
  a Cartier divisor of $\W$ supported on $\psi|_{X_A}$-exceptional components. We set 
  \[
 \M:=\O_{\wt\C^3}\left(-\sum_{1\le i,k\le p}\sum_{m\in\{1,\dots,p\}} \alpha_{i,k,m} C_{i,k,m,\psi}\right).
  \]
Fix smooth points $Q_{\gamma_1}$ and $Q_{\gamma_2}$ of $X_A$ lying respectively on $E_{R_1,\gamma_1}$ and $E_{R_2,\gamma_2}$.

  We prove that item (ii) implies item (iii). Suppose that $A$ is synchronized. Then the set function sending a tail $Y$ of $C(2)$ to the tail $\mu(Y)$ of $C$ induces a bijection
  \[
  \wh\T_A\lra \T_A.
  \]
  By the definition of $\wh\T_A$ and by Theorem \ref{main1} this means that there is a Cartier divisor $D_2$ of $\W$ supported on $\psi|_{X_A}$-exceptional components such that 
\[
\theta^*\xi^*\P|_{X_A}(-Q_{\gamma_1}-Q_{\gamma_2})\otimes\theta^*\M(D_2)|_{X_A}
\]
  is a $\psi^{-1}(\sigma(0))$-quasistable invertible sheaf on $X_A$. Since \eqref{First-Rel} implies the relation
  \[
  \theta^*\L_\psi|_{X_A}\simeq \O_\W(D_1-D_2)|_{X_A}\otimes\theta^*\xi^*\P|_{X_A}(-Q_{\gamma_1}-Q_{\gamma_2})\otimes\theta^*\M(D_2)|_{X_A},
\]
 combining Lemmas  \ref{main2} and \ref{compadm}, we see that $\psi|_{X_A*}(\L_\psi|_{X_A})$ is $\sigma(0)$-quasistable. 
  
  We prove that item (iv) implies item (i). Suppose $A$ is not quasistable. Then $A_1$ and $A_2$ both are not quasistable. Up to switching $A$ with the other point in $\{A_1,A_2\}$, we may assume that there are tails in $\T^2_{\gamma_1,\gamma_2}$ containing $R_1$ and $R_2$ as terminal points. By Lemma \ref{adia-modif}, there are subcurves $Y_1$ and $Y_2$ of $C$ such that 
  \[
  \O_\C(Z_{\gamma_1,\gamma_2})|_C\simeq\O_\C(Z_{\gamma'_1,\gamma_2}-Y_1)|_C 
  \; \text{ and } \; 
  \O_\C(Z_{\gamma_1,\gamma_2})|_C\simeq\O_\C(Z_{\gamma_1,\gamma'_2}-Y_2)|_C
  \]
  and such that
  \[
  C_{\gamma_1}\subseteq Y_1, \; C_{\gamma'_1}\subseteq Y_1^c,\;   
   C_{\gamma_2}\subseteq Y_2 \; \text{ and } \; C_{\gamma'_2}\subseteq Y_2^c.
  \]
  Fix $t\in\{1,2\}$. The subcurve $Y_t$ is obtained by the rule
  \[
  Y_t=\sum_{Z\in\T_{\gamma_1,\gamma_2}} Z - \sum_{Z\in\T_{\gamma_{3-t},\gamma'_t}}Z, 
  \]
  and $\T^s_{\gamma_1,\gamma_2|\gamma_{3-t}\gamma'_t}$ is a totally ordered set with respect to inclusion for each $s$ in $\{1,2,3\}$ (see Proposition \ref{Diff}). Therefore, 
 if $R_{3-t}$ is a terminal point of $Y_t$, we have 
  \[
  C_{\gamma_{3-t}}\subseteq Y_t \; \text{ and } \; C_{\gamma'_{3-t}}\subseteq Y_t^c.
  \] 
For each $t\in\{1,2\}$, we define
  \[
 Y'_t:=
 \begin{cases}
 \begin{array}{ll}
 Y^c_t  & \text{if } Y^c_t \text{ crosses } R_{3-t}; \\
  Y_t &  \text{otherwise}.
  \end{array}
 \end{cases}
 \]
We define the (nonempty) subcurve $Y$ of $C$ 
\[
Y:=Y'_1\cup Y'_2.
\]
 
 We claim that for every maximal chain $E$ of $\psi|_{X_A}$-exceptional components lying over a node of $Y\wedge Y^c$, there is a component of $E$ over which the invertible sheaf $\L_\psi$ has nonzero degree. Indeed, if we put 
  \[
  D:=-\sum_{m=1}^p\alpha_{\gamma_1,\gamma_2,m}(C_{\gamma_1,\gamma_2,m,\psi}+C_{\gamma_1,\gamma'_2,m,\psi}+C_{\gamma'_1,\gamma_2,m,\psi}),
  \]
  then $\theta^*\L_\psi$ is obtained by the pull-back via $\theta$ of the invertible sheaf 
   \[
\xi^*\P\otimes\I_{\wt\Delta_1|\wt\C^3}\otimes \I_{\wt\Delta_2|\wt\C^3} \otimes\O_{\wt\C^3}\left(D+\sum_{C_m\subseteq Y_2}C_{\gamma_1,\gamma'_2,m,\psi}+\sum_{C_m\subseteq Y_1}C_{\gamma'_1,\gamma_2,m,\psi}\right).
 \]
 By Lemma \ref{mu*} the degree of $\O_{\wt\C^3}(D)$ on every $\psi|_{X_A}$-exceptional component is 0. Let $S$ be a node in $Y\wedge Y^c\setminus\{R_1,R_2\}$. Since $\I_{\wt\Delta_1|\wt\C^3}\otimes \I_{\wt\Delta_2|\wt\C^3}$ has degree 0 on each component of $E_S$, we see that $\L_\psi$ has nonzero degree on at least one component of $E_S$. Finally, consider $R_t$, for some $t\in\{1,2\}$. It is easy to check that if $R_t\in\term_Y$, then $R_t\in\term_{Y_1}\cap \term_{Y_2}$ and $C_{\gamma_t}\subseteq Y$.
  Thus,  using \eqref{First-Rel}, we see that there is a Cartier divisor $D_3$ of $\W$ supported on the $\psi|_{X_A}$-exceptional components $E_{R_t,\gamma_t}$ and $E_{R_t,\gamma'_t}$ such that
 \[
\deg_{E_{R_t,\gamma_t}} \theta^*\L_\psi(D_3)=1 \; \text{ and } \deg_{E_{R_t,\gamma'_t}} \theta^*\L_\psi(D_3)=0.
 \]
By Theorem \ref{main2}, $\L_\psi$ is $\psi$-admissible, hence one of the following properties holds
 \[
 \deg_{E_{R_t,\gamma_t}}\L_\psi=\deg_{E_{R_t,\gamma'_t}}\L_\psi+1=0
 \; \text{ or }  \;
 \deg_{E_{R_t,\gamma_t}}\L_\psi+1=\deg_{E_{R_t,\gamma'_t}}\L_\psi-1=0.
 \]
 The proof of the claim is complete.
 
 The claim implies that $\psi|_{X_A*}(\L_\psi|_{X_A})$ is non invertible at $Y\wedge Y^c$. Since $Y\wedge Y^c$ forms a disconnecting sets of nodes of $C$, the sheaf $\psi|_{X_A*}(\L_\psi|_{X_A})$ is not simple. 
    \end{proof}

Let $\S$ be an ordered set of Weil divisors of $\C^2$ given by
\[
\S=\{W_{1,1}\times W_{1,2},\dots,W_{N,1}\times W_{N,2}\},
\]
where $W_{i,j}$ is a subcurve of $C$. We consider the blowup sequence
\[
\phi_\S\col \wh{\C}^2_\S:=\wh\C^2_N\stackrel{\phi_N}{\lra}\cdots \stackrel{\phi_2}{\lra}\wh\C^2_1\stackrel{\phi_1}{\lra}\wh\C^2_0\stackrel{\phi_0}{\lra}\C^2,
\]
where $\phi_0$ is the blowup of $\C^2$ along the diagonal subscheme of $\C^2$, and where $\phi_i$ is the blowup of $\wh\C^2_{i-1}$ along the strict transform of $W_{i,1}\times W_{i,2}$ via the composed map $\phi_0\circ\cdots\circ\phi_{i-1}$.  

We say that $\phi_\S$ is \emph{symmetric} if for every $i$ in $\{1,\dots,N\}$ there is $j$ in $\{1,\dots,N\}$ (possibly with $i=j$) such that
\[
W_{i,1}=W_{j,2} \; \text{ and } \; W_{i,2}=W_{j,1}.
\]

We say that \emph{$\phi_\S$ resolves the degree-2 Abel--Jacobi map of $\C/B$} if the composed rational map 
\[
\alpha^2\circ\phi_\S\col\wh\C^2_\S\dra \J
\]
is a morphism, where $\alpha^2\col\C^2\dashrightarrow \J$ is the degree-2 Abel--Jacobi map of $\C/B$.

\begin{Thm}\label{main5}
 Let $\pi\:\C\ra B$  be a smoothing of a nodal curve $C$ with a section $\sigma$ through its smooth locus. Let $C_1,\dots,C_p$ be the irreducible components of $C$. If $\S$ is an ordered set of Weil divisors of $\C^2$ given by products of type $W\times W$, for tails $W$ of $C$ with $k_W\in\{1,2,3\}$, then $\phi_\S$ is independent of the ordering of the Weil divisors appearing in $\S$. Consider the set of Weil divisors of $\C^2$
\[
  \T=\{W\times W : W \text{ is a tail of $C$, with } k_W\in\{2,3\}\}.
  \]
 Then $\phi_\T$ is symmetric and resolves the degree-2 Abel--Jacobi map of $\C/B$. 
\end{Thm}

    \begin{proof}
     For every tail $W$ of $C$ with $k_W\in\{1,2,3\}$, the blowup of $\C^2$ at $W\times W$ is an isomorphism away from the locus of the points $(R_1,R_2)$, with $\{R_1,R_2\}\subseteq \term_W$. 
    Recall that the first blowup appearing in the sequence $\phi_\S$ is the blowup along the diagonal of $\C^2$. Thus, if $\S'$ is obtained by a permutation of the tails of $\S$ and if $\phi_{\S}$ is distinct from $\phi_{\S'}$,  then there are tails $W_1$ and $W_2$ of $C$ such that $k_{W_i}\in\{2,3\}$ and such that
   \[
     \{R_1,R_2\}\subseteq \term_{W_1}\cap \term_{W_2} \text{ with } R_1\ne R_2,
    \] 
   and where  the blowups $\eta_1$ and $\eta_2$ of $\C^2$ respectively along $W_1\times W_1$ and $W_2\times W_2$ are distinct locally around $(R_1,R_2)$. In particular, we see that
   $(W_1, W_2)$ is perfect. Let $(\gamma_1,\gamma'_1)$ and $(\gamma_2,\gamma'_2)$ in $\{1,\dots,p\}^2$ be such that
      \begin{equation}\label{usual}
   R_1\in C_{\gamma_1}\cap C_{\gamma'_1} \;\text{ and } \; R_2\in C_{\gamma_2}\cap C_{\gamma'_2}.
   \end{equation}
   Then, locally around the point $(R_1,R_2)$, we may assume that $\eta_1$ is the blowup of $\C^2$ along $C_{\gamma_1}\times C_{\gamma_2}$ and $\eta_2$ the one along $C_{\gamma_1}\times C_{\gamma'_2}$.  We see that \[
   C_{\gamma_1}\subseteq W_1\wedge W_2,\; C_{\gamma_2}\subseteq W_1\wedge W_2^c\; \text{ and } C_{\gamma'_2}\subseteq W_1^c\wedge W_2,
   \]
    and hence $(W_1,W_2)$ can not be perfect, yielding a contradiction.
    
    The fact that $\phi_\T$ is symmetric is clear. Let us show that $\phi_\T$ resolves the degree-2 Abel map of $\C/B$. 
    Let 
 $\phi\col\wt\C^2\ra\C^2$ be \emph{any} good partial desingularization such that 
$\phi=\phi_\T\circ\eta$, for some $\eta\col\wt\C^2\ra \wh\C^2_\T$.  

We claim that if $A\in\phi^{-1}(R_1,R_2)$ is a distinguished point of $\wt\C^2$, with  $R_1\ne R_2$ and $\{R_1,R_2\}\subseteq \N(C)$, then $A$ is quasistable. Indeed, let 
 $(\gamma_1,\gamma'_1)$ and $(\gamma_2,\gamma'_2)$ in $\{1,\dots,p\}^2$ be such that \eqref{usual} holds. Suppose that   
    \begin{equation}\label{12}
    R_1\in \term_W\;  \text{ and } \; R_2\in \term_{W'}, \text{ for }\{W,W'\}\subseteq \T^2_{\gamma_1,\gamma_2}\cup \T^3_{\gamma_1,\gamma_2}.
     \end{equation} 
     Notice that, if $W\ne W'$, then $W$ and $W'$ do not contain each other and hence, up to switching $W$ and $W'$, we have $W\in \T^2_{\gamma_1,\gamma_2}$ and $W'\in\T^3_{\gamma_1,\gamma_2}$ and $(W,W')$ is free. Using \cite[Lemma 3.3(iii)]{P1} in the case $W\ne W'$, we get in any case
          \[
     (W\wedge W') \times (W\wedge W') \in \T, 
     \]
      with
      \[
C_{\gamma_1}\cup C_{\gamma_2}\subseteq W\wedge W'        \text{ and } C_{\gamma'_1}\cup C_{\gamma'_2}\subseteq (W\wedge W')^c.
     \] 
    This implies that $\phi_\T$ is, locally around $(R_1,R_2)$, isomorphic to the blowup of $\C^2$ along $C_{\gamma_1}\times C_{\gamma_2}$, and one of the following conditions holds
\begin{equation}\label{34}
A\in  C_{\gamma_1,\gamma'_2,\phi}\cap C_{\gamma_1,\gamma_2,\phi}\cap C_{\gamma'_1,\gamma'_2,\phi}\; \text{ or } 
A\in  C_{\gamma'_1,\gamma_2,\phi}\cap C_{\gamma'_1,\gamma'_2,\phi}\cap C_{\gamma_1,\gamma_2,\phi},
\end{equation}
     and hence $A$ is quasistable. 

Let $\psi\col\wt\C^3\ra \wt\C^2\times_B\C$ be \emph{any} good partial desingularization. By the claim, Lemma \ref{ext-red} and Theorems \ref{main2}, \ref{main3} and \ref{main4}, we see that the rational map $\alpha^2\circ\phi\col\wt\C^2\dashrightarrow \J$ induced by the relative sheaf $\psi_*\L_\psi$ on $\wt\C^2\times_B\C/\wt\C^2$ is a morphism. 

In particular, the rational map 
\[
\alpha^2\circ\phi_\T\col\wh\C^2_\T\dashrightarrow \J
\]
is defined away from the locus of points of $\wh\C^2_\T$ lying isomorphically over points $(R_1,R_2)$ of $\C^2$, with $\{R_1,R_2\}\subseteq \N(C)$ and $R_1\ne R_2$. Abusing notation, let $(R_1,R_2)$ be such a point of $\wh\C^2_\T$ and assume that \eqref{usual} holds. Let 
\[
\phi^{(R_1,R_2)}_1\col \X\stackrel{\eta_1}{\ra}\wh\C^2_\T\stackrel{\phi_\T}{\ra}  \C^2 \; \text{ and } \; \phi^{(R_1,R_2)}_2\col\Y\stackrel{\eta_2}{\ra}\wh\C^2_\T\stackrel{\phi_\T}{\ra}  \C^2
\]
be good partial desingularizations, where, locally around $(R_1,R_2)$, the maps $\eta_1$ and $\eta_2$ are the blowups of $\wh\C^2_\T$ respectively along the strict transforms of $C_{\gamma_1}\times C_{\gamma_2}$ and $C_{\gamma_1}\times C_{\gamma'_2}$ via $\phi_\T$. By what we have already shown, the maps 
\[
\alpha^2\circ\phi^{(R_1,R_2)}_1\col\X\dra \J\; \text{ and } 
\alpha^2\circ\phi^{(R_1,R_2)}_1\col\Y\dra \J
\]
are morphisms. Thus, arguing as in the proof of \cite[Theorem 6.1]{CEP},  both maps factor through a morphism $U\ra \J$ defined on a Zariski neighborhood $U$ of  $\wh\C^2_\T$ containing $(R_1,R_2)$ and agreeing with $\alpha^2\circ\phi_\T$ away from $(R_1,R_2)$, hence $\alpha^2\circ\phi_\T$ is a morphism.
\end{proof}

    \begin{Rem}
Assume that $\phi_\S$ resolves the degree-2 Abel--Jacobi map of $\C/B$. We say that $\phi_\S$  is \emph{minimal} if for any  ordered set $\S'$ of Weil divisors of $\C^2$ for which $\phi_{\S'}$ resolves the degree-2 Abel--Jacobi map, there is a morphism $\eta\col\wh\C^2_{\S'}\ra\wh\C^2_\S$ such that $\phi_{\S'}=\phi_\S\circ\eta$.
    
  There exist curves for which the blowup $\phi_\T$ of Theorem \ref{main5} is not minimal. For example, let $C$ be the union of $C_1,C_2,C_3$, with $\#C_1\cap C_i=1$, for $i\in\{2,3\}$ and $\#C_2\cap C_3=2$. Assume that $\sigma(0)\in C_1$. We have 
  \[
    \T_{1,1}=\T_{1,2}=\T_{1,3}=\emptyset, \; \T_{2,2}=\T_{2,3}=\T_{3,3}=\{C_2\cup C_3\}. 
    \]
Let $\S=\{(C_2\cup C_3)\times (C_2\cup C_3)\}$. Using Theorem \ref{main4}, it is easy to see 
 $\phi_\S$ resolves the degree-2 Abel--Jacobi map of $\C/B$; moreover, $\phi_\S$ is minimal and $\phi_\S\ne \phi_\T$.   
  
   Notice that $(C_2, C_2\cup C_2)$ is terminal. In fact, it is not difficult to show that if $\phi_\T$ is not minimal, then there is a terminal pair $(W_1,W_2)$ of tails of $C$, with $k_{W_1}=2$ and $k_{W_2}=3$, and with $\sigma(0)\in W_1\subseteq W_2$.
    \end{Rem}
    
   \begin{Rem}
   Since $\T$ is symmetric, we get a commutative diagram
\[
\SelectTips{cm}{11}
\begin{xy} <16pt,0pt>:
\xymatrix{
\wh\C^2_{\T}\ar[d]^{\phi_\T}\ar[r]^{\chi\;\;\;\;\;}&  \wh\C^2_\T/S_2 \ar[d]\UseTips\ar[dr]^{\ol\beta^2_{\T}}&\\
 \C^2\ar[r] & \C^2/S_2 \ar@{..>}[r]^{\;\;\alpha^2} &\J
}
\end{xy}
\]
 By the Abel Theorem for smooth curves, the generic fiber $C_\eta$ of $\pi\col\C\ra B$ is hyperelliptic if and only if the restriction of  $\ol\beta_\T$ over the generic point of $B$ is not constant; in this case, its fiber is isomorphic to the $g^1_2$ of $C_\eta$. A natural question is whether or not a similar property holds for the special fiber of $\pi$. 
   \end{Rem}

\section*{Acknowledgments}
\noindent
This paper is a continuation of a project started in collaboration with J. Coelho and E. Esteves. I wish to thank them for several stimulating discussions. The author was partially supported by CNPq, processo 300714/2010-6.

\bigskip
\bigskip

Marco Pacini, Universidade Federal Fluminense (Uff)

Rua M. S. Braga, s/n, Valonguinho 24020-005

Niter\'oi (RJ), Brazil
 
 email: \emph{pacini@impa.br} and \emph{pacini@vm.uff.br}
    

\begin{thebibliography}{llll}

\bibitem{ACP} A. Abreu, J. Coelho, M. Pacini, 
\emph{On the geometry of Abel maps for nodal curves}. 
Preprint available at http://arxiv.org/abs/1303.6491.

\bibitem{AK} A. Altman, S. Kleiman, 
\emph{Compactifying the Picard scheme}, 
Adv. Math. {\bf 35} (1980) 50--112. 


\bibitem{BLR} S. Bosch, W. L\"uktebohmert, M. Raynaud, \emph{N\'eron models}.  
Ergeb. der Mat. 21, Springer, 1980. 

\bibitem{B} S. Busonero, 
\emph{CompactiÞed Picard schemes and Abel maps for singular curves}, 
Ph.D. thesis, Roma La Sapienza, 2008. 

\bibitem{C} L. Caporaso, 
\emph{A compactification of the universal Picard variety over the moduli space of stable curves}. 
J. Amer. Math. Soc. {\bf 7} (1994) 589--660.


\bibitem{CAJM} L. Caporaso, 
\emph{N\'eron models and compactified Picard schemes over the moduli stack of stable curves}. Amer. J. Math. {\bf 130} (2008), 1--47.

\bibitem{CE} L. Caporaso, E. Esteves, 
\emph{On Abel maps of stable curves}. 
Mich. Math. J. {\bfseries 55} (2007) 575--607.

\bibitem{CCE} L. Caporaso, J. Coelho, E. Esteves, 
\emph{Abel maps of Gorenstein curves}. 
Rendiconti del Circolo Matematico di Palermo {\bfseries  57} (2008), 
33--59.


\bibitem{CEP} J. Coelho, E. Esteves, M. Pacini, \emph{Degree-2 Abel maps for nodal curves}. 
Preprint available at http://arxiv.org/abs/1212.1123.

\bibitem{Co} J. Coelho, \emph{Abel maps for reducible curves}. Doctor Thesis, IMPA, Rio de Janeiro, 2006.

\bibitem{CP} J. Coelho, M. Pacini,
\emph{Abel maps for curves of compact type}. 
J. Pure Appl. Algebra  214  (2010),  no. {\bf 8}, 1319--1333.


\bibitem{CG} M. Coppens, L. Gatto, 
\emph{Limit Weierstrass schemes on stable curves with $2$ irreducible components}. 
Atti Accad. Naz. Lincei {\bfseries 9} (2001), 205--228.


\bibitem{EH} D. Eisenbud, J. Harris, 
\emph{Limit linear series: Basic theory}. 
Invent. Math. {\bf 85} (1986) 337--371.


\bibitem{E01} E. Esteves, 
\emph{Compactifying the relative Jacobian over 
families of reduced curves}. 
Trans. Amer. Math. Soc. {\bfseries 353} (2001), 3045--3095.


\bibitem{EM} E. Esteves, N. Medeiros, \emph{Limit canonical systems on 
curves with two components}. Invent. Math. {\bf 149} (2002) 267--338.

\bibitem{EO} E. Esteves, B. Osserman, 
\emph{Abel maps and limit linear series}. 
Rendic. Circ. Mat. Pal. {\bfseries 62} issue 1 (2013), 79--95.


\bibitem{Gi} D. Gieseker, 
\emph{Stable curves and special divisors: Petri's conjecture}. 
Invent. Math. {\bf 66} (1982), 251--275.

\bibitem{GH} P. Griffiths, J. Harris,
\emph{On the variety of special linear systems on a general algebraic curve}.
Duke Math. J. {\bf 47} (1980), 233--272. 


\bibitem{K} J. Kass, \emph{Degenerating the Jacobian: the N\'eron Model versus Stable Sheaves}. To appear in Algebra and Number Theory, Preprint available at  http://arxiv.org/abs/1012.2576.

\bibitem{MeVi} M. Melo, F. Viviani, \emph{Fine compactified Jacobians}. To appear in Math. Nachr.

\bibitem{O} B. Osserman, \emph{A limit linear series moduli scheme}. 
Annales de l'Institut Fourier {\bf 56} (2006) 1165--1205.

\bibitem{P1} M. Pacini, \emph{The degree-2 Abel--Jacobi map for a nodal curve -- I}. Preprint  available at http://arxiv.org/abs/1304.5093.


\end{thebibliography}
\end{document}